\documentclass[10pt]{amsart}
\usepackage{amsmath,amsthm,amscd,amssymb,amsfonts}
\usepackage{xcolor,mathtools}
\usepackage{mathrsfs,upgreek}
\usepackage[hidelinks]{hyperref}
\hypersetup{
  colorlinks,
  linkcolor={red!85!black},
  citecolor={blue!85!black},
  urlcolor={blue!95!black}
}

\usepackage[sans]{dsfont} 
\usepackage[T1]{fontenc}
\DeclareMathAlphabet{\mathpzc}{OT1}{pzc}{m}{it} 

\numberwithin{equation}{section} 
\numberwithin{figure}{section} 

\theoremstyle{plain}

\newtheorem*{maintheo}{Main Theorem}

\newtheorem{prop}{Proposition}[section]

\newtheorem{lemm}[prop]{Lemma}
\newtheorem{sublemma}[prop]{Sublemma}

\newtheorem{proproman}{Proposition}

\theoremstyle{definition}
\newtheorem{defi}[prop]{Definition}

\theoremstyle{remark}
\newtheorem{rema}[prop]{Remark}

\newtheoremstyle{citing}
{3pt}
{3pt}
{\itshape}
{}
{\bfseries}
{.}
{.5em}
{\thmnote{#3}}

\theoremstyle{citing}
\newtheorem*{generic}{}

\newcommand{\C}{\mathbb{C}}
\newcommand{\D}{\mathbb{D}}

\newcommand{\N}{\mathbb{N}}

\newcommand{\R}{\mathbb{R}}

\newcommand{\Z}{\mathbb{Z}}

\newcommand{\cE}{\mathcal{E}}

\newcommand{\cI}{\mathcal{I}}

\newcommand{\cK}{\mathcal{K}}

\newcommand{\cM}{\mathcal{M}}

\newcommand{\cO}{\mathcal{O}}
\newcommand{\cP}{\mathcal{P}}

\newcommand{\cR}{\mathcal{R}}

\newcommand{\cW}{\mathcal{W}}
\newcommand{\cX}{\mathcal{X}}

\newcommand{\fD}{\mathfrak{D}}

\newcommand{\sF}{\mathscr{F}}

\newcommand{\sH}{\mathscr{H}}

\newcommand{\sM}{\mathscr{M}}

\newcommand{\sP}{\mathscr{P}}

\newcommand{\sS}{\mathscr{S}}

\newcommand{\sU}{\mathscr{U}}

\newcommand{\hA}{\widehat{A}}
\newcommand{\hB}{\widehat{B}}
\newcommand{\hC}{\widehat{C}}
\newcommand{\hD}{\widehat{D}}

\newcommand{\hJ}{\widehat{J}}

\newcommand{\hU}{\widehat{U}}
\newcommand{\hV}{\widehat{V}}
\newcommand{\hW}{\widehat{W}}
\newcommand{\hX}{\widehat{X}}

\newcommand{\hkappa}{\widehat{\kappa}}

\newcommand{\hXi}{\widehat{\Xi}}

\newcommand{\hrho}{\widehat{\rho}}

\newcommand{\hSigma}{\widehat{\Sigma}}

\newcommand{\hvarphi}{\widehat{\varphi}}

\newcommand{\tA}{\widetilde{A}}

\newcommand{\tC}{\widetilde{C}}

\newcommand{\tI}{\widetilde{I}}
\newcommand{\tJ}{\widetilde{J}}

\newcommand{\tL}{\widetilde{L}}

\newcommand{\tR}{\widetilde{R}}

\newcommand{\tW}{\widetilde{W}}

\newcommand{\tY}{\widetilde{Y}}

\newcommand{\talpha}{\widetilde{\alpha}}
\newcommand{\tbeta}{\widetilde{\beta}}

\newcommand{\tmu}{\widetilde{\mu}}

\newcommand{\tPi}{\widetilde{\Pi}}
\newcommand{\trho}{\widetilde{\rho}}

\newcommand{\partn}[1]{{\smallskip \noindent \textbf{#1.}}}
\renewcommand{\=}{\coloneqq}
\newcommand{\dd}{\hspace{1pt}\operatorname{d}\hspace{-1pt}}

\DeclareMathOperator{\diam}{diam}
\DeclareMathOperator{\dist}{dist}

\DeclareMathOperator{\modulus}{mod} 
\DeclareMathOperator{\supp}{supp} 

\DeclareMathOperator{\crit}{crit}

\newcommand{\pV}{V_f}
\newcommand{\pP}{P_{f,n+1}(0)}
\newcommand{\pL}{L_f}
\newcommand{\pF}{F_f}
\newcommand{\pD}{D_f}

\newcommand{\pchicrit}{\chi_{\crit}(f)}

\newcommand{\psP}{\mathscr{P}_f}
\newcommand{\pfD}{\mathfrak{D}_f}

\newcommand{\df}{\widehat{f}_\lambda}

\newcommand{\signs}{\{ +, - \}^{\N}}
\newcommand{\uvarsigma}{\underline{\varsigma}}
\newcommand{\uiota}{\underline{\iota}}
\newcommand{\fs}{f_{\uvarsigma}}

\newcommand{\hcX}{\widehat{\cX}}

\newcommand{\wtp}{p^+}
\newcommand{\whc}{\widehat{c}}
\newcommand{\whf}{\widehat{f}}
\newcommand{\whm}{\widehat{m}}
\newcommand{\whn}{\widehat{n}}
\newcommand{\whp}{p^-}
\newcommand{\whr}{\widehat{r}}
\newcommand{\whx}{\widehat{x}}

\newcommand{\uW}{\underline{W}}

\newcommand{\chicritf}{\chi_{\crit}(f)}

\newcommand{\chicritfs}{\chi_{\crit}(\fs)}

\newcommand{\chiinfR}{\chi_{\operatorname{inf}}^{\R}}
\newcommand{\chiinfC}{\chi_{\operatorname{inf}}}
\newcommand{\cl}[1]{{\rm cl}({#1})}

\newcommand{\Asup}{A_{\sup}}
\newcommand{\Ainf}{A_{\inf}}

\usepackage{microtype}
\begin{document}

\title[Sensitive dependence of geometric Gibbs states]{Robust sensitive dependence of geometric Gibbs states for analytic families of quadratic maps}
\author{Daniel Coronel}
\address{Daniel Coronel, Facultad de Matem{\'a}ticas, Pontifica Universidad Cat{\'o}lica de Chile, Avenida Vicu{\~n}a Mackenna~4860, Santiago, Chile}
\email{acoronel@mat.uc.cl}

\author{Juan Rivera-Letelier}
\address{Department of Mathematics, University of Rochester. Hylan Building, Rochester, NY~14627, U.S.A.}
\email{riveraletelier@gmail.com}
\urladdr{\url{http://rivera-letelier.org/}}

\begin{abstract}
  For quadratic-like maps, we show a phenomenon of sensitive dependence of geometric Gibbs states: There are analytic families of quadratic-like maps for which  an arbitrarily small perturbation of the parameter can have a definite effect on the low-temperature geometric Gibbs states.
  Furthermore, this phenomenon is robust: There is an open set of analytic 2\nobreakdash-parameter families of quadratic-like maps that exhibit sensitive dependence of geometric Gibbs states. 
  We introduce a geometric version of the Peierls condition for contour models ensuring that the low-temperature geometric Gibbs states are concentrated near the critical orbit.
\end{abstract}

\maketitle

%
%

\section{Introduction}

A central problem in statistical mechanics and the thermodynamic formalism, is the study of phase transitions.
Here we focus on ``zero-temperature'' phase transitions.\footnote{Also known as the ``chaotic dependence of Gibbs states'' as the temperature parameter drops to zero.}
In good situations, as in the case of contour models satisfying the Peierls condition, Gibbs states converge to a ground state as the temperature drops to zero, see for example~\cite[II, \S2]{Sin82} or the summary in~\cite[\S B.4]{vEnFerSok93}, and~\cite{Bre03,ChaGamUga11,Con16,Lep05} and references therein for other convergence results.
There are several examples of non-convergence, see~\cite{BisGarThi18,ChaHoc10,vEnRus07}, and the companion paper~\cite{CorRiv15a}.

Here we focus on the thermodynamic formalism of analytic maps and geometric potentials.\footnote{Note that in the setting considered in~\cite{BisGarThi18,Bre03,ChaGamUga11,ChaHoc10,CorRiv15a,Con16,Gar17,vEnRus07,Lep05} the map is fixed, and the potential is allowed to vary independently of the map. In contrast, the geometric potential considered here is entirely determined by the map.}
The geometric potential arises naturally in several important problems, like in the construction of physical measures, as in the pioneering work of Sina{\u{\i}}~\cite{Sin72}, Ruelle~\cite{Rue76}, and Bowen~\cite{Bow75}.
The pressure of the geometric potential, as a function of the inverse temperature, is also connected to several multifractal spectra, and large deviations rate functions.

The simplest case of interest is that of circle expanding maps.
A folklore result asserts that generically there is a unique ground state for the geometric potential, and that geometric Gibbs states converge to this ground state as the temperature drops to zero~\cite{CorRiv_expanding}.
On the other hand, some of the non-convergence examples mentioned above can be adapted to the case of smooth circle expanding maps, as shown in~\cite{CorRiv_expanding}.\footnote{For real analytic maps it is an open problem to show that for every real analytic circle expanding map the geometric Gibbs states converge to a ground state as the temperature drops to zero, see~\cite{CorRiv_expanding}.}
However, these examples, as well as those in~\cite{BisGarThi18,ChaHoc10,CorRiv15a,vEnRus07}, are given by constructions that require infinitely many non-trivial conditions.
They are therefore of infinite co-dimension, and a generic finite-dimensional family of circle expanding maps cannot contain such a map.

Our goal is to show that in the next simplest case, of (real and complex) quadratic-like maps on a single variable, the occurrence of zero-temperature phase transitions is a robust phenomenon for 2\nobreakdash-parameter families.
In fact, our main result implies that there is an open set of 2\nobreakdash-parameter families of quadratic-like maps that exhibit a phenomenon of \emph{sensitive dependence of Gibbs states}, which is similar in spirit to the sensitive dependence on initial conditions that is characteristic of chaotic dynamical systems.
More precisely, for a 2\nobreakdash-parameter family of quadratic-like maps in this open set, an arbitrarily small perturbation of the parameters can have a drastic effect on the low-temperature geometric Gibbs states.
In the companion paper~\cite{CorRiv19} we show a similar phenomenon at positive temperature, and in~\cite{CorRiv15a} we study it for classical lattice systems.
The situation is however significantly simpler in~\cite{CorRiv15a}, since the potential is independent of the system and there are no differentiability issues.

One of the main technical difficulties to study geometric Gibbs states of quadratic-like maps is the presence of the critical point, which is a serious obstruction to uniform hyperbolicity.
This leads to some complications, like the fact that there is no obvious characterization of ground states.\footnote{In fact, there are quadratic maps without a Lyapunov minimizing measure, see for example~\cite[Example~5.4]{BruKel98}, \cite[Corollary~2]{BruTod06}, and~\cite[Main Theorem]{CorRiv15a}.}
In particular, the ergodic optimization approach to study low-temperature Gibbs states, described for example in~\cite{BarLepLop13,Con16,Gar17}, breaks down for quadratic-like maps.

The main tool introduced in this paper is the ``Geometric Peierls condition''.
It ensures that the geometric Gibbs states concentrate on the critical orbit as the temperature drops to zero, under certain circumstances.
By considering a critical orbit that accumulates on~2 different periodic orbits with the same Lyapunov exponent,\footnote{In contrast with the usual Peierls condition for countour models, where the ground state is assumed to be supported on a periodic configuration, the Geometric Peierls Condition introduced here is compatible with a non-periodic critical orbit, see Remark~\ref{r:geometric Peierls}.} this creates the non-convergence of geometric Gibbs states.
This is why we use 2-parameter families: One parameter is needed to ensure that these~2 periodic oribts have the same Lyapunov exponent and the other parameter is needed to control the combinatorics of the critical orbit. 
A somewhat similar idea was used by Hofbauer and Keller to produce an example of a quadratic map without a physical measure~\cite{HofKel90}, see also~\cite{HofKel95}.
However, the mechanisms are different: Hofbauer and Keller used long parabolic cascades to control almost every point with respect to the Lebesgue measure; we use a fine control of derivatives of orbits far from the critical orbit to control the mass of the geometric Gibbs states at low temperatures.

To state our results more precisely, we recall the concept of quadratic-like maps of Douady and Hubbard~\cite{DouHub85a}.
Given simply-connected subsets~$U$ and~$V$ of~$\C$ such that the closure of~$U$ is compact and contained in~$V$, a holomorphic map ${f \colon U \to V}$ is a \emph{quadratic-like map} if it is proper of degree~2.
Such a map has a unique point at which the derivative~$Df$ vanishes; it is the \emph{critical point of~$f$}.
The \emph{filled Julia set} of a quadratic-like map~$f \colon U \to V$ is
$$ K(f)
\=
\{ z \in U \mid \text{ for every integer~$n \ge 1$, $f^n(z) \in U$} \}. $$
The \emph{Julia set $J(f)$} of~$f$ is the boundary of~$K(f)$, and it coincides with the closure of the repelling periodic points of~$f$.

Given a quadratic-like map~$f$, denote by~$\sM_f$ the space of probability measures on~$J(f)$ that are invariant by~$f$.
For~$\mu$ in~$\sM_f$ denote by~$h_\mu(f)$ the measure-theoretic entropy of~$\mu$, and for each~$\beta$ in~$\R$ put
$$ P_f(\beta)
\=
\sup \left\{ h_\mu(f) - \beta \int \log |Df| \dd \mu \mid \mu \in \sM_f \right\}. $$
It is the \emph{pressure of~$f|_{J(f)}$ for the potential~$-\beta \log |Df|$}.
A measure~$\mu$ realizing the supremum above is an \emph{equilibrium state of~$f|_{J(f)}$ for the potential~$- t \log |Df|$}, or a \emph{geometric Gibbs state}.

A quadratic-like map~$f \colon U \to V$ is \emph{real} if~$U$ and $V$ are invariant under complex conjugation, and if~$f$ commutes with complex conjugation.
The critical point of such a map is real.
A real quadratic-like map with critical point~$c$ is \emph{essentially topologically exact}, if~$f^2(c)$ is defined and is different from~$f(c)$, if~$f$ maps the interval~$I(f)$ bounded by~$f(c)$ and~$f^2(c)$ to itself, and if~$f|_{I(f)}$ is topologically exact.
For such a map~$f$ we consider both, the interval map~$f|_{I(f)}$, and the complex map~$f$ acting on its Julia set~$J(f)$.

Let~$f$ be a real quadratic-like map that is essentially topologically exact.
Denote by~$\sM_f^{\R}$ the space of all probability measures on~$I(f)$ that are invariant by~$f$.
For~$\mu$ in~$\sM_f^{\R}$ we denote by~$h_\mu(f)$ the measure-theoretic entropy of~$\mu$, and for each~$\beta$ in~$\R$ we put
$$ P^{\R}_f(\beta)
\=
\sup \left\{ h_\mu(f) - \beta \int \log |Df| \dd \mu \mid \mu \in \sM_f^{\R} \right\}. $$
It is the \emph{pressure of~$f|_{I(f)}$ for the potential~$-\beta \log |Df|$}.
A measure~$\mu$ realizing the supremum above is an \emph{equilibrium state of~$f|_{I(f)}$ for the potential~$- t \log |Df|$}, or a \emph{geometric Gibbs state}.

For~$f$ as above, we refer to the family of functions~$(- \beta \log |Df|)_{\beta > 0}$ as the \emph{geometric potentials} of~$f$.
We follow the usual terminology in statistical mechanics, where the parameter~$\beta$ is interpreted as the inverse of the ``temperature''.

\begin{defi}[Sensitive dependence of Gibbs states]
  \label{d:sensitive dependence}
  Let~$\Lambda$ be a topological space, and~$(f_\lambda)_{\lambda \in \Lambda}$ a continuous family quadratic-like maps.
  For~$\lambda_0$ in~$\Lambda$ the family~$(f_\lambda)_{\lambda \in \Lambda}$ has \emph{sensitive dependence of low-temperature geometric Gibbs states at~$\lambda_0$}, if for every sequence~$(\beta_\ell)_{\ell \in \N}$ satisfying~$\beta_\ell \to +\infty$ as~$\ell \to +\infty$, there is a parameter~$\lambda$ in~$\Lambda$ arbitrarily close to~$\lambda_0$ such that:
  \begin{enumerate}
  \item[1.]
    For each~$\beta$ in ${(0, +\infty)}$, there is a unique equilibrium state~$\rho_t(\lambda)$ of~$f|_{J(f_\lambda)}$ for the potential~$-\beta \log |Df_\lambda|$.
  \item[2.]
    The sequence of equilibrium states~$\left( \rho_{\beta_\ell} (\lambda) \right)_{\ell \in \N}$ does not converge in the weak* topology.
  \end{enumerate}
  When this holds, we also say that~$(f_\lambda)_{\lambda \in \Lambda}$ has \emph{sensitive dependence of low-temperature geometric Gibbs states}.

  If in addition for every~$\lambda$ in~$\Lambda$ the quadratic-like map~$f_\lambda$ is real, then for~$\lambda_0$ in~$\Lambda$ the family~$(f_\lambda|_{I(f_\lambda)})_{\lambda \in \Lambda}$ has \emph{sensitive dependence of low-temperature geometric Gibbs states at~$\lambda_0$} if the properties above are satisfied with~$f|_{J(f_\lambda)}$ replaced by~$f|_{I(f_\lambda)}$.
  In this case, we also say that~$(f_\lambda|_{I(f_\lambda)})_{\lambda \in \Lambda}$ has \emph{sensitive dependence of low-temperature geometric Gibbs states}.
\end{defi}

Our main result is stated as the Main Theorem in~\S\ref{ss:Main Theorem}.
The following is a simple consequence of this result, which is easier to state.

\begin{generic}[Sensitive Dependence of Geometric Gibbs States]
  There is an open subset~$\Lambda_0$ of~$\C$ intersecting~$\R$, a holomorphic family of quadratic-like maps~$(\df)_{\lambda \in \Lambda_0}$, and a Cantor set~$\Lambda$ contained in~$\Lambda_0 \cap \R$, such that the following properties hold.
  For every~$\lambda$ in~$\Lambda$ the map~$\df$ is real and essentially topologically exact, and each of the families~$(\df|_{I(f_\lambda)})_{\lambda \in \Lambda}$ and~$(\df)_{\lambda \in \Lambda}$ has sensitive dependence of low-temperature geometric Gibbs states at every parameter in~$\Lambda$.
\end{generic}

In fact, we prove that the conclusions of the Sensitive Dependence of Geometric Gibbs States hold for an open set of holomorphic 2\nobreakdash-parameter families of quadratic-like maps, see Remark~\ref{r:robustness}.
Thus, for quadratic-like maps, the sensitive dependence of geometric Gibbs states is a robust phenomenon for 2\nobreakdash-parameter families.

Note that the Sensitive Dependence of Geometric Gibbs States does not say anything about the behavior of the low-temperature geometric Gibbs states of~$\whf_{\lambda_0}$.
We show that the parameter~$\lambda_0$ can be chosen so that the geometric Gibbs states of~$\whf_{\lambda_0}$ converge as the temperature drops to zero, and that~$\lambda_0$ can be chosen so that they do not converge, see Remark~\ref{r:convergence or divergence}.
In the latter case we show that the set of accumulation measures of the geometric Gibbs states of~$\whf_{\lambda_0}$ is a segment joining certain periodic measures, see Remark~\ref{r:set of accumulation measures}.
In the former case we show that the convergence of the geometric Gibbs states is super-exponential, and that the large deviation principle for Gibbs states studied in~\cite{BarLepLop13} holds with a degenerated rate function, see Remark~\ref{r:convergence}.
Our estimates also show that for every~$\lambda$ in~$\Lambda$ the geometric pressure of~$\df$ is super-exponentially close to its asymptote, see Remark~\ref{r:pressure assymptotic}.

For each~$\lambda$ in~$\Lambda$ the map~$\df$ has a non-recurrent critical point, so it is non-uniformly hyperbolic in a strong sense.
In fact, the maps in the family~$(\df)_{\lambda \in \Lambda}$ satisfy various non-uniform hyperbolicity conditions with uniform constants.
For example, the critical orbit is non-recurrent in a uniform way: There is a neighborhood of~$z = 0$ that for each~$\lambda$ in~$\Lambda$ is disjoint of the forward orbit of the critical value of~$\df$.
Furthermore, all the maps in the family~$(\df)_{\lambda \in \Lambda}$ satisfy the Collet-Eckmann condition with uniform constants: There are constants~$C$ and~$\eta$ satisfying ${C > 0}$ and ${\eta > 1}$, such that for every~$\lambda$ in~$\Lambda$ and every~$n$ in~$\N$, we have~$\left| D\df^n(\df(0)) \right| \ge C \eta^n$.
Moreover, all maps in~$(\df)_{\lambda \in \Lambda}$ have uniform ``goodness constants'' in the sense of~\cite[Definition~2.2]{BalBenSch15}, \emph{cf}. Proposition~\ref{p:transporting Peierls}.
This supports the idea that the lack of expansion is not responsible for the sensitive dependence of geometric Gibbs states.

The Sensitive Dependence of Geometric Gibbs States provides the first examples of an analytic map having a ``zero-temperature'' phase transition.
In the case of a quadratic-like map~$f$, this completes the classification of phase-transitions for~$\beta$ in ${(0, +\infty)}$.\footnote{Compare with the discussion in the introduction of~\cite{CorRiv13}.}
Restricting to transitive maps in the real case, there are only~3 types of phase transitions:
\begin{description}
\item[High-temperature]
  A phase transition at the first zero of the geometric pressure function.
  Such a phase transition appears if and only if~$f$ is not uniformly hyperbolic, and if it does not satisfy the Collet-Eckmann condition, see~\cite[Theorem~A]{NowSan98} or~\cite[Corollary~E]{Riv20} for the real case, and~\cite[Main Theorem]{PrzRivSmi03} for the complex case;
\item[Low-temperature]
  A phase transition occurring after the first zero the geometric pressure function.
  In this case~$f$ cannot be uniformly hyperbolic, and it must satisfy the Collet-Eckmann condition, see~\cite{CorRiv13,CorRiv15b,CorRiv19};
\item[Zero-temperature]
  The geometric pressure function is real analytic on~${(0, +\infty)}$, and for every~$\beta$ in this set there is a unique geometric Gibbs state for the potential~$-\beta \log |Df|$, but these measures do not converge as~$\beta \to +\infty$.
  A map exhibiting such a phase transition must be uniformly hyperbolic, or satisfy the Collet-Eckmann condition.
\end{description}
See~\cite{PrzRiv11,PrzRiv19} and the survey article~\cite{Prz18b} for general results on the thermodynamic formalism of one-dimensional maps.

Roughly speaking, the mechanism responsible for high-temperature phase transitions is the lack of (non-uniform) expansion.
However, the lack of (non-uniform) expansion is not responsible for zero-temperature phase transitions.
The irregular behavior of the critical orbit seems to be responsible for low and zero-temperature phase transitions.
As mentioned above, in~\cite{CorRiv_expanding} we give an example of a smooth circle expanding map having a zero-temperature phase transition.
However, it is an open problem if there is a uniformly hyperbolic quadratic-like map having a zero-temperature phase transition.

\subsection{Notes and references}
\label{ss:notes}
The family of quadratic-like maps~$(\df)_{\lambda \in \Lambda_0}$ in the theorem is given explicitly in~\S\ref{ss:deformation of quadratic}.

It follows from the proof of the Sensitive Dependence of Geometric Gibbs states that there is a definite oscillation of the geometric Gibbs states.
More precisely, there is a continuous function~$\varphi \colon \C \to [0, 1]$ that only depends on~$\lambda_0$, such that that for every~$(\beta_\ell)_{\ell \in \N}$ and~$\lambda$ as in the definition of sensitive dependence of geometric Gibbs states we have
$$ \limsup_{\ell \to +\infty} \int \varphi \dd \rho_{\beta_\ell}(\lambda) = 1,
\text{ and }
\liminf_{\ell \to +\infty} \int \varphi \dd \rho_{\beta_{\ell}}(\lambda) = 0. $$
In fact, at certain temperatures the geometric Gibbs state is super-exponentially close to a certain periodic measure, and at others temperatures they are close to a different periodic measure, see the Main Theorem in~\S\ref{ss:Main Theorem}.

\subsection{Organization}
\label{ss:organization}
After some preliminaries about the quadratic family in~\S\ref{s:preliminaries}, we state the Main Theorem in~\S\ref{s:Main Theorem}, and prove the Sensitive Dependence of Geometric Gibbs states assuming this result (\S\ref{ss:proof of sensitive dependence}).
The Main Theorem is stated for ``uniform families'' of quadratic-like maps, which are defined in~\S\ref{ss:uniform families}.
This notion is inspired from the work of Douady and Hubbard~\cite{DouHub85a}, and it is satisfied for a large class of holomorphic families of quadratic-like maps, see Remark~\ref{r:a priori bounds and uniformity}.

The rest of the paper is devoted to the proof of the Main Theorem.
In~\S\ref{s:uniform estimates} we introduce the Geometric Peierls Condition (Definition~\ref{d:geometric Peierls}), which roughly speaking requires the derivatives along the orbit of the critical value to outweigh the derivatives of orbits that stay far from the critical point.
We also give a criterion for this condition (Proposition~\ref{p:transporting Peierls}) whose proof occupies~\S\S\ref{ss:uniform geometric estimates}, \ref{ss:proving Peierls}.
In~\S\ref{ss:uniform estimates} we make various estimates for uniform families of maps, most of which are deduced from analogous estimates for quadratic maps in~\cite{CorRiv13}.
In~\S\ref{s:estimating pressure} we implement an inducing scheme (\S\ref{ss:Inducing scheme}), analogous to that in~\cite{CorRiv13} for quadratic maps.
For a map satisfying the Geometric Peierls Condition, we also show how to control the pressure of the induced map in terms of the derivatives of the map along the orbit of the critical value (Proposition~\ref{p:improved MS criterion} in~\S\ref{ss:Bowen type formula}).

The proof of the Main Theorem is given in~\S\ref{s:chaotic dependence at zero-temperature}.
After introducing a family of itineraries and other combinatorial objects in~\S\ref{ss:itineraries}, in~\S\ref{ss:estimating postcritical series} we estimate the postcritical series in terms of certain 2~variables series that only depends on the combinatorics of the postcritical orbit (Lemma~\ref{l:2 variables functions}).
The main estimates needed in the proof of the Main Theorem can be stated only in terms of these 2~variables series, and are relegated to Appendix~\ref{s:abstract estimates}.
These are given in an abstract setting that is independent of the rest of the paper.
The proof of the Main Theorem is completed in~\S\ref{ss:proof of Main Theorem}.

\subsection{Acknowledgments}
We would like to thank Jairo Mengue for useful discussions regarding Remark~\ref{r:convergence}.

The first named author acknowledges partial support from FONDECYT grant 1201612.
The second named author acknowledges partial support from NSF grant DMS-1700291.

\section{Preliminaries}
\label{s:preliminaries}
We use~$\N$ to denote the set of integers that are greater than or equal to~1, and~$\N_0 \= \N \cup \{ 0 \}$.

For a Borel measure~$\rho$ on~$\C$, denote by~$\supp(\rho)$ its support.

For an annulus~$A$ contained in~$\C$, we use~$\modulus(A)$ to denote the conformal modulus of~$A$.

\subsection{Koebe principle}
\label{ss:Koebe}
We use the following version of Koebe distortion theorem that can be found, for example, in~\cite{McM94b}.
Given an open subset~$G$ of~$\C$ and a map~$f \colon G \to \C$ that is a biholomorphism onto its image, the \emph{distortion of~$f$} on a subset~$C$ of~$G$ is
$$ \sup_{x, y \in C} \frac{|Df(x)|}{|Df(y)|}. $$
\begin{generic}[Koebe Distortion Theorem]
  For each~$A > 0$ there is a constant~$\Delta > 1$ such that for each topological disk~$\hW$ contained in~$\C$ and each compact set~$K$ contained in~$\hW$ and such that~$\hW \setminus K$ is an annulus of modulus at least~$A$, the following property holds: For each open topological disk~$U$ contained in~$\C$ and every biholomorphic map~$f \colon U \to \hW$, for every~$x$, $y$ and~$z$ in $f^{-1}(K)$ we have 
  $$
  \Delta^{-1}|Df(z)|
  \le
  \frac{|f(x)-f(y)|}{|x-y|}
  \le
  \Delta |Df(z)|.
  $$
  Moreover, the distortion of~$f$ on~$f^{-1}(K)$ is bounded by~$\Delta$.
\end{generic}

\subsection{Quadratic polynomials, Green's functions, and B{\"o}ttcher coordinates}
\label{ss:quadratic polynomials}
In this subsection and the next we recall some basic facts about the dynamics of
complex quadratic polynomials, see for instance~\cite{CarGam93} or~\cite{Mil06c}
for references.

For~$c$ in~$\C$ we denote by~$f_c$ the complex quadratic polynomial
$$ f_c(z) \= z^2 + c, $$
and by~$K_c$ the \emph{filled Julia set} of $f_c$; that is, the set of all
points~$z$ in~$\C$ whose forward orbit under~$f_c$ is bounded in~$\C$.
The set~$K_c$ is compact and its complement is the connected set consisting of
all points whose orbit converges to infinity in the Riemann sphere.
Furthermore, we have $f_c^{-1}(K_c) = K_c$ and~$f_c(K_c) = K_c$. 
The boundary~$J_c$ of~$K_c$ is the \emph{Julia set of~$f_c$}.

For a parameter~$c$ in~$\C$, the \emph{Green's function of~$K_c$} is the function
$G_c \colon \C \to [0,+\infty)$ that is identically~$0$ on~$K_c$, and that
for~$z$ outside~$K_c$ is given by the limit,
\begin{equation}
  \label{def:Green function}
  G_c(z) \= \lim_{n\rightarrow +\infty} \frac{1}{2^n} \log |f_c^n(z)| > 0.
\end{equation}
The function~$G_c$ is continuous, subharmonic, satisfies~$G_c \circ f_c = 2G_c$ on~$\C$, and it is harmonic and strictly positive outside~$K_c$.
On the other hand, the critical values of~$G_c$ are bounded from above by~$G_c(0)$, and
the open set
$$ U_c \= \{z\in \C \mid G_c(z) > G_c(0)\} $$
is homeomorphic to a punctured disk.
Notice that $G_c(c)=2G_c(0)$, thus~$U_c$ contains~$c$ if~$0$ is not in~$K_c$. 

By B{\"o}ttcher's Theorem there is a unique conformal representation
\[
  \varphi_c\colon U_c
  \rightarrow
  \{z\in \C \mid |z| > \exp (G_c(0)) \},
\]
and this map conjugates~$f_c$ to $z \mapsto z^2$.
It is called \emph{the B{\"o}ttcher coordinate of~$f_c$} and satisfies $G_c =
\log |\varphi_c|$.

\subsection{External rays and equipotentials}
\label{ss:rays and equipotentials}
Let~$c$ be in~$\C$.
For~$v > 0$ the \emph{equipotential~$v$ of~$f_c$} is by definition~$G_c^{-1}(v)$.
A \emph{Green's line of~$G_c$} is a smooth curve on the complement of~$K_c$ in~$\C$ that is orthogonal to the equipotentials of~$G_c$ and that is maximal with this property. 
Given~$t$ in~$\R / \Z$, the \emph{external ray of angle~$t$ of~$f_c$}, denoted by~$R_c(t)$, is the Green's line of~$G_c$ containing
$$ \{ \varphi_c^{-1}(r \exp(2 \pi i t)) \mid \exp(G_c(0))< r < +\infty \}. $$
By the identity~$G_c \circ f_c= 2G_c$, for each~$v > 0$ and each~$t$ in~$\R / \Z$ the map~$f_c$ maps the equipotential~$v$ to the equipotential~$2v$ and maps~$R_c(t)$ to~$R_c(2t)$.
For~$t$ in~$\R / \Z$ the external ray~$R_c(t)$ \emph{lands at a point~$z$}, if~$G_c \colon R_c(t) \to (0, +\infty)$ is a bijection and if~$G_c|_{R_c(t)}^{-1}(v)$ converges to~$z$ as~$v$ converges to~$0$ in~$(0, +\infty)$.
By the continuity of~$G_c$, every landing point is in $J_c = \partial K_c$.

The \emph{Mandelbrot set~$\cM$} is the subset of~$\C$ of those
parameters~$c$ for which~$K_c$ is connected.
The function
\[
  \begin{array}{cccl}
    \Phi \colon &\C \setminus \cM & \to & \C \setminus \cl{\D}\\
                &     c         &  \mapsto           & \Phi(c) \= \varphi_c(c)
  \end{array}
\]
is a conformal representation, see~\cite[VIII, \emph{Th{\'e}or{\`e}me}~1]{DouHub84}.
For~$v > 0$ the \emph{equipotential~$v$ of~$\cM$} is by definition
$$ \cE(v) \= \Phi^{-1}(\{z\in \C \mid |z| = v \}). $$ 
On the other hand, for~$t$ in~$\R / \Z$ the set
$$ \cR(t) \= \Phi^{-1}(\{r \exp(2 \pi i t) \mid r > 1 \}) $$
is called the \emph{external ray of angle~$t$ of~$\cM$}.
We say that $\cR(t)$ \emph{lands at a point~$z$} in~$\C$, if~$\Phi^{-1} (r \exp(2\pi i t))$ converges to~$z$ as $r \searrow 1$.
When this happens~$z$ belongs to~$\partial \cM$.

\subsection{The wake~1/2}
\label{ss:wake 1/2}
In this subsection we recall a few facts that can be found for example
in~\cite{DouHub84} or~\cite{Mil00c}.

The external rays~$\cR(1/3)$ and~$\cR(2/3)$ of~$\cM$ land at the parameter~$c = -3/4$, and these are the only external rays of~$\cM$ that land at this point, see
for example~\cite[Theorem~1.2]{Mil00c}.
In particular, the complement in~$\C$ of the set
$$ \cR(1/3) \cup \cR(2/3) \cup \{ - 3/4 \} $$
has~2 connected components; we denote by~$\cW$ the connected component
containing the point~$c = -2$ of~$\cM$.

For each parameter~$c$ in~$\cW$ the map~$f_c$ has~2 distinct fixed points; one
of the them is the landing point of the external ray~$R_c(0)$ and it is denoted
by~$\beta(c)$; the other one is denoted by~$\alpha(c)$.
The only external ray landing at~$\beta(c)$ is~$R_c(0)$, and
the only external ray landing at~$-\beta(c)$
is~$R_c(1/2)$.

Moreover, for every parameter~$c$ in~$\cW$ the only external rays 
of~$f_c$ landing at~$\alpha(c)$ are~$R_c(1/3)$
and~$R_c(2/3)$, see for example~\cite[Theorem~1.2]{Mil00c}.
The complement of~$R_c(1/3) \cup R_c(2/3) \cup \{ \alpha(c)
\}$ in~$\C$ has~2 connected components; one containing~$- \beta(c)$ and~$z =
c$, and the other one containing~$\beta(c)$ and~$z = 0$.
On the other hand, the point~$\alpha(c)$ has~2 preimages by~$f_c$: Itself and~$\talpha(c) \= - \alpha(c)$.
The only external rays landing at~$\talpha(c)$ are~$R_c(1/6)$
and~$R_c(5/6)$.

\subsection{Yoccoz puzzles and para-puzzle}
\label{ss:puzzles}
In this subsection we recall the definitions of Yoccoz puzzles and para-puzzle.
We follow \cite{Roe00}.

\begin{defi}[Yoccoz puzzles]
  Fix~$c$ in~$\cW$ and consider the open region $X_c \= \{z\in \C \mid G_c(z) <
  1\}$. 
  The \emph{Yoccoz puzzle of~$f_c$} is given by the following sequence of
  graphs~$(I_{c, n})_{n = 0}^{+\infty}$ defined for~$n = 0$ by:
  \[
    I_{c,0} \= \partial X_c \cup (X_c \cap \cl{R_c(1/3)} \cap \cl{R_c(2/3)}),
  \]
  and for~$n \ge 1$ by~$I_{c,n} \= f_c^{-n}(I_{c,0})$.
  The \emph{puzzle pieces of depth~$n$} are the connected components of $f_c^{-n}(X_c) \setminus I_{c,n}$.
  The puzzle piece of depth~$n$ containing a point~$z$ is denoted by~$P_{c,n}(z)$.
\end{defi}

Note that for a real parameter~$c$, every puzzle piece intersecting the real line is invariant under complex conjugation.
Since puzzle pieces are simply-connected, it follows that the intersection of such a puzzle piece with~$\R$ is an interval.

\begin{defi}[Yoccoz para-puzzle\footnote{In contrast to~\cite{Roe00}, we only consider the para-puzzle in the wake~$\cW$.}]
  Given an integer~$n \ge 0$, put
  $$ J_n
  \=
  \{t\in [1/3,2/3] \mid 2^n t ~ (\mathrm{mod}\, 1) \in \{1/3,2/3\} \}, $$
  let~$\cX_n$ be the intersection of~$\cW$ with the open region in the parameter plane bounded by the equipotential~$\cE(2^{-n})$ of~$\cM$, and put
  \[
    \cI_{n}
    \=
    \partial \cX_n \cup \left( \cX_n \cap \bigcup_{t\in J_n} \cl{\cR(t)} \right).
  \]
  Then the \emph{Yoccoz para-puzzle of~$\cW$} is the sequence of graphs~$(\cI_n)_{n = 0}^{+\infty}$.
  The \emph{para-puzzle pieces of depth~$n$} are the connected components of $\cX_n \setminus \cI_n$.
  The para-puzzle piece of depth~$n$ containing a parameter~$c$ is denoted by~$\cP_n(c)$.  
\end{defi}
Observe that there is only~1 para-puzzle piece of depth~$0$, and only~1 para-puzzle piece of depth~1; they are bounded by the same external rays but different equipotentials.
Both of them contain~$c = - 2$.

Fix a parameter~$c$ in~$\cP_0(-2)$.
There are precisely~2 puzzle pieces of depth~$0$: $P_{c, 0}(\beta(c))$ and~$P_{c, 0}(-\beta(c))$.
Each of them is bounded by the equipotential~1 and by the closures of the
external rays landing at~$\alpha(c)$.
Furthermore, the critical value~$c$ of~$f_c$ is contained 
in~$P_{c, 0}(- \beta(c))$ and the critical point in~$P_{c, 0}(\beta(c))$.
It follows that the set~$f_c^{-1}(P_{c, 0}(\beta(c)))$ is the 
disjoint union of~$P_{c, 1}(- \beta(c))$ and~$P_{c, 1}(\beta(c))$, so~$f_c$ maps each of the sets~$P_{c, 1}(-\beta(c))$ and~$P_{c,
  1}(\beta(c))$ biholomorphically to~$P_{c, 0}(\beta(c))$.
Moreover, there are precisely~3 puzzle pieces of depth~1: 
$$ P_{c, 1}(-\beta(c)),
P_{c, 1}(0)
\quad \text{and} \quad
P_{c, 1}(\beta(c)); $$
$P_{c, 1}(- \beta(c))$ is bounded by the equipotential~1/2 and by the 
closures of the external rays that land at~$\alpha(c)$; $P_{c, 1}(\beta(c))$ is
bounded by the equipotential~1/2 and by the closures of the external rays that
land at~$\talpha(c)$; and~$P_{c, 1}(0)$ is bounded by the equipotential~1/2
and by the closures of the external rays that land at~$\alpha(c)$ and
at~$\talpha(c)$.
In particular, the closure of~$P_{c, 1}(\beta(c))$ is 
contained in~$P_{c, 0}(\beta(c))$.
It follows from this that for each integer~$n \ge 1$ the map~$f_c^n$ maps~$P_{c, n}(- \beta(c))$ biholomorphically to~$P_{c, 0}(\beta(c))$.

\subsection{The uniformly expanding Cantor set}
\label{ss:expanding Cantor set}
For a parameter~$c$ in~$\cP_3(-2)$, the maximal invariant set~$\Lambda_c$ of~$f_c^3$ in~$P_{c, 1}(0)$ plays an important r{\^o}le in the proof of the Main Theorem.

Fix~$c$ in~$\cP_3(-2)$.
There are precisely~2 connected components of~$f_c^{-3}(P_{c, 1}(0))$ contained in~$P_{c, 1}(0)$ that we denote by~$Y_c$ and~$\tY_c$.
The closures of these sets are disjoint and contained in~$P_{c, 1}(0)$.
The sets~$Y_c$ and~$\tY_c$ are distinguished by the fact that~$Y_c$ contains in its boundary the common landing point of the external rays~$R_c(7/24)$ and~$R_c(17/24)$, denoted~$\gamma(c)$, and that~$\tY_c$ contains in its boundary the common landing point of the external rays~$R_c(5/24)$ and~$R_c(19/24)$.
The map~$f_c^3$ maps each of the sets~$Y_c$ and~$\tY_c$ biholomorphically to~$P_{c, 1}(0)$.
Thus, if we put
\[
  \begin{array}{cccl}
    g_c & \colon Y_c \cup \tY_c & \to & P_{c,1}(0)\\
        &    z & \mapsto & g_c(z) \= f_c^{3}(z),
  \end{array}
\]
then
$$ \Lambda_c = \bigcap_{n\in \N} g_c^{-n}(\cl{P_{c,1}(0)}). $$

\subsection{Parameters}
\label{ss:Parameters}
In this subsection we recall the definition of a certain parameter sets in~\cite[Proposition~3.1]{CorRiv13} that are important in what follows.
To avoid confusions with the notation introduced in~\S\ref{ss:wake 1/2}, in this section, and in the rest of the paper, we use the parameter~$t$ to denote the inverse temperature.

Given an integer~$n \ge 3$, let~$\cK_n$ be the set of all those real parameters~$c$ in~${(-\infty, 0)}$ such that
$$ f_c(c) > f_c^2(c) > \cdots > f_c^{n-1}(c) > 0
\quad \text{and} \quad
f_c^n(c) \in \Lambda_c. $$
Note that for a parameter~$c$ in~$\cK_n$, the first return time of~$0$ to~$P_{c, 1}(0)$ is equal to ${n + 1}$.
On the other hand, the critical point of~$f_c$ cannot be asymptotic to a non-repelling periodic point.
This implies that all the periodic points of~$f_c$ in~$\C$ are hyperbolic
repelling and therefore that~$K_c = J_c$, see~\cite{Mil06c}.
On the other hand, we have~$f_c(c) > c$ and the interval~$I_c = [c, f_c(c)]$ is
invariant by~$f_c$.
This implies that~$I_c$ is contained in~$J_c$ and hence that for every real
number~$t$ we have~$P_c^{\R}(t) \le P_c(t)$.
Note also that~$f_c|_{I_c}$ is not renormalizable, so~$f_c$ is topologically
exact on~$I_c$, see for example~\cite[Theorem~III.4.1]{dMevSt93}.

Since for~$c$ in~$\cK_n$ the critical point of~$f_c$ is not periodic, for every integer~$k \ge 0$ we have~$f_c^{n + 3k}(c) \neq 0$.
The \emph{itinerary of~$f_c$}, is the sequence~$\iota(c)$ in~$\{0, 1 \}^{\N_0}$ defined for each~$k$ in~$\N_0$ by
$$ \iota(c)_k
\=
\begin{cases}
  0 & \text{ if }  f_c^{n + 3k}(c) \in Y_c; \\
  1  & \text{ if } f_c^{n + 3k}(c) \in \tY_c.
\end{cases} $$

\begin{prop}
  \label{p:ps}
  For each integer~$n \ge 3$, the set~$\cK_n$ is a compact subset of
  $$ \cP_n(-2) \cap (-2, -3/4), $$
  and the function ${\iota \colon \cK_n \to \{0, 1\}^{\N_0}}$ is a homeomorphism.
  Finally, for each~$\delta > 0$ there is~$n_0 \ge 3$ such that for each
  integer~$n \ge n_0$ the set~$\cK_n$ is contained in the interval~$(-2, -2 +
  \delta)$.
\end{prop}

\begin{proof}
  Except for the assertion that~$\iota$ is a homeomorphism, this is~\cite[Proposition~3.1]{CorRiv13}.
  In this last result it is shown that~$\iota$ is a bijection, so it only remains to observe that, since for each~$c$ in~$\cP_3(-2)$ the map~$f_c$ is uniformly expanding on~$\Lambda_c$ \cite[\S3.3]{CorRiv13}, the map~$\iota$ is continuous, and therefore a homeomorphism.
\end{proof}

\begin{rema}
  \label{r:kneading}
  The proposition implies that for every integer~$n$ satisfying ${n \ge 3}$ and every~$\uiota$ in~$\{0, 1\}^{\N_0}$, there is a unique~$c$ in~$\cK_n$ for which the itinerary~$\iota(c)$ of~$f_c$ is equal to~$\uiota$.
  The real parameter~$c$ is uniquely characterized by the following properties:
  \begin{itemize}
  \item
    For every~$j$ in~$\{1, \ldots, n - 1\}$, we have ${f_c^j(c) > 0}$;
  \item
    For every~$k$ in~$\N_0$ and~$r$ in~$\{0, 1, 2\}$, we have
    \begin{equation*}
      f_c^{n + 3k + r}(c)
      \begin{cases}
        > 0
        & \text{if } \uiota_k = 1 \text{ and $r = 0$, or if } r = 2;
        \\
        < 0
        & \text{if } \uiota_k = 0 \text{ and $r = 0$, or if } r = 1. 
      \end{cases}
    \end{equation*}
  \end{itemize}
\end{rema}

\section{Main Theorem}
\label{s:Main Theorem}
In this section we state the Main Theorem, and prove the Sensitive Dependence of Geometric Gibbs States assuming this result.

The Main Theorem is stated in~\S\ref{ss:Main Theorem} for ``uniform families'' of quadratic-like maps, which are defined in~\S\ref{ss:uniform families}.
By the work of Douady and Hubbard~\cite{DouHub85a}, there is a large class of holomorphic families of quadratic-like maps that are uniform, see Remark~\ref{r:a priori bounds and uniformity}.
We use this to exhibit in~\S\ref{ss:deformation of quadratic} a concrete (real) 1\nobreakdash-parameter family of quadratic-like maps satisfying the hypotheses of the Main Theorem.
This family is used in~\S\ref{ss:proof of sensitive dependence} to prove the Sensitive Dependence of Geometric Gibbs States assuming the Main Theorem.

\subsection{Uniform families of quadratic-like maps}
\label{ss:uniform families}
A quadratic-like map~$f \colon U \to V$ is \emph{normalized}, if its unique critical point is~$0$, and if~$D^2f(0) = 2$.
For such a map~$f$ there is a holomorphic function~$R_f \colon U \to \C$ such that for~$w$ in~$U$ we have
$$ f(w) = f(0) + w^2 + w^3 R_f(w). $$
Note that~$f$ is uniquely determined by its critical value~$f(0)$, and the function~$R_f$.

In what follows, we endow a given family of normalized quadratic-like maps with a topology that will be used in the statement of the Main Theorem.
Let~$\sU$ be the set of all topological disks in~$\C$ containing~$0$ that are not biholomorphic to~$\C$. 
Endow~$\sU$ with the \emph{Carath\'eorody topology}, which is the topology generated by the following families of subsets of $\sU$:
$$
\{\{U\in \sU \mid K\subseteq U \} \mid K\subseteq \C \text{ compact} \}
$$
and
$$
\{\{U \in \sU \mid N \nsubseteq U\} \mid N\subseteq \C \text{ open and connected with } 0\in N\}.
$$
Let~$\sS$ be the set of all holomorphic univalent maps~$f$ defined on the open unit disk~$\D$, such that ${f(0) = 0}$ and ${Df(0) > 0}$.
Equip~$\sS$ with the topology of locally uniform convergence.
The map ${R \colon \sS \to \sU}$ sending~$f$ in~$\sS$ to~$f(\D)$ in~$\sU$ is a bijection by the Riemann mapping theorem, and a homeomorphism by~\cite[\S 4]{Oes86}.
Let~$\sH$ be the set of all holomorphic functions defined on~$\D$, equipped with the topology of locally uniform convergence.
Using the homeomorphism~$R$ we can inject each normalized family of quadratic-like maps~$\sF$ into~$\sU \times \sH$, sending ${f \colon U \to \C}$ to ${(U, f\circ R^{-1}(U))}$.
Using this injection, we endow~$\sF$ with the pull-back of the product topology on~$\sU\times \sH$.
We call this topology the \emph{Carath{\'e}odory topology} or the \emph{topology locally uniform convergence on~$\sF$}.
Thus, a sequence ${f_n \colon U_n\to \C}$ in~$\sF$ converges to ${f \colon U \to \C}$ in~$\sF$ if and only if ${(U_n, f_n\circ R^{-1}(U_n))}$ converges to $(U, f\circ R^{-1}(U))$ in~$\sU\times \sS$.
Observe that this is equivalent to~$U_n$ converging to~$U$ in the Carath\'eodory topology on~$\sU$, and that for every compact set~$K$ in~$U$ the sequence~$f_n$ converges uniformly to~$f$ on~$K$.
The latter is the usual definition of the Carath{\'e}odory convergence of maps, see \cite[\S5.1]{McM94b}.

By the straightening theorem of Douady and Hubbard \cite{DouHub85a}, for every quadratic-like map~$f \colon U \to V$ there is~$c$ in~$\C$ and a quasi\nobreakdash-conformal homeomorphism~$h \colon \C \to \C$ that conjugates the quadratic polynomial~$f_c$ to~$f$ on a neighborhood of~$J_c$.
In the case~$f$ is real, $c$ is real, and~$h$ can be chosen so that it commutes with the complex conjugation.
In all the cases, the quasi\nobreakdash-conformal homeomorphism~$h$ can be chosen to be holomorphic on a neighborhood of infinity, and tangent to the identity there.

Put
$$ \cX \= \{ c \in \C \mid G_c(c) \le 1 \}
\quad \text{and} \quad
\hcX \= \{ c \in \C \mid G_c(c) \le 2 \}, $$
and for~$c$ in~$\C$, put
$$ X_c \= \{ z \in \C \mid G_c(z) \le 1 \}
\quad \text{and} \quad
\hX_c \= \{ z \in \C \mid G_c(z) \le 2 \}. $$
Note that~$X_c$ is contained in the interior of~$\hX_c$, and that
$$ \cX = \{ c \in \C \mid c \in X_c \}
\quad \text{and} \quad
\hcX = \{ c \in \C \mid c \in \hX_c \}. $$
\begin{defi}[Uniform family of quadratic-like maps]
  A family~$\sF$ of normalized quadratic-like maps is \emph{uniform}, if there are constants~$K\ge 1$ and~$R>0$, such that for each~$f$ in~$\sF$ there are~$c(f)$ in~$\cX$ and a $K$\nobreakdash-quasi\nobreakdash-conformal homeomorphism~$h_f$ of~$\C$ satisfying the following properties.
  \begin{enumerate}
  \item[1.]
    The homeomorphism~$h_f$ conjugates~$f_{c(f)}$ on~$\hX_{c(f)}$ to~$f$ on~$h_f(\hX_{c(f)})$.
    Furthermore, if~$f$ is real, then~$h_f$ commutes with the complex conjugation.
  \item[2.]
    The set~$\hX_{c(f)}$ is contained in~$B(0, R)$, and the homeomorphism~$h_f$ is holomorphic on~$\C \setminus \cl{B(0, R)}$, and it is tangent to the identity at infinity.
  \end{enumerate}
\end{defi}

Note that property~1 implies that~$h_f(0) = 0$.

\begin{rema}
  \label{r:a priori bounds and uniformity}
  Although it is not needed in this paper, we remark that a family~$\sF$ of normalized quadratic-like maps with connected Julia sets is uniform if and only if  the following property holds: There is a constant~$m > 0$ such that for each~$f \colon U \to V$ in~$\sF$ there is an essential annulus in $V \setminus U$ whose conformal modulus is at least~$m$.
\end{rema}

Let~$\sF$ be a uniform family of quadratic-like maps.
For each~$f$ in~$\sF$ put
$$ X_f \= h_f(X_{c(f)})
\quad \text{and} \quad
\hX_f \= h_f(\hX_{c(f)}). $$
By the definition of uniform family, the puzzle pieces of~$f_{c(f)}$ can be push-forward to~$X_f$ by~$h_f$.
We call to these sets the \emph{puzzle pieces of~$f$}.
We say that a puzzle piece of $f$ has \emph{depth}~$n$  if it is the push-forward of a puzzle piece of $c(f)$ with depth~$n$.
The puzzle piece of depth~$n$ of~$f$ containing~$w$ is denoted~$P_{f,n}(w)$.
Thus, we have
$$
P_{f,n}(w) \= h_f(P_{c(f),n}(h_f^{-1}(w))).
$$
Set 
$$\beta(f) \= h_f(\beta(c(f)))
\quad \text{and} \quad
\tbeta(f) \= h_f(-\beta(c(f))).$$

For every  integer $n\ge 0$, put
$$\cP_n(\sF) \= \{f \in \sF \mid c(f) \in \cP_n(-2)\},$$  
and for
$n \ge 3$, put 
$$\cK_n(\sF) \= \{f \in \sF \mid c(f) \in \cK_n\}.$$  

Moreover, for~$f$ in~$\cP_3(\sF)$ put
\begin{displaymath}
  Y_f \= h_f(Y_{c(f)}),
  \text{ and }
  \tY_f \= h_f(\tY_{c(f)}),
\end{displaymath}
and let~$g_f \colon h_f(Y_{c(f)}\cup \tY_{c(f)}) \to P_{f, 1} (0)$ be defined by~$g_f \= h_f \circ g_{c(f)}\circ h_f^{-1}$.
Denote by~$p(f)$ and~$\wtp(f)$ the unique fixed point of~$g_f$ in~$Y_f$ and $\tY_f$, respectively, and denote by~$\whp(f)$ the unique fixed point of~$g_f^2$ in~$\tY_f$ that is different from~$\wtp(f)$; it is a periodic point of~$g_f$ of minimal period~2.
Furthermore, denote by
\begin{displaymath}
  \cO^+(f) \= \left\{ f^j(p^+(f)) \mid j \in \{0, 1, 2 \} \right\}
  \text{ and }
  \cO^-(f) \= \left\{ f^j(p^-(f)) \mid j \in \{0, 1, \dots, 5 \} \right\}
\end{displaymath}
the orbits of~$p^+(f)$ and~$p^-(f)$ under~$f$, respectively.

For every~$f$ in~$\sF$ such that~$c(f)$ is real and belongs to~$[-2,0)$, denote by~$I(f)$ the image under~$h_f$ of the interval ${[c(f), f_{c(f)}(c(f))]}$.
In the case where~$f$ is real, the set~$I(f)$ is a subinterval of~$\R$.
In all of the cases, $I(f)$ is a closed topological arc satisfying ${f(I(f)) = I(f)}$.
A quadratic-like map~$f$ in~$\sF$ is \emph{essentially topologically exact} if~$c(f)$ is in $[-2,0)$, and if~$f|_{I(f)}$ is topologically exact.
For such a map~$f$ we consider both, the map~$f|_{I(f)}$, and the complex map~$f$ acting on its Julia set~$J(f)$.
We also define~$\sM_f^{\R}$, $P_f^{\R}$, and \emph{equilibrium states} or \emph{Geometric Gibbs states} of~$f|_{I(f)}$ as in the introduction. 
In the case~$f$ is real, the definitions above coincide with those in the introduction.

Let~$n$ be an integer satisfying ${n \ge 5}$, and let~$f$ be in~$\cK_n(\sF)$.
The \emph{itinerary of~$f$}, is the sequence~$\iota(f)$ defined by ${\iota(f) \= \iota(c(f))}$, see~\S\ref{ss:Parameters}.
So, for every~$k$ in~$\N_0$ we have  
$$ \iota(f)_k = \begin{cases}
  0 & \text{ if } f^{n + 3k}(f(0)) \in Y_f;\\
  1 & \text{ if } f^{n + 3k}(f(0)) \in \tY_f.
\end{cases} $$
Here is an alternative description of the combinatorics of~$f$ in terms of its itinerary~$\iota(f)$, in the spirit of kneading theory.
Note first that the parameter~$c(f)$ is real and belongs to~$[-2, 0)$, and that the critical point~$0$ of~$f$ and its orbit are contained in~$I(f)$.
Removing~$0$ from the closed arc~$I(f)$, we obtain~2 disjoint semi-open arcs
\begin{equation*}
  I(f)^+
  \=
  h_f((0, f_{c(f)}(c(f))])
  \text{ and }
  I(f)^-
  \=
  h_f([c(f), 0)).
\end{equation*}
In view of Remark~\ref{r:kneading}, we have the following properties:
\begin{itemize}
\item
  For every~$j$ in~$\{1, \ldots, n - 1\}$, the point~$f^j(f(0))$ is in~$I(f)^+$;
\item
  For every~$k$ in~$\N_0$ and~$r$ in~$\{0, 1, 2\}$, we have
  \begin{equation*}
    f^{n + 3k + r}(f(0))
    \begin{cases}
      \in I(f)^+
      & \text{if } \iota(f)_k = 1 \text{ and $r = 0$, or if } r = 2;
      \\
      \in I(f)^-
      & \text{if } \iota(f)_k = 0 \text{ and $r = 0$, or if } r = 1. 
    \end{cases}
  \end{equation*}
\end{itemize}

\subsection{Main Theorem}
\label{ss:Main Theorem}
For every normalized quadratic-like map~$f$, and every periodic point~$p$ of~$f$ with period~$m$ in $\N$, put
$$
\chi_f(p)\=\frac{1}{m} \log|Df^m(p)|.
$$

In this subsection we state the Main Theorem, which is based on the following concept.
\begin{defi}[Admissible family of quadratic-like maps]
  \label{d:admissibility}
  A uniform family of quadratic-like maps~$\sF$ is \emph{admissible}, if for every sufficiently large integer~$n$ the following properties hold.
  \begin{enumerate}
  \item[1.]
    If we endow~$\sF$ with the topology of locally uniform convergence, then there is a continuous function~$s_n \colon \cK_n \to \cK_n(\sF)$ such that~$c \circ s_n$ is the identity.
  \item[2.]
    For every~$f$ in~$s_n(\cK_n)$, we have
    \begin{equation}
      \label{eq:1}
      \chi_f(p(f))
      >
      \chi_f(p^+(f))
      \text{ and }
      \chi_f(p^+(f))
      =
      \chi_f(p^-(f)).
    \end{equation}  
  \end{enumerate}
\end{defi}

Before stating the Main Theorem, we give a rough description of the combinatorics of the maps appearing on it.
The Main Theorem asserts that for every admissible uniform family of quadratic-like maps, there is a continuous subfamily~$(\fs)_{\uvarsigma \in \signs}$ satisfying some properties regarding the limit behavior of the equilibrium states as the temperature drops to zero.
The main feature of this subfamily is that for every map~$\fs$ on it, the orbit of the critical point remains most of the time close to the orbits of the periodic points~$p^+(\fs)$ and~$p^-(\fs)$.
The sequence~$\uvarsigma$ in~$\signs$ indicates how the orbit of the critical point alternates between these periodic orbits, or remains close to them.
For example, a large string of repeated~$+$'s in~$\uvarsigma$ indicates that the orbit of the critical point is very close to~$p^+(\fs)$ for some period of time.
With a careful choice of the family of itineraries~$(\iota(\fs))_{\uvarsigma \in \signs}$, we show that for each~$m$ in~$\N$ there is a certain range of inverse temperatures for which the corresponding equilibrum states of~$\fs$ are either concentrated near the orbit of~$p^+(\fs)$, or that of~$p^-(\fs)$, depending on whether ${\uvarsigma_m = +}$ or ${\uvarsigma_m = -}$.
Another interesting feature of this construction, is that the limit behavior of the equilibrium states of~$\fs$ depends on the tail of the sequence~$\uvarsigma$.
So, we can have maps of the subfamily that are close among them but with significantly different behavior of its equilibrium states, which is the basic idea behind the sensitive dependence notion stated in the introduction.

Endow the set~$\{+, - \}$ with the discrete topology, and~$\signs$ with the corresponding product topology.

\begin{maintheo}
  For every~$R > 0$ there is a constant~$K_0 > 1$ such that if~$\sF$ is an admissible uniform family of quadratic-like maps with constants~$K_0$ and~$R$, then for every sufficiently large integer~$n$ there is a continuous subfamily~$(\fs)_{\uvarsigma \in \signs}$ of~$s_n(\cK_n)$ such that the following properties hold.
  \begin{enumerate}
  \item[1.]
    For each~$\uvarsigma$ in~$\signs$ the map~$\fs$ is essentially topologically exact.
    Moreover, for each~$t$ in~${(0, +\infty)}$ there is a unique equilibrium state~$\rho_{t}^{\R}(\uvarsigma)$ (resp. $\rho_{t}(\uvarsigma)$) of~$\fs|_{I(\fs)}$ (resp.~$\fs|_{J(\fs)}$) for the potential $- t \log |D\fs|$.
  \item[2.]
    There are constants~$C_0 > 0$ and~$\upsilon_0 > 0$, and a continuous function
    \begin{displaymath}
      A \colon \signs \to (0, +\infty),
    \end{displaymath}
    such that for every sequence~$\uvarsigma = (\varsigma(m))_{m \in \N}$ in~$\signs$, the following properties hold.
    Let~$m$ and~$\whm$ be integers such that
    $$ \whm \ge m \ge 1
    \quad \text{and} \quad
    \varsigma(m) = \cdots = \varsigma(\whm), $$
    and let~$t$ be in~$[A(\uvarsigma) m, A(\uvarsigma) \whm]$.
    Then the equilibrium state~$\rho_{t}^{\R}(\uvarsigma)$ (resp.~$\rho_{t}(\uvarsigma)$) of~$\fs|_{I(\fs)}$ (resp.~$\fs|_{J(\fs)}$) is super-exponentially close to the orbit~$\cO^{\varsigma(m)}(\fs)$ of~$p^{\varsigma(m)}(\fs)$, in that
    \begin{multline*}
      \rho_t^{\R}(\uvarsigma) \left( B \left( \cO^{\varsigma(m)}(\fs), \exp(- \upsilon_0 t^2) \right) \right)
      \ge
      1 - C_0 \exp(- \upsilon_0 t^2)
      \\
      \left( \text{resp. } \rho_{t}(\uvarsigma) \left( B \left( \cO^{\varsigma(m)}(\fs), \exp(- \upsilon_0 t^2) \right) \right)
        \ge
        1 - C_0 \exp(- \upsilon_0 t^2)
      \right).
    \end{multline*}
  \end{enumerate}
\end{maintheo}

Note that for each~$\uvarsigma$ in~$\signs$ the map~$\fs$ has a non-recurrent critical point, so it is non-uniformly hyperbolic in a strong sense.

\begin{rema}[Robustness]
  \label{r:robustness}
  It follows from the theory of quadratic-like maps of Douady and Hubbard~\cite{DouHub85a} that condition~1 in Definition~\ref{d:admissibility} is satisfied for every holomorphic 1\nobreakdash-parameter family of quadratic-like maps $(\whf_\lambda)_{\lambda \in \Lambda_0}$ intersecting the combinatorial class of the quadratic map~$f_{-2}$ transversally.
  That is, if there is a parameter~$\lambda_0$ in~$\Lambda_0$ such that
  $$ \whf_{\lambda_0}^2(0) = \beta(\whf_{\lambda_0}),
  \text{ and }
  \frac{\partial}{\partial \lambda} \left( \whf_\lambda^2(0) - \beta(\whf_\lambda) \right)|_{\lambda = \lambda_0} \neq 0. $$
  So, condition~1 of Definition~\ref{d:admissibility} is satisfied for an open set of holomorphic 1\nobreakdash-parameter families of quadratic-like maps.
  If in addition~$\chi_{\whf_{\lambda_0}}(p(\whf_{\lambda_0})) > \chi_{\whf_{\lambda_0}}(p^+(\whf_{\lambda_0}))$, then the inequality in~\eqref{eq:1} is also satisfied for an open set of holomorphic 1\nobreakdash-parameter families of quadratic-like maps.

  On the other hand, the equality in~\eqref{eq:1} imposes a restriction, but there is an open set of holomorphic 2\nobreakdash-parameter families of quadratic-like maps that have a holomorphic 1\nobreakdash-parameter subfamily satisfying this condition.
  Thus, the conclusions of the Main Theorem hold for an open set of holomorphic 2\nobreakdash-parameter families of quadratic-like maps.
\end{rema}

\begin{rema}[Sensitivity is compatible with convergence and non-convergence]
  \label{r:convergence or divergence}
  In the proof of the Sensitive dependence of Gibbs states in~\S\ref{ss:proof of sensitive dependence} we show that for any choice of~$\uvarsigma_0$ in~$\signs$, a uniform family~$\sF$ as in the Main Theorem has sensitive dependence of low-temperature geometric Gibbs states at~$f_{\uvarsigma_0}$.
  If we choose~$\uvarsigma_0$ that is not eventually constant, then the Main Theorem implies that the geometric Gibbs states of~$f_{\uvarsigma_0}$ do not converge as the temperature drops to zero.
  On the other hand, if~$\uvarsigma_0$ is eventually constant, then the geometric Gibbs states converge.
  This shows that in the Sensitive Dependence of Geometric Gibbs states the parameter~$\lambda_0$ can be chosen so that the geometric Gibbs states of~$\whf_{\lambda_0}$ converge as the temperature drops to zero, and that it can also be chosen so that they do not converge.
\end{rema}

\begin{rema}[Accumulation measures]
  \label{r:set of accumulation measures}
  Our estimates show that for every~$\uvarsigma$ in~$\signs$, and every~$t$ in ${(0, +\infty)}$ we have
  \begin{multline*}
    \rho_t^{\R}(\uvarsigma) \left( B \left( \cO^+(\fs) \cup \cO^-(\fs), \exp(- \upsilon_0 t^2) \right) \right)
    \ge
    1 - C_0 \exp(- \upsilon_0 t^2)
    \\
    \left( \text{resp. } \rho_{t}(\uvarsigma) \left( B \left( \cO^+(\fs) \cup \cO^-(\fs), \exp(- \upsilon_0 t^2) \right) \right)
      \ge
      1 - C_0 \exp(- \upsilon_0 t^2)
    \right),
  \end{multline*}
  see Remark~\ref{r:periodic concentration}.
  In particular, every accumulation measure of~$( \rho_t^{\R}(\uvarsigma) )_{t > 0}$ and of~$(\rho_t(\uvarsigma))_{t > 0}$ is supported on~$\cO^+(\fs) \cup \cO^-(\fs)$.
  Combined with the Main Theorem this implies that, if the sequence~$\uvarsigma$ is not eventually constant, then the set of accumulation measures of~$( \rho_t^{\R}(\uvarsigma) )_{t > 0}$ and that of~$(\rho_t(\uvarsigma))_{t > 0}$, are both equal to the segment joining the invariant probability measure supported on~$\cO^+(\fs)$ to the invariant probability measure supported on~$\cO^-(\fs)$.
\end{rema}

\begin{rema}[Speed of convergence to ground states]
  \label{r:convergence}
  If the sequence~$\uvarsigma$ is eventually constant, then the Main Theorem implies that, as the temperature drops to zero, the geometric Gibbs states of~$\fs$ converge super-exponentially to the periodic measure supported either on~$\cO^+(\fs)$, or~$\cO^-(\fs)$.
  In other situations the convergence is only exponential, as in the case of the shift map and a locally constant potential~\cite{Bre03}.
  For the shift map and a potential admitting a unique ground state, the exponential convergence can be derived from the large deviation principle in~\cite[\S3.1.3]{BarLepLop13}, using the fact that the rate function is finite on a dense set.
  The Main Theorem shows that this large deviation principle holds, and that the corresponding rate function is everywhere equal to~$+\infty$, except on~$\cO^+(\fs)$ or on~$\cO^-(\fs)$ (depending on the choice of~$\uvarsigma$) where it vanishes.
\end{rema}

\begin{rema}[Pressure at low temperatures]
  \label{r:pressure assymptotic}
  Our estimates show that there is a constant~$\gamma$ in~$(0, 1)$ such that for every~$\uvarsigma$ in~$\signs$, and every sufficiently large~$t > 0$ we have
  $$ P_{\fs}^{\R}(t)
  \sim
  P_{\fs}(t)
  \sim
  - t \frac{\chicritfs}{2} + \frac{\log 2}{3} \gamma^{\left( \frac{4}{A(\uvarsigma)} t \right)^3}, $$
  see~\eqref{e:pressure upper bound} and~\eqref{e:pressure lower bound} for precisions.
\end{rema}

\subsection{A concrete admissible family}
\label{ss:deformation of quadratic}
In this subsection we exhibit a concrete (real) 1\nobreakdash-parameter family of quadratic-like maps satisfying the hypotheses of the Main Theorem.
We use this family to prove the Sensitive Dependence of Geometric Gibbs States, in~\S\ref{ss:proof of sensitive dependence} below.

Recall that for every~$c$ in~$\C$, we denote by~$f_c$ the complex quadratic polynomial given by ${f_c(z) = z^2 + c}$.
For each parameter~$\lambda$ in~$\cP_3(-2)$, put~$\whp(\lambda) \= \whp(f_{\lambda})$ and define the polynomial
\begin{multline*}
P_\lambda(w)
\=
\left( w^2 - \beta(\lambda)^2 \right)
\prod_{i = 0}^2 \left[ \left( w - f_\lambda^i(p(\lambda)) \right) \left( w - f_\lambda^i(\wtp(\lambda)) \right)\right]^2
\\ \cdot
(w - \whp(\lambda)) \prod_{j = 1}^5 \left( w - f_\lambda^j(\whp(\lambda)) \right)^2.
\end{multline*}
Noting that~$DP_\lambda(\whp(\lambda)) \neq 0$, define
$$ \omega(\lambda)
\=
\frac{2}{\whp(\lambda)^2 DP_{\lambda}(\whp(\lambda))} \left( \frac{\left(Df_\lambda^3(\wtp(\lambda)) \right)^2}{Df_{\lambda}^6(\whp(\lambda))} - 1 \right), $$
and the polynomial
$$ \df(w) \= \lambda + w^2 + w^3 \omega(\lambda) P_\lambda(w). $$
Note that each of the coefficients of~$\df$ depends holomorphically on~$\lambda$ in~$\cP_3(-2)$, and that~$\df$ is real when~$\lambda$ is real.
Moreover, we have~$\omega(-2) = 0$, so~$\whf_{-2}$ coincides with the quadratic polynomial~$f_{-2}$.

By definition, for each~$\lambda$ in~$\cP_3(-2)$ the polynomial~$\df$ coincides with~$f_\lambda$ on ${\{\beta(\lambda), -\beta(\lambda)\}}$, and on the orbits of~$p(\lambda)$, $\wtp(\lambda)$, and~$\whp(\lambda)$.
Moreover, the derivative of~$\df$ coincides with that of~$f_\lambda$ at every point in the orbit of~$p(\lambda)$ and~$\wtp(\lambda)$, so
\begin{equation}
  \label{e:Lyapunov equality I}
\chi_{\df}(p(\lambda)) = \chi_{f_{\lambda}}(p(\lambda))
\quad \text{and} \quad
\chi_{\df}(\wtp(\lambda)) = \chi_{f_{\lambda}}(\wtp(\lambda)).
\end{equation}
On the other hand,
$$ D\df(\whp(\lambda))
=
2 \whp(\lambda) \frac{\left(Df_\lambda^3(\wtp(\lambda)) \right)^2}{Df_{\lambda}^6(\whp(\lambda))}, $$
and for each~$j$ in~$\{1, \ldots, 5 \}$ the derivative of~$\df$ coincides with that of~$f_\lambda$ at~$f_{\lambda}^j(\whp(\lambda))$.
Thus~$D\df^6(\whp(\lambda)) = \left(Df_\lambda^3(\wtp(\lambda)) \right)^2$ and
\begin{equation}
  \label{e:Lyapunov equality II}
\chi_{\df}(\whp(\lambda))
=
\chi_{f_{\lambda}}(\wtp(\lambda))
=
\chi_{\df}(\wtp(\lambda)).
\end{equation}
\begin{lemm}
  \label{l:deformation of quadratic}
 Let~$K_0 > 1$ be given.
 For each~$\lambda$ in~$\cP_3(-2)$, put~$U_\lambda \= \df^{-1}(B(0, 80))$.
Then there are~$r_\# > 0$  and a map $\chi: B(-2, r_{\#})\to \cX$ such that for every~$\lambda$ in~$B(-2, r_\#)$ the map~$\df \colon U_\lambda \to B(0, 80)$ is a normalized quadratic-like map, and the family
$$ \sF_0
\=
\{ \df \colon U_\lambda \to B(0, 80) \mid \lambda \in B(-2, r_{\#}) \} $$
is uniform with constants~$K_0$ and~$R = 80$, and with~$c(\df)=\chi(\lambda)$.
Moreover, there is~$\delta > 0$ such that~$\chi$ maps~$ [-2, -2 + r_\#)$ homeomorphically onto ${[-2, -2 + \delta)}$, and there is~$n_\# \ge 1$ such that for every integer~$n \ge n_\#$ there is a continuous map ${\sigma_n \colon \cK_n \to [-2, -2 + r_{\#})}$ such that~$c \mapsto \chi (\sigma_n(c))$ is the identity on~$\cK_n$.
In particular,
\begin{equation*}
  \{\df \in \sF_0 \mid \lambda \in  \sigma_n(\cK_n)\}
  \subseteq
  \cK_n(\sF_0),
\end{equation*}
the map ${s_n \colon \cK_n\to \cK_n(\sF_0)}$ given by ${s_n(c) \= \widehat{f}_{\sigma_n(c)}}$ is continuous, and the family~$\sF_0$ is admissible.
\end{lemm}
\begin{proof}
Since~$\omega(-2) = 0$, and~$\omega$ and~$P_\lambda$ are holomorphic
in~$\lambda$, we can choose~$r_1 > 0$ such that~$B(-2, r_1)$ is contained in~$\cP_3(-2)$, and such that for every~$\lambda$ in~$B(-2, r_1)$ the closure of the open set~$U_\lambda$ is contained in~$B(0, 80)$ and~$\df \colon U_\lambda \to B(0, 80)$ is a quadratic-like map.
For each~$r$ in $(0, r_1]$, consider the family of quadratic-like maps
$$ \sF(r)
\=
( \df \colon U_\lambda \to B(0, 80) )_{\lambda \in B(0, r)}. $$
Noting that for~$\lambda$ close to~$-2$ the set~$\partial U_\lambda$ is an analytic Jordan curve that is close to~$\partial U_{-2}$ in the $C^1$ topology, it follows that there is~$r_2$ in~$(0, r_1)$ such that the map ${B(0,r_2) \to  \sU}$ given by ${\lambda \mapsto U_\lambda}$ is continuous with respect to the Carath{\'e}odory topology on~$\sU$, and that the family of quadratic-like maps~$\sF(r_2)$ is analytic in the sense of~\cite[\S II, 1]{DouHub85a}.
Moreover, the considerations in~\cite[\S II]{DouHub85a} imply that for every~$r_3$ sufficiently small in~$(0, r_2)$ the family~$\sF(r_3)$ satisfies the following.
There is a map ${\chi \colon B(0, r_3) \to \cX}$ such that the family~$\sF(r_3)$ is uniform with constants~$K_0$ and~$R = 80$, and with the function ${c \colon \sF_0 \to \cX}$ given by ${c(f) \= \chi(f(0))}$.
Note that for every~$\lambda$ in~$B(0, r_3)$, we have
\begin{equation*}
  \df(0)
  =
  \lambda
  \text{ and }
  c(\df)
  =
  \chi(\lambda).
\end{equation*}
Moreover, for every real parameter~$\lambda$ in~$B(0, r_3)$ the conjugacy  $h_{\df}$  commutes with the complex conjugation, and therefore~$c(\df)$ is real.
If~$\lambda$ in~$B(-2, r_3)$ is real and satisfies~$\lambda > -2$, then we have~$\df(0) = \lambda > - \beta(\lambda)$.
Together with
\begin{equation*}
  \df(-\beta(\lambda))
  =
  \df(\beta(\lambda))
  =
  \beta(\lambda),
\end{equation*}
this implies that~$c(\df) > -2$.
By~\cite[Lemma~A.1]{CorRiv13} there is~$\varrho > 0$, such that for every~$\lambda$ in~$(-2, -2 + \varrho)$ we have by~\eqref{e:Lyapunov equality I}
$$ \chi_{\df}(p(\lambda)) = \chi_{f_{\lambda}}(p(\lambda)) > \chi_{f_\lambda}(\wtp(\lambda)) = \chi_{\df}(\wtp(\lambda)). $$
Thus, reducing~$r_3$ if necessary, the inequality in~\eqref{eq:5} is satisfied.

Note that~\cite[Proposition~17 and Theorem~4]{DouHub85a} implies that there is~$r_\#$ in $(0, r_3)$ such that~$\chi$ is locally injective at each point of~$B(-2, r_\#) \setminus \{ -2 \}$.
Thus, there is~$\delta > 0$ such that~$\chi$ maps~$[-2, -2 + r_\#)$ homeomorphically onto~$[-2, -2 + \delta)$.
Since by Proposition~\ref{p:ps} there is~$n_\# \ge n_1$ such that for every integer~$n
\ge n_{\#}$ the set~$\cK_n$ is contained in~$(-2, -2 + \delta)$, we conclude that for every integer ${n \ge n_\#}$ there is a continuous map ${\sigma_n \colon \cK_n \to [-2, -2 + r_{\#})}$ such that~$c \mapsto \chi(\sigma_n(c))$ is the identity on~$\cK_n$.
 In particular,
 \begin{equation*}
   \{\df \in \sF_0 \mid \lambda \in  \sigma_n(\cK_n)\}
   \subseteq
   \cK_n(\sF_0).
 \end{equation*}
 To finish the proof, note that the map ${[-2, -2 + r_{\#}) \to \sF_0}$ given by ${\lambda \mapsto \df}$ is continuous and thus, for every integer ${n \ge n_\#}$ the map ${s_n \colon \cK_n \to  \cK_n(\sF_0)}$ given by ${s_n(c) \= \whf_{\sigma_n(c)}}$  is continuous, and~$c \circ s_n$ is the identity on~$\cK_n$.
 This completes the proof that~$\sF_0$ is an admissible family, and of the lemma.
\end{proof}

\subsection{Proof of the Sensitive Dependence of Geometric Gibbs States assuming the Main Theorem}
\label{ss:proof of sensitive dependence}
Let~$K_0$ be the constant given by the Main Theorem with~$R = 80$, and let~$r_\# > 0$ and the family of quadratic-like maps~$\sF_0$ be given by Lemma~\ref{l:deformation of quadratic} for this choice of~$K_0$.
Putting ${\Lambda_0 \= B(-2, r_\#)}$, we have ${\sF_0 = ( \df )_{\lambda \in \Lambda_0}}$.
On the other hand, $\sF_0$ is uniform with constants~$K_0$ and~$80$, and admissible by Lemma~\ref{l:deformation of quadratic}. 
Fix a sufficiently large integer~$n$ for which the conclusions of the Main Theorem are satisfied with ${\sF = \sF_0}$, and let~$(\fs)_{\uvarsigma \in \signs}$ and~$A$ be as in the statement of the Main Theorem.
Given~$\uvarsigma$ in~$\signs$, denote by~$\lambda(\uvarsigma)$ the unique parameter in~$\Lambda_0$ such that~$\whf_{\lambda(\uvarsigma)} = \fs$.
By Lemma~\ref{l:deformation of quadratic}, the parameter~$\lambda(\uvarsigma)$ is real.
Then we prove the Sensitive Dependence of Geometric Gibbs States with~$\Lambda = \{ \lambda(\uvarsigma) \mid \uvarsigma \in \signs \}$.

Put
$$ \Asup \= \sup_{\uvarsigma \in \signs} A(\uvarsigma)
\quad \text{and} \quad
\Ainf \= \inf_{\uvarsigma \in \signs} A(\uvarsigma). $$
Let~$(\beta_\ell)_{\ell \in \N}$ be a sequence of inverse temperatures such that~$\beta_\ell \to +\infty$ as~$\ell \to +\infty$.
Replacing~$( \beta_\ell )_{\ell \in \N}$ by a subsequence if necessary, assume that~$\beta_1 \ge \Asup$, and that for every~$\ell$ in~$\N$ we have
\begin{equation}
  \label{e:temperature subsequence growth}
  \beta_{\ell + 1}
  \ge
  \Asup \left( \frac{\beta_{\ell}}{\Ainf} + 2 \right).  
\end{equation}
For each~$\ell$ in~$\N$ put~$m(\ell) \= \lfloor \beta_{\ell} / \Asup \rfloor$, and note that $m(1)\ge 1$ and that
$$ m(\ell + 1)
\ge
\frac{ \beta_{\ell + 1}}{\Asup} - 1
\ge
\frac{\beta_{\ell}}{\Ainf} + 1
\ge
m(\ell) + 1. $$

Fix a sequence~$\uvarsigma_0 = (\varsigma_0(m))_{m \in \N}$ in~$\signs$, let~$\lambda_0$ in~$\Lambda_0$ be such that~$\whf_{\lambda_0} = f_{\uvarsigma_0}$, and let~$\varepsilon > 0$ be given.
Then there is~$\ell_0 \ge 1$ such that for every~$\uvarsigma \= (\varsigma(m))_{m \in \N}$ in~$\signs$ such that for every~$m$ in~$[0, m(\ell_0) - 1]$ we have~$\uvarsigma(m) = \uvarsigma_0(m)$, the parameter~$\lambda$ in~$\Lambda_0$ such that~$\whf_{\lambda} = \fs$, satisfies~$|\lambda - \lambda_0| < \varepsilon$.
Let~$\uvarsigma$ be the unique such sequence, such that in addition for every even (resp. odd) integer~$\ell \ge \ell_0 + 1$, and every~$m$ in~$[m(\ell), m(\ell) + 1]$, we have~$\varsigma(m) = +$ (resp. $\varsigma(m) = -$).

For every integer~$\ell \ge \ell_0 + 1$, we have
$$ \beta_\ell \ge \Asup m(\ell) \ge A(\uvarsigma) m(\ell), $$
and by~\eqref{e:temperature subsequence growth}
$$ A(\uvarsigma) (m(\ell + 1) - 1)
\ge
\Ainf \left( \frac{\beta_{\ell + 1}}{\Asup} - 2 \right)
\ge
\beta_{\ell}. $$
This proves that~$\beta_{\ell}$ is in~$\left[ A(\uvarsigma) m(\ell), A(\uvarsigma) (m(\ell + 1) - 1) \right]$.
Since by definition of~$\uvarsigma$ for every~$\ell \ge \ell_0$ we have
$$ \varsigma(m(\ell)) = \cdots = \varsigma(m(\ell + 1) - 1), $$
and since~$\varsigma(m(\ell))$ alternates between~$+$ and~$-$ according to whether~$\ell$ is even or odd, the Main Theorem implies the desired assertion for the map~$\whf_{\lambda} = \fs$.
\section{The Geometric Peierls condition, and uniform estimates}
\label{s:uniform estimates}
In this section we introduce the Geometric Peierls condition, and give a criterion for maps in a uniform family to satisfy this condition with uniform constants.
We also make other uniform estimates that are used in the rest of the paper, which are mostly deduced from analogous estimates for quadratic maps in~\cite{CorRiv13}.

To state the Geometric Peierls condition, we introduce some notation.
For every normalized quadratic-like map~$f$, put
$$
\chicritf \= \liminf_{n\to +\infty} \frac{1}{n} \log |Df^n(f(0))|.
$$

Let~$\sF$ be a uniform family of quadratic-like maps, $n$ an integer satisfying ${n \ge 5}$, and~$f$ a map in~$\cK_n(\sF)$.
Put
$$ 
\pV \= \pP = f^{-1}(P_{f, n}( \tbeta(f))),
$$ 
and
$$
\pD'
\=
\{w\in \C \setminus \pV \mid f^{m}(w)\in \pV \text{ for some } m \in \N\}.
$$
For~$w$ in $\pD'$ denote by~$m_f(w)$ the least~$m$ in~$\N$ such that $f^{m}(w)\in \pV$, and
call it the \emph{first landing time of~$w$ to~$\pV$}.
The \emph{first landing map to~$\pV$} is the map~$\pL \colon \pD' \to \pV$ defined
by~$\pL(w) \= f^{m_f(w)}(w)$.

\begin{defi}[Geometric Peierls Condition]
  \label{d:geometric Peierls}
  Let~$\sF$ be a uniform family of quadratic-like maps, let~$n$ be an integer satisfying ${n \ge 5}$, and~$f$ a map in~$\cK_n(\sF)$.
  Given~$\kappa > 0$ and $\upsilon > 0$, a quadratic-like map~$f$ in~$\sF$ satisfies the \emph{Geometric Peierls Condition with constants~$\kappa$ and~$\upsilon$}, if for every~$z$ in~$L_f^{-1}(\pV)$ we have
  \begin{equation}
    \label{eq:2}
    |DL_f(z)| \ge \kappa \exp((\chicritf/2 + \upsilon) m_f(z)).
  \end{equation}
\end{defi}

\begin{rema}
  \label{r:geometric Peierls}
  The analogy between~\eqref{eq:2} and the usual Peierls conditions for contour models is as follows.
  We use the terminology in~\cite[II]{Sin82}.
  As usual, the one-point interaction energy corresponds to the geometric potential~$- \log |Df|$.
  The (orbit of the) critical point~$z = 0$ of~$f$ plays the r{\^o}le of the unique ground state.
  In contrast to the usual Peierls condition for contour models where the ground state is assumed to be supported on a periodic configuration, it is crucial for the Main Theorem to allow the orbit of~$0$ to be nonperiodic.
  However, we do require later that the Lyapunov exponent~$\lim_{n\to +\infty} \frac{1}{n} \log |Df^n(f(0))|$ exists, so that the ``$\liminf$'' that defines~$\chicritf$ is actually a limit.
  Consider an initial condition~$w$ near the critical point~$0$ of~$f$.
  Following the definition of the boundary of a configuration~\cite[II, Definition~2.2]{Sin82}, we see that for the ``boundary'' of~$w$ with respect to~$0$ to be finite, it is enough to assume that for some integer~$\tau \ge 1$ we have~$f^{\tau}(w) = f^{\tau}(0)$.
  The orbit of~$w$ shadows that of~$0$ up to a certain time~$\ell$, so that the derivatives of~$f^{\ell - 1}$ at~$f(w)$ and at~$f(0)$ are comparable.
  After time~$\ell$, the orbit of~$w$ can be significantly different from that of~$0$.
  To simplify, assume that the point~$z$ defined by ${z \= f^{\ell}(w)}$ satisfies ${L_f(z) = 0}$, so we have
  \begin{equation*}
    f^{\ell + m_f(z)}(w)
    =
    f^{m_f(z)}(z)
    =
    L_f(z)
    =
    0,
  \end{equation*}
  and therefore the boundary of~$w$ with respect to~$0$ is bounded from above by~$m_f(z)$.
  Up to a uniform distortion constant, the Hamiltonian at~$z$ relative to~$w$ is equal to
  \begin{equation*}
    - \log |Df^{\ell + m_f(z)}(w)| - \left( - \frac{\ell + m_f(z)}{2} \chicritf \right)
    \sim
    - \log |DL_f(z)| + \frac{m_f(z)}{2} \chicritf.
    \footnote{Here we replaced $\log |D f^{\ell + m_f(z)}(0)|$ ($= -\infty$) by~$(\ell + m_f(z)) \chicritf/2$, which is what appears naturally in several estimates, see for example Lemmas~\ref{l:first return to central derivative} and~\ref{l:critical line}.}
  \end{equation*}
  Thus, condition~\eqref{eq:2} becomes Peierls condition as in~\cite[II, Definition~2.3]{Sin82}.\footnote{To follow the analogy, we should require the constant~$\kappa$ to be larger than the implicit distortion constant in the computation above.
    Later on we compensate a possible small value of~$\kappa$ by assuming that the map~$f$ is in~$\cK_n(\sF)$ for a sufficiently large integer~$n$.}
\end{rema}

The following is a criterion for the Geometric Peierls Condition.
For future reference, it is stated in a slightly stronger form than what is needed for this paper.
\begin{prop}
  \label{p:transporting Peierls}
  For every~$\upsilon > 0$ satisfying~$\upsilon < \frac{1}{2} \log 2$ and every~$R >0$, there are constants~$K_1 > 1$, $n_1 \ge 6$, and~$\kappa_1 > 0$, such that the following property holds.
  If the family of quadratic-like maps~$\sF$ is uniform with constants~$K_1$ and~$R$, then for every integer~$n \ge n_1$, every element~$f$ of~$\cK_n(\sF)$ satisfies the Geometric Peierls Condition with constants~$\kappa_1$ and~$\upsilon$.
  Furthermore, we have
  \begin{displaymath}
    \chi_f(\beta(f)) > \chicritf + 2\upsilon,
    \chicritf > 2 \upsilon,
    \text{ and }
    \chi_f(p(f)) < \chi_f(p^+(f)) + \frac{\upsilon}{4}.
  \end{displaymath}
\end{prop}

Roughly speaking, to prove the geometric Peierls condition we show that ${\chicritf \sim \log 2}$ (Lemma~\ref{l:bound chicritf}), and that for every~$z$ in~$L_f^{-1}(V_f)$ we have ${L_f(z) \sim 2^{m_f(z)}}$.
Assuming that the dilatation constant of the uniform family is sufficiently close to~$1$, we derive these estimates from the analogous estimates for the quadratic map~$f_{c(f)}$ established in~\cite{CorRiv13}.
This is done with the help of the uniform geometric estimates established in~\S\ref{ss:uniform geometric estimates}.

After some uniform geometric estimates in~\S\ref{ss:uniform geometric estimates}, the proof of Proposition~\ref{p:transporting Peierls} is given in~\S\ref{ss:proving Peierls}.
In~\S\ref{ss:uniform estimates} we make various uniform estimates.

\subsection{Uniform geometric estimates}
\label{ss:uniform geometric estimates}
In this subsection we use Mori's theorem on the modulus of continuity of normalized quasi\nobreakdash-conformal maps, to obtain some preliminary estimates for quadratic-like maps in a given uniform family.

\begin{lemm}
  \label{l:uniform family}
  Given~$K > 1$ and~$R > 0$ there is a constant~$C_1 > 1$ such that for every uniform family~$\sF$ of quadratic-like maps with constants~$K$ and~$R$, the following property holds for every~$f$ in~$\sF$ and~$w$ in~$\hX_f$:
  \begin{equation}
    \label{e:derivative estimate II}
    C_1^{-1} |w|
    \le
    |f(w) - f(0)|^{1/2}
    \le
    C_1|w|.
  \end{equation}
  Moreover, if in addition~$w$ is in~$X_f$, then
  \begin{equation}
    \label{e:derivative estimate I}
    C_1^{-1}|w| \le |Df(w)|\le C_1|w|,
  \end{equation}
  \begin{equation}
    \label{e:derivative Mori}
    C_1^{-1} |Df_{c(f)}(h_f^{-1}(w))|^{K}
    \le
    |Df(w)|
    \le
    C_1 |Df_{c(f)}(h_f^{-1}(w))|^{\frac{1}{K}}.
  \end{equation}
\end{lemm}
The proof of this lemma is after the following one.
\begin{lemm}
  \label{l:Mori}
  For each~$R > 0$ there is a constant~$C_2 > 1$ such that the following property holds.
  Let~$K \ge 1$ be given, and let~$h$ be a $K$\nobreakdash-quasi\nobreakdash-conformal homeomorphism of~$\C$ that is holomorphic outside~$\cl{B(0, R)}$ and that is tangent to the identity at infinity.
  Then for every~$z$ and~$z'$ in~$B(0, 2R)$ we have
  $$ C_2^{-K} |z - z'|^K
  \le
  |h(z) - h(z')|
  \le
  C_2 |z - z'|^{\frac{1}{K}}. $$
\end{lemm}
\begin{proof}
  Replacing~$h$ by~$h - h(0)$ if necessary, assume~$h(0) = 0$.
  Put~$D \= h(B(0, 2R))$, let~$\varphi \colon D \to B(0, 2R)$ be a bi-holomorphic map fixing~$z = 0$, and note that~$\varphi \circ h|_{B(0, 2R)}$ is a $K$\nobreakdash-quasi\nobreakdash-conformal homeomorphism of~$B(0, 2R)$ fixing~$z = 0$.
  Thus, Mori's theorem implies that for every~$z$ and~$z'$ in~$B(0, 2R)$, we have
  \begin{multline}
    \label{e:Mori}
    (16^K 2R^{K - 1})^{-1} |z - z'|^K
    \le
    |\varphi \circ h(z) - \varphi \circ h(z')|
    \\ \le
    16 (2R)^{1 - \frac{1}{K}} |z - z'|^{1/K},
  \end{multline}
  see for example~\cite[p.~47]{Ahl66}.
  It remains to estimate the distortion of~$\varphi$ on~$D$.
  Note first that the holomorphic function ${g \colon B \left( 0, R^{-1} \right) \setminus \{ 0 \} \to \C}$ defined by ${g(\zeta) \= h \left( \zeta^{-1} \right)^{-1}}$ extends holomorphically to~$ \zeta = 0$, and that the extension, also denoted by~$g$, satisfies~$g(0) = 0$ and~$Dg(0) = 1$.
  By Koebe's $\frac{1}{4}$\nobreakdash-theorem and the version of the Koebe Distortion theorem in~\cite[Theorem~1.6]{CarGam93}, we have
  $$ B \left(0, (8R)^{-1} \right)
  \subset
  g \left( B \left(0, (2R)^{-1} \right) \right)
  \subset B \left(0, 2R^{-1} \right), $$
  and therefore,
  $$ B(0, R/2) \subset D \subset B(0, 8R). $$
  By Schwarz' Lemma and Koebe's $\frac{1}{4}$\nobreakdash-theorem we have~$\frac{1}{4} \le |D\varphi(0)| \le 4$.
  Next we show that~$\varphi$ has a univalent extension to~$\hD \= h(B(0, 4R))$.
  Note first that~$\varphi$ extends continuously to~$\cl{D}$, since~$\partial D$ is a Jordan curve; we denote this extension also by~$\varphi$.
  Consider the holomorphic involution~$\iota$ of~$A \= h \left(B(0, 4R) \setminus \cl{B(0, R)} \right)$ defined by~$\iota(w) \= h( 4R^2 / h^{-1}(w))$.
  Let~$\hvarphi \colon \hD \to \C$ be the function that coincides with~$\varphi$ on~$\cl{D}$, and that for~$z$ in~$\hD \setminus \cl{D}$ is given by~$\hvarphi(z) \= 4R^2 / \varphi(\iota(z))$.
  Then~$\hvarphi$ is homeomorphism from~$\hD$ to~$B(0, 4R)$, and by Schwarz reflection principle it is holomorphic.
  By the Koebe Distortion Theorem, there is a universal constant~$\Delta > 1$ independent of~$h$ such that for every distinct~$z$ and~$z'$ in~$B(0, 2R)$ we have
  \begin{displaymath}
    (4 \Delta)^{-1}
    \le
    \Delta^{-1} |D \varphi(0)|
    \le
    \frac{|\varphi\circ h(z) - \varphi \circ h(z')|}{|h(z) - h(z')|}
    \le
    \Delta |D \varphi(0)|
    \le
    4 \Delta.
  \end{displaymath}
  Together with~\eqref{e:Mori}, this proves the desired chain of inequalities with ${C_2 = 16 \cdot 2R \cdot 4\Delta}$.
\end{proof}
\begin{proof}[Proof of Lemma~\ref{l:uniform family}]
  Let~$C_2 > 0$ be the constant given by Lemma~\ref{l:Mori}.
  Then
  $$ \hD
  \=
  \sup_{f \in \sF} \diam(f(\hX_f))
  \le
  C_2 \left( \sup_{c \in \hcX} \diam(f_c(\hX_c)) \right)^{\frac{1}{K}}
  <
  +\infty. $$
  Observe that, since~$\hX_f \subset f(\hX_f)$ and since~$\hX_f$ contains~$0$, for every~$w$ in~$\hX_f$ we have~$|w| \le \hD$.

  On the other hand, since for~$c$ in~$\hcX$ the sets
  $$ \partial X_c \= \{ z \in \C \mid G_c(z) = 1 \}
  \quad \text{and} \quad
  \partial \hX_c = \{ z \in \C \mid G_c(z) = 2 \}$$
  are disjoint and depend continuously with~$c$, we have
  $$ r \= \inf_{c \in \hcX} \dist(\partial X_c, \partial \hX_c) > 0. $$
  In particular, for every~$c$ in~$\hcX$ the set~$\hX_c$ contains~$B(0, r)$.
  Combined with Lemma~\ref{l:Mori} this implies that, if we put~$\whr \= (r / C_2)^{K}$, then for every~$f$ in~$\sF$ and~$w$ in~$X_f$ the set~$\hX_f$ contains~$B(w, \whr)$.

  To prove~\eqref{e:derivative estimate II}, note that for every~$f$ in~$\sF$ and~$w$ in~$\hX_f$ we have
  \begin{equation}
    \label{e:reminder as diameter}
    |w^2 + w^3 R_f(w)| = |f(w) - f(0)| \le \diam(f(\hX_f)) \le \hD,  
  \end{equation}
  and therefore~$|w^3 R_f(w)| \le \hD + \hD^2$.
  So, if we put~$\tR \= (\hD + \hD^2) / \whr^3$, then the maximum principle implies~$|R_f| \le \tR$ on~$B(0, \whr)$.
  Letting~$\whr_0 \= \min \left\{ \whr, 1 / (2\tR) \right\}$, for every~$w$ in~$B(0, \whr_0)$ we have~$|w R_f(w)| \le 1/2$, and therefore
  $$ \frac{1}{2} \le \left| \frac{f(w) - f(0)}{w^2} \right|
  \le
  \frac{3}{2}. $$
  Let~$w$ in~$\hX_f \setminus B(0, \whr_0)$ be given and put~$z \= h_f^{-1}(w)$.
  Applying Lemma~\ref{l:Mori} twice and using $h_f(0)=0$ we obtain
  \begin{multline*}
    |f(w) - f(0)|
    =
    |h_f(f_{c(f)}(z)) - h_f(f_{c(f)}(0))|
    \ge
    C_2^{- K} |f_{c(f)}(z) - f_{c(f)}(0)|^{K}
    \\ =
    C_2^{- K} |z|^{2K}
    \ge
    C_2^{-K- 2K^2} |w|^{2K^2}
    \ge
    C_2^{-K- 2K^2} \whr_0^{2K^2}.
  \end{multline*}
  Together with~\eqref{e:reminder as diameter} these estimates imply~\eqref{e:derivative estimate II} with
  $$ C_1
  =
  \max \left\{ 2, C_2^{K+2K^2} \hD^2 \whr_0^{- 2K^2}, \hD \whr_0^{-2} \right\}^{\frac{1}{2}}. $$

  To prove~\eqref{e:derivative estimate I} and~\eqref{e:derivative Mori}, note first that by Schwarz' Lemma for every~$w$ in~$B(0, \whr/2)$ we have~$|DR_f(w)| \le 2 \tR / \whr$.
  Then, putting~$\whr_1 \= \min \left\{ \whr, 1 / (10 \tR) \right\}$, for every~$w$ in~$B(0, \whr_1)$ we have $|3 w R_f(w) + w^2 DR_f(w)| \le 1/2$, so
  $$ \frac{3}{2} \le \left| \frac{Df(w)}{w} \right|
  \le
  \frac{5}{2}. $$
  Let~$w$ in~$X_c \setminus B(0, \whr_1)$ be given.
  Using that~$B(w, \whr)$ is contained in~$\hX_f$ and the definition of~$\hD$, Schwarz' lemma implies~$|Df(w)| / |w| \le (\hD / \whr) / \whr_1$.
  To estimate~$|Df(w)| / |w|$ from below, put
  $$ z \= h_f^{-1}(w),
  r_1 \= \min \{ C_2^{-K} \whr_1^{K}, r \},
  \quad \text{and} \quad
  B(z) \= B(z, r_1). $$
  By Lemma~\ref{l:Mori} we have~$r_1 \le |z|$, so~$f_{c(f)}$ is injective on~$B(z)$.
  By Koebe's $\frac{1}{4}$-theorem, the set~$f_{c(f)}(B(z))$ contains
  $$ B(f_{c(f)}(z), |Df_{c(f)}(z)| r_1/4), $$
  and therefore~$B(f_{c(f)}(z), r_1^2/2)$.
  By Lemma~\ref{l:Mori} we have
  $$ h_f(B(z)) \subset B \left( w, C_2 r_1^{\frac{1}{K}} \right) $$
  and
  $$ \hB(w)
  \=
  B(f(w), C_2^{- K} (r_1^2 / 2)^{K})
  \subset
  h_f(f_{c(f)}(B(z))). $$
  Thus, if we put~$\varepsilon \= C_2^{-1 - K} 2^{- K} r_1^{2K - \frac{1}{K}}$, then Schwarz' lemma applied to~$f^{-1}|_{\hB(w)}$ implies
  $$ |Df(w)| \ge \varepsilon
  \quad \text{and} \quad
  \frac{|Df(w)|}{|w|}
  \ge
  \frac{\varepsilon}{\hD}. $$
  This completes the proof of~\eqref{e:derivative estimate I} with~$C_1$ equal to~$\tC \= \max \{ 3, \hD / (\whr \whr_1), \hD / \varepsilon \}$.
  Combined with Lemma~\ref{l:Mori} this last estimate implies,
  \begin{multline}
    \frac{\tC^{-1}}{(2C_2)^{K}}|Df_{c(f)}(z)|^{K}
    =
    \frac{\tC^{-1}}{C_2^{K}}|z|^{K}
    \le
    \tC^{-1} |h_f(z)|
    =
    \tC^{-1}|w|
    \\ \le
    |Df(w)|
    \le
    \tC|w|
    =
    \tC|h_f(z)|
    \le
    \tC C_2 |z|^{\frac{1}{{K}}}
    =
    \frac{\tC C_2}{2^{1/{K}}} |Df_{c(f)}(z)|^{\frac{1}{{K}}}.
  \end{multline}
  This proves~\eqref{e:derivative Mori} with~$C_1 = \tC (2C_2)^{K}$, and completes the proof of the lemma.
\end{proof}

\subsection{Proving the Geometric Peierls Condition}
\label{ss:proving Peierls}
In this subsection we prove Proposition~\ref{p:transporting Peierls}.
The proof is given after a couple of lemmas.

\begin{lemm}
  \label{l:qc transport}
  Given~$K > 1$ and~$R > 0$ there is a constant~$C_3 > 1$ such that for every uniform family of quadratic-like maps~$\sF$ with constants~$K$ and~$R$, every integer~$n \ge 6$, and every~$f$ in~$\cK_n(\sF)$, the following property holds.
  Let~$w$ be a point in~$X_f$ and~$m \ge 1$ an integer such that~$f^m$ maps a neighborhood of~$w$ biholomorphically onto~$P_{f, 1}(0)$.
  Then
  $$ C_3^{-1} |Df_{c(f)}^m(h_f^{-1}(w))|^{\frac{1}{K}}
  \le
  |Df^m(w)|
  \le
  C_3 |Df_{c(f)}^m(h_f^{-1}(w))|^K. $$
\end{lemm}
\begin{proof}
  Let~$C_2$ be the constant given by Lemma~\ref{l:Mori}.
  From the proof of \cite[Lemma~5.4]{CorRiv13}, we have
  $$ 
  \Xi
  \=
  \inf_{c \in \cP_4(-2)} \diam(P_{c, 1}(0)) > 0.
  $$
  We prove the lemma with $C_3 = C_2^{1+K} \Xi^{\tfrac{1}{K}-K}$.
  Let $w$ be as in the statement of the lemma.
  By hypothesis, there is a neighborhood~$W$ of~$w$ that is mapped biholomorphically onto $P_{f,1}(0)$ by $f^m$.
  Take arbitrary points~$u$ and~$v$ in~$W$, and put
  $$ z \= h_f^{-1}(w),
  x \= h_f^{-1}(u),
  \text{ and }
  y \= h_f^{-1}(v). $$
  By Lemma~\ref{l:Mori}, we have
  \begin{displaymath}
    \begin{split}
      \frac{|f^m(u) - f^m(v)|}{|u-v|} 
      & = 
      \frac{|f^m(h_f(x)) - f^m(h_f(y))|}{|h_f(x) - h_f(y)|}
      \\ & \le 
      C_2^{1+K} \frac{|f_{c(f)}^m(x) - f_{c(f)}^m(y)|^{\tfrac{1}{K}}}{|x-y|^{K}}
      \\ & = 
      C_2^{1+K} |f_{c(f)}^m(x) - f_{c(f)}^m(y)|^{\tfrac{1}{K}-K} \frac{|f_{c(f)}^m(x) - f_{c(f)}^m(y)|^{K}}{|x-y|^{K}}
      \\ & \le  
      C_2^{1+K} \diam(P_{c(f),1}(0))^{\tfrac{1}{K}-K}
      \frac{|f_{c(f)}^m(x) - f_{c(f)}^m(y)|^{K}}{|x-y|^{K}}.
    \end{split}
  \end{displaymath}
  Since~$u$ and~$v$ are arbitrary points of~$W$, we conclude that
  \begin{multline*}
    |Df^m(w)| \le    C_2^{1+K} \diam(P_{c(f),1}(0))^{\tfrac{1}{K}-K}
    |Df_{c(f)}^m(z)|^{K} \\
    \le  C_2^{1+K} \Xi^{\tfrac{1}{K}-K}
    |Df_{c(f)}^m(z)|^{K}.
  \end{multline*}
  The proof of the other inequality follows similar arguments.
\end{proof}

\begin{lemm}
  \label{l:bound chicritf}
  For every~$K > 1$, $R > 0$, and~$\varepsilon > 0$ there is an integer $n_2\ge 5$ such that for every uniform family of quadratic-like maps~$\sF$ with constants~$K$ and~$R$, the following property holds.
  For every integer $n\ge n_2$, and every~$f$ in~$\cK_n(\sF)$, we have
  \begin{equation}
    \label{e:transported critical lyapunov}
    K^{-1} (1 - \varepsilon) \log 2
    \le
    \chicritf
    \le
    K (1 + \varepsilon) \log 2,
  \end{equation}
  and for every periodic point~$p$ of~$f$ in~$h_f(\Lambda_{c(f)})$, we have
  \begin{equation}
    \label{e:transported periodic lyapunov}
    K^{-1} (1 - \varepsilon) \log 2
    \le
    \chi_f(p)
    \le
    K (1 + \varepsilon) \log 2.
  \end{equation}
\end{lemm}
\begin{proof}
  Let~$C_3$ be the constant given by Lemma~\ref{l:qc transport}.
  
  Combining~\cite[Lemma~4.2]{CorRiv13} and~\cite[Lemma~5.3]{CorRiv13} with~$m_1 = 4$, we conclude that there are constants~$\hC_0 > 0$ and~$n_0 \ge 3$ such that for each integer~$n \ge n_0$ and each parameter~$c$ in~$\cK_n$, we have for every~$z$ in~$\Lambda_c$ and every integer~$m \ge 1$,
  $$ \hC_0^{-1} 2^{(1 - \varepsilon) m}
  \le
  |Df_c^m(z)| \le \hC_0 2^{(1 + \varepsilon) m}. $$
  Note that~$f_c^n(c)$ is in~$\Lambda_c$, so we can take~$z = f_c^n(c)$ above.
  Noting that for every~$f$ in~$\cK_n(\sF)$ and every integer~$k \ge 1$ the map~$f^{3k}$ maps a neighborhood of~$h_f(f_{c(f)}^{n+1}(0)) = f^{n + 1}(0)$ biholomorphically onto~$P_{f, 1}(0)$ (\emph{cf.}, \cite[Lemma~5.1]{CorRiv13}), by Lemma~\ref{l:qc transport} for every integer~$m \ge 1$ we have
  $$ C_3^{-1} \hC_0^{\frac{1}{K}} 2^{K(1 + \varepsilon) m}
  \le
  |Df^m(f^{n + 1}(0))|
  \le
  C_3 \hC_0^{K} 2^{K(1 + \varepsilon) m}, $$
  and
  $$ C_3^{-1} \hC_0^{\frac{1}{K}} 2^{K(1 + \varepsilon) m}
  \le
  |Df^m(p)|
  \le
  C_3 \hC_0^{K} 2^{K(1 + \varepsilon) m}. $$
  Taking logarithms, dividing by $m$, and letting $m\to +\infty$, we conclude the proof of the lemma.
\end{proof}

\begin{proof}[Proof of Proposition~\ref{p:transporting Peierls}]
  Let~$\varepsilon > 0$ be sufficiently small so that
  \begin{equation*}
    \varepsilon
    <
    \frac{2}{3} \left(\frac{1}{2} - \frac{\upsilon}{\log 2} \right)
    \text{ and }
    \varepsilon
    <
    \frac{\upsilon}{8 \log 2},
  \end{equation*}
  and let~$K_1 > 1$ be sufficiently close to~1 so that
  $$ \upsilon'
  \=
  \left( \frac{1 - \varepsilon}{K_1} - \frac{K_1 (1 + \varepsilon)}{2} \right) \log 2
  >
  \upsilon
  >
  4 \left( K_1(1 + \varepsilon) - \frac{1 - \varepsilon}{K_1} \right) \log 2. $$
  Let $n_2$ be given by Lemma~\ref{l:bound chicritf} for this value of $\varepsilon$.
  In view of Proposition~\ref{p:ps} and of the formula~$D f_{-2}(\beta(-2)) = 4$, we can take~$n_2$ larger if necessary so that for every integer~$n \ge n_2$ and every parameter~$c$ in~$\cK_n$ we have~$\chi_f(\beta(f)) \ge (1 - \varepsilon) \log 4$.
  Assume~$\sF$ is uniform with constants~$R$ and~$K_1$, and let~$C_3$ be the constant given by Lemma~\ref{l:qc transport} with~$K = K_1$.
  Note that~$C_3$ depends on~$R$ and~$\upsilon$ only.

  By Lemma~\ref{l:bound chicritf}, for every integer $n\ge n_2$ and every $f$ in $\cK_n(\sF)$ we have~\eqref{e:transported critical lyapunov} with~$K$ replaced by~$K_1$.
  On the other hand, by~\cite[Proposition~B]{CorRiv13} there are~$\hkappa_1 > 0$ and~$\whn_1 \ge 4$, such that for every integer~$n \ge \whn_1$, every parameter~$c$ in~$\cK_n$, and every~$z$ in~$L_c^{-1}(V_c)$, we have
  \begin{equation*}
    |DL_c(z)| \ge \hkappa_1 2^{(1 - \varepsilon) m_c(z)}.
  \end{equation*}
  Noting that for each~$f$ in~$\cK_n(\sF)$ and each~$w$ in~$L_f^{-1}(\pV)$ the map~$f^{m_f(z)}$ maps a neighborhood of~$w$ biholomorphically onto~$P_{f, 1}(0)$, by Lemma~\ref{l:qc transport} we have
  $$ |DL_f(w)|
  \ge
  C_3 \hkappa_1^{K_1} 2^{\frac{1 - \varepsilon}{K_1} m_f(z)}. $$
  Noting that by definition of~$\upsilon'$ we have
  $$ 2^{\frac{1 - \varepsilon}{K_1}}
  =
  2^{\frac{K_1(1 + \varepsilon)}{2}} \exp(\upsilon')
  \ge
  \exp(\chicritf/2 + \upsilon'), $$
  inequality~\eqref{e:transported critical lyapunov} implies the first assertion of the proposition with 
  $$
  n_1 = \min \{n_2, \whn_1 \} \text{ and } \kappa_1 = C_3 \hkappa_1^{K_1}.
  $$

  It remains to prove the inequalities.
  The second inequality follows from~\eqref{e:transported critical lyapunov}, and the definition of~$\upsilon'$.
  To prove the third inequality, note that by~\eqref{e:transported critical lyapunov} and the definition of~$\upsilon'$, we have
  $$ \chi_f(p(f)) - \chi_f(p^+(f))
  \le
  \left( K_1 (1 + \varepsilon) - \frac{1}{K_1} (1 - \varepsilon) \right) \log 2
  <
  \frac{\upsilon}{4}. $$
  To prove the first inequality, note that by Lemma~\ref{l:qc transport} and our choice of~$n_2$, we have
  $$ \chi_f(\beta(f))
  \ge
  \frac{1}{K_1} \chi_{f_{c(f)}}(\beta(c(f)))
  \ge
  \frac{1 - \varepsilon}{K_1} \log 4. $$
  Combined with~\eqref{e:transported critical lyapunov} and the definition of~$\upsilon'$, this implies
  $$ \chi_f(\beta(f)) - \chicritf
  \ge
  \left( \frac{1 - \varepsilon}{K_1} - \frac{K_1 (1 + \varepsilon)}{2} \right) \log 4
  =
  2 \upsilon'
  >
  2 \upsilon. $$
  This concludes the proof of the proposition.
\end{proof}

\subsection{Uniform estimates}
\label{ss:uniform estimates}
In this subsection we prove various uniform estimates.
Throughout this subsection we fix a uniform family of quadratic-like maps~$\sF$, with constants~$K$ and~$R$.

\begin{lemm}
  \label{l:landing to central derivative}
  There is a constant~$\Delta_1 > 1$ that only depends on~$K$ and~$R$, such that for each $f$ in $\cP_2(\sF)$ the following properties hold for each integer~$k \ge 2$: 
  For each point~$y$ in~$P_{f, k}(\tbeta(f))$ or in~$P_{f, k}(\beta(f))$ we have
  $$ \Delta_1^{-1} |Df(\beta(f))|^k \le |Df^k(y)|
  \le
  \Delta_1 |Df(\beta(f))|^k. $$
\end{lemm}
\begin{proof}
  The proof follows the same lines that \cite[Lemma~3.6]{CorRiv13}, and we only need to check that some constants are finite and others are positive.
  Let $C_1$ be the constant given by Lemma~\ref{l:uniform family}, and let $C_2$ be that given by Lemma~\ref{l:Mori}.
  By the proof of  \cite[Lemma~3.6]{CorRiv13}  we have
  \begin{align*}
    \Xi_1
    & \=
      \sup_{c \in \cP_0(-2)} \sup_{z \in P_{c, 1}(\beta(c))} |Df_c(z)|
      < 
      +\infty,
    \\
    \Xi_2
    & \=
      \inf_{c \in \cP_2(-2)} \inf_{z \in P_{c, 1}(\beta(c))} |Df_c(z)|
      >
      0,
    \intertext{and}
    \Xi_3
    & \=
      \inf_{c \in \cP_2(-2)} \modulus (P_{c, 0}(\beta(c)) \setminus \cl{P_{c, 1}(\beta(c))})
      >
      0.  
  \end{align*}
  By Lemmas~\ref{l:uniform family} and~\ref{l:Mori},
  we have
  \begin{align*}
    \hXi_1
    & \=
      \sup_{f\in \cP_0(\sF) }\sup_{w \in P_{f,1}(\beta(f))} |Df(w)| \le C_1 \Xi_1^{\frac{1}{K}}
      < 
      +\infty,
    \\
    \hXi_2
    & \=
      \inf_{f\in \cP_2(\sF) }\inf_{w \in P_{f,1}(\beta(f))} |Df(w)| \ge C_1^{-1}\Xi_2^K
      >
      0,
  \end{align*}
  and since for every $f$ in $\sF$ the conjugacy $h_f$ is $K$\nobreakdash-quasi\nobreakdash-conformal
  $$ \hXi_3
  \=
  \inf_{f\in \cP_2(\sF)} \modulus (P_{f, 0}(\beta(f)) \setminus \cl{P_{f,1}(\beta(f))}) \ge \frac{1}{K}\Xi_3
  >
  0. $$
  Let~$\Delta > 1$ be the constant given by Koebe Distortion Theorem with~$A = \hXi_3$.
  The desired inequalities with $\Delta_1 = \Delta \hXi_1 \hXi_2^{-1}$ follow from the fact that $f^{k-1}$ maps each of the sets $P_{f,k}(\beta(f))$ and $P_{f,k}(\tbeta(f))$ biholomorphically to $P_{f,1}(\beta(f))$.
\end{proof}

For a parameter~$c$ in~$\cP_2(-2)$ the external rays~$R_c(7/24)$ and~$R_c(17/24)$ land at the point~$\gamma(c)$ in~$P_{c,1}(0)$, see~\cite[Section~3.3]{CorRiv13}.
Let~$\hU_c$ be the open disk containing~$-\beta(c)$ that is bounded by the
equipotential~2 and by
\begin{equation*}
  \label{eq:pleasant cut}
  R_c(7/24) \cup \{ \gamma(c) \} \cup R_c(17/24).  
\end{equation*}
Put $\hW_c \= f_c^{-1}(\hU_c)$, and for every $n\ge 3$ and every $f$ in $\cK_n(\sF)$ put $\hW_f \= h_f(\hW_{c(f)})$.

\begin{lemm}[Uniform distortion bound]
  \label{l:distortion to central 0}
  There is a constant~$\Delta_2 > 1$ that only depends on~$K$ and~$R$, such that for each integer~$n \ge 4$, and each~$f$ in~$\cK_n(\sF)$ the following properties hold:
  For each integer~$m \ge 1$ and each connected component~$W$ of~$f^{-m}(P_{f,1}(0))$ on which~$f^m$ is univalent, $f^m$ maps a neighborhood of~$W$ biholomorphically to~$\hW_f$ and the distortion of this map on~$W$ is bounded by~$\Delta_2$.
\end{lemm}
\proof
We follow the proof of \cite[Lemma~4.3]{CorRiv13}.
From that proof we have that for each parameter~$c$ in~$\cP_4(-2)$ the set~$\hW_c$ contains the
closure of~$P_{c, 1}(0)$ and 
$$ \tA
\=
\inf_{c \in \cP_4(-2)} \modulus (\hW_c \setminus \cl{P_{c, 1}(0)})
>
0. $$
Since for every $f$ in $\sF$ the conjugacy $h_f$ is $K$\nobreakdash-quasi\nobreakdash-conformal, we have
$$ \hA
\=
\inf_{f \in \cP_4(\sF)} \modulus (\hW_f \setminus \cl{P_{f,1}(0)})
\ge \frac{\tA}{K} >
0. $$
By \cite[Lemma~4.2]{CorRiv13}, $f_c^m$ maps a neighborhood of $h_f^{-1}(W)$ biholomorphically to $\hW_{c(f)}$. By conjugacy,   $f^m$ maps a neighborhood of $W$ biholomorphically to $\hW_f$. 
The conclusion follows from Koebe Distortion Theorem with $A=\hA$.
\endproof

\begin{lemm}
  \label{l:first return to central derivative}
  There is a constant~$C_4 > 1$ that only depends on~$K$ and~$R$, such that for each integer~$n \ge 4$ and 
  each~$f$ in~$\cK_n(\sF)$, the following properties hold for each integer~$q \ge1$:
  For each open set~$W$ that is mapped biholomorphically to~$P_{f,1}(0)$ 
  by~$f^{q}$, and each~$x$ in~$W$, we have
  $$ |Df(x)| \ge C_4^{-1} |Df^{q - 1} (f(x))|^{-\frac{1}{2}}. $$
\end{lemm}

\begin{proof}
  We follow the proof of \cite[Lemma~5.4]{CorRiv13}.  Let $C_1$ be the constant given by Lemma~\ref{l:uniform family}, and let $C_2$ be that given by Lemma~\ref{l:Mori}.
  Let~$\Delta_1 > 1$ and~$\Delta_2 > 1$ be the constants given by Lemmas~\ref{l:landing to central derivative} and~\ref{l:distortion to central 0}, respectively.
  From the proof of  \cite[Lemma~5.4]{CorRiv13}, we have
  $$ \Xi_1
  \=
  \inf_{c \in \cP_4(-2)} \diam(P_{c, 1}(0))
  >
  0
  \quad \text{and} \quad
  \Xi_2
  \=
  \sup_{c \in \cP_4(-2)} |Df_c(\beta(c))|
  < 
  +\infty,$$
  and  that for each~$c$ in~$\cP_3(-2)$ the closure of~$P_{c, 1}(0)$ is contained in~$\hW_c$ and
  $$ \Xi_3
  \=
  \inf_{c \in \cP_{4}(-2)} \modulus (\hW_c \setminus \cl{P_{c, 1}(0)})
  >
  0. $$
  Since for every $f$ in $\sF$ the conjugacy $h_f$ is $K$\nobreakdash-quasi\nobreakdash-conformal, we have
  $$ \hXi_3
  \=
  \inf_{f \in \cP_4(\sF)} \modulus (\hW_f \setminus \cl{P_{f,1}(0)})
  \ge \frac{\Xi_3}{K} >
  0, $$
  and by Lemma~\ref{l:Mori} and inequality~\eqref{e:derivative Mori},
  we have 
  $$ \hXi_1
  \=
  \inf_{f \in \cP_4(\sF)} \diam(P_{f, 1}(0))
  \ge C_2^{-K}\Xi_1^{K}
  >
  0$$
  and 
  $$
  \hXi_2
  \=
  \sup_{f \in \cP_4(\sF)} |Df(\beta(f))|
  \le C_1 \Xi_2^{\frac{1}{K}}< 
  +\infty. $$

  Let~$n \ge 4$ be a integer and~$f$ in~$\cK_n(\sF)$.
  Note that~$f^{q}$ maps a neighborhood~$\tW$ of~$W$ biholomorphically to~$\hW_f$ (Lemma~\ref{l:distortion to central 0}).
  So, if we put~$\tW' \= f(\tW)$, then~$f(0)$ is not in~$\tW'$ and~$f^{q - 1}$ 
  maps~$\tW'$ biholomorphically to~$\hW_f$; in particular we have 
  $$ \modulus (\tW' \setminus \cl{f(W)})
  =
  \modulus (\hW_f \setminus \cl{P_{f,1}(0)})
  \ge
  \hXi_3. $$
  Thus there is a constant~$A_1 > 0$ independent of~$n$, $f$ and~$q$ such that for every~$x$ in~$W$, we have 
  \begin{equation*}
    |f(x) - f(0)|
    \ge
    \dist(f(W), f(0))
    \ge
    \dist(f(W), \partial \tW')
    \ge
    A_1 \diam(f(W))
  \end{equation*}
  (\emph{cf}., \cite[Teichm{\"u}ller's module theorem, II, \S1.3]{LehVir73}).
  Thus, if we put~$A_2 \= C_1^{-2}(A_1 \Delta_2^{-1} \hXi_1)^{1/2}$, then by Lemmas~\ref{l:uniform family} and~\ref{l:distortion to central 0} with~$m = q - 1$ and with~$W$ replaced by~$f(W)$, we have
  $$ |Df(x)|
  \ge
  C_1^{-2} A_1^{1/2} \diam(f(W))^{1/2}
  \ge
  A_2 |Df^{q - 1} (f(x))|^{- 1/2}. $$
  This proves the lemma with constant~$C_4 = A_2^{-1}$.
\end{proof}

\begin{lemm}\label{l:contractions}
  There are constants $C_5 > 0$ and $\upsilon_1 > 0$ that only depend on~$K$ and~$R$, such that for every~$f$ in~$\cP_5(\sF)$, every~$\ell$ in~$\N$, and every connected
  component~$W$ of $g_f^{-\ell}(P_{f,1}(0))$, we have
  $$
  \max\{\diam(W), \diam(f(W)),\diam(f^2(W))\} \le C_5\exp(-\upsilon_1\ell).
  $$
\end{lemm}
\begin{proof}
  Let $K\ge 1$ and $R>0$ be the constants of the family $\sF$.  For this $R$, let $C_2$  be the constant of Lemma~\ref{l:Mori}.
  By \cite[Lemma~2.4]{CorRiv15b} there are constants $C'_0 > 0$ and $\upsilon'_0 > 0$ such that for every~$c$ in~$\cP_5(-2)$, every~$\ell$ in~$\N$, and every connected
  component~$W'$ of $g_{c}^{-\ell}(P_{c,1}(0))$, we have
  $$
  \diam(W') \le C'_0\exp(-\upsilon'_0\ell).
  $$
  Put
  $$
  \Xi_0 \= \inf_{c\in\cP_5(-2)}\dist(Y_c\cup \tY_c, \partial P_{c,1}(0)) > 0,
  $$
  $$
  \Xi_1 \= \sup_{c\in\cP_0(-2)}\diam(\{z\in \C \mid G_{c}(z)\le 2\}) < +\infty,
  $$
  and  $\hC_0\=\max\{C'_0,\Xi_0^{-1}\Xi_1C'_0\}$.
  Fix~$c$ in $\cP_5(-2)$, $\ell$ in~$\N$, and a connected
  component~$W'$ of $g_{c}^{-\ell}(P_{c,1}(0))$. 
  For every $w$ in $Y_c\cup \tY_c$ define the holomorphic maps
  $$
  z\mapsto  \frac{f_c(z)-f_c(w)}{z-w} \text{ and }  z\mapsto  \frac{f_c^2(z)-f_c^2(w)}{z-w}
  $$
  on $X_c$. Notice that for $z$ in  $\partial P_{c,1}(0)$ both maps are bounded from above by $\Xi_0^{-1}\Xi_1$.
  By the maximun principle for every~$z$ and~$w$ in $Y_c\cup \tY_c$ we have
  $$
  \max\{|f_c(z)-f_c(w)|, |f_c^2(z)-f_c^2(w)|\}
  \le
  \Xi_0^{-1}\Xi_1 |z-w|.
  $$
  In particular, 
  $$\max\{\diam(W'),  \diam(f_c(W')),\diam(f_c^2(W'))\}
  \le
  \hC_0\exp(-\upsilon'_0\ell).$$
  From Lemma~\ref{l:Mori} by putting $C_5\=C_2\hC_0^{\tfrac{1}{K}}$ and $\upsilon_1\=\upsilon'_0/K$, we conclude the proof of the lemma.
\end{proof}

\section{Estimating the geometric pressure function}
\label{s:estimating pressure}
The aim of this section is to prove Proposition~\ref{p:improved MS criterion}, stated at the beginning of~\S\ref{ss:Bowen type formula}.
For a uniform family, this proposition allows us to control the geometric pressure function with the itinerary of the critical point.
We achieve this aim using an inducing scheme and by adapting a similar result for quadratic maps~\cite[Proposition~D]{CorRiv13}.
The general scheme of this adaptation is the following.
The arguments in the original proof can be grouped into~3 types.
Purely combinatorial arguments depending only on the combinatorics of the Yoccoz puzzle.
Geometric estimates of the sizes of the puzzle pieces.
An estimate of the derivative of the first landing map to a neighborhood of the critical point and an estimate of the Lyapunov exponent of the critical value.
In the adaptation, the combinatorial arguments follow directly from the conjugacy, and the geometric estimates follow from the H\"older continuity of the conjugacy (\emph{cf}., Lemma~\ref{l:Mori}).
The third type of arguments use the Geometric Peierls condition in a crucial way.
This condition is introduced in Definition~\ref{d:geometric Peierls}, and it is used in Propositions~\ref{p:improved MS criterion}, \ref{p:Bowen type formula} and~\ref{p:estimating Z_1 by the postcritical series}, and in Lemmas~\ref{l:landing contribution} and~\ref{l:Poincare series}.

In~\S\ref{ss:Inducing scheme} we introduce an inducing scheme, and we prove a result on the existence of conformal measures and equilibrium states (Proposition~\ref{p:equilibria}) that is analogous to general results in~\cite{PrzRiv11}.
In~\S\ref{ss:Bowen type formula} we state Proposition~\ref{p:improved MS criterion}, and prove a Bowen type formula, and other general properties of the geometric pressure function.
Finally, the proof of Proposition~\ref{p:improved MS criterion} is given in~\S\ref{ss:proof of MS criterion}.

Throughout this section we fix a uniform family of quadratic-like maps~$\sF$, with constants~$K$ and~$R$.

\subsection{Inducing scheme}
\label{ss:Inducing scheme}
In this subsection we introduce the inducing scheme to estimate the geometric pressure function for maps in~$\cK_n(\sF)$.

Let~$n$ be an integer satisfying ${n \ge 5}$ and let~$f$ be in~$\cK_n(\sF)$.
Recall that ${\pV = \pP}$, and put 
$$
\pD\=\{z\in \pV  \mid  f^m(z) \in \pV  \text{ for some } m\ge 1\}.
$$
For~$w$ in~$\pD$ put~$m_f(w)\=\min\{m\in \N \mid f^m(w) \in \pV\}$, and call it the \emph{first return time of~$w$ to~$\pV$}.
The \emph{first return map to~$\pV$} is defined by
$$ \begin{array}{rcl}
     F_f \colon \pD & \to & \pV \\
     w & \mapsto & F_f(w) \= f^{m_f(w)}(w).
   \end{array} $$
   It is easy to see that~$\pD$ is a disjoint union of puzzle pieces; 
   so each connected component of~$\pD$ is a puzzle piece.
   Note furthermore that in each of these puzzle pieces~$W$, 
   the return time function~$m_f$ is constant; denote the common value
   of~$m_f$ on~$W$ by~$m_f(W)$.

   Throughout the rest of this subsection we put $\hV_f \= P_{f,4}(0)$. The proof of the following lemma is the same as for  \cite[Lemma~6.1]{CorRiv13}. 
   The reason  is that the combinatorics  and Koebe space are preserved by the conjugacy.

   \begin{lemm}[Uniform distortion bound]
     \label{l:bounded distortion to nice}
     There is a constant~$\Delta_3 > 1$ that only depends on~$K$ and~$R$, such that for each integer~$n \ge 5$, and each~$f$ in~$\cK_n(\sF)$ the following property holds: For every connected
     component~$W$ of~$\pD$ the map~$F_f|_W$ is univalent and its distortion is
     bounded by~$\Delta_3$.
     Furthermore, the inverse of~$F_f|_W$ admits a univalent extension to~$\hV_f$
     taking images in~$\pV$.
     In particular, $F_f$ is uniformly expanding with respect to the hyperbolic
     metric on~$\hV_f$.
   \end{lemm}

   Denote by~$\pfD$ the collection of connected components of~$\pD$ and
   if $c(f)$ is real denote by~$\pfD^{\R}$ the sub-collection of~$\pfD$ of those sets intersecting~$I(f)$.
   For each~$W$ in~$\pfD$ denote by~$\phi_W \colon \hV_f \to \pV$ the extension of~$F_f|_{W}^{-1}$ given by Lemma~\ref{l:bounded distortion to nice}. Given an
   integer~$\ell \ge 1$ we denote by~$E_{f,\ell}$ (resp. $E_{f,\ell}^{\R}$) the
   set of all words of length~$\ell$ in the alphabet~$\pfD$ (resp. $\pfD^{\R}$).
   Again by Lemma~\ref{l:bounded distortion to nice}, for each integer~$\ell \ge 1$ and
   each word~$W_1 \cdots W_\ell$  in~$E_{f,\ell}$ the composition
   $$ \phi_{W_1 \cdots W_\ell} \= \phi_{W_1} \circ \cdots \circ \phi_{W_\ell} $$
   is defined on~$\hV_f$.
   We also put
   $$ m_f(W_1 \cdots W_\ell) \= m_f(W_1) + \cdots + m_f(W_\ell). $$

   For~$t, p$ in~$\R$ and an integer~$\ell \ge 1$ put
   \begin{align*}
     Z_{\ell}(t, p)
     & \=
       \sum_{\underline{W} \in E_{f,\ell}} \exp(-m_f(\underline{W}) p) \left( \sup \{
       |D\phi_{\underline{W}}(z) | \mid z \in \pV \} \right)^t
       \intertext{ and }
       Z_{\ell}^{\R}(t, p)
     & \=
       \sum_{\underline{W} \in E_{f,\ell}^{\R}} \exp(-m_f(\underline{W}) p) \left(
       \sup \{ |D\phi_{\underline{W}}(z) | \mid z \in \pV \} \right)^t.
   \end{align*}
   For a fixed~$t$ and~$p$ in~$\R$ 
   the sequence
   $$ \left(\frac{1}{\ell} \log Z_{\ell}(t, p) \right)_{\ell = 1}^{+\infty}
   \left( \text{resp. } \left(\frac{1}{\ell}  \log Z_{\ell}^{\R}(t, p) \right)_{\ell = 1}^{+\infty} \right) $$
   converges to the pressure function of~$F_f$ (resp. $F_f|_{\pD \cap I(f)}$) for the
   potential ${- t \log |D\pF| - p m_f}$; 
   we denote it by~$\psP(t, p)$ (resp. $\psP^{\R}(t, p)$).
   On the set where it is finite, the function~$\psP$ (resp.~$\psP^{\R}$) so
   defined is continuous and strictly decreasing in each of its variables.

   Given~$t > 0$ and~$p$ in~$\R$, a finite measure~$\tmu$ on~$\C$ that is supported on the maximal invariant set of~$F|_{D_f \cap \R}$ (resp.~$F$) is \emph{$(t, p)$\nobreakdash-conformal for~$F_f$}, if for every~$W$ in~$\fD^{\R}_f$ (resp. $\fD_f$), and every Borel subset~$U$ of~$W \cap \R$ (resp.~$W$), we have
   \begin{displaymath}
     \tmu(F_f(U))
     =
     \exp(p m_f(W)) \int_U |DF_f|^t \dd \tmu.
   \end{displaymath}
   Note that in this case we have
   \begin{multline}
     \label{eq:3}
     \exp(- p m_f(W)) \inf_{z \in W} |DF_f (z) |^{-t}
     \le
     \tmu(W)
     \\ \le
     \exp(- p m_f(W)) \sup_{z \in W} |DF_f (z) |^{-t}.
   \end{multline}

   \begin{prop}
     \label{p:equilibria}
     Let~$n \ge 5$ be an integer, $f$ in~$\cK_n(\sF)$, and~$t > 0$ such that
     \begin{equation}
       \label{e:Bowen formula}
       \sP^{\R}_f(t, P_f^{\R}(t)) = 0
       \left( \text{resp. }
         \sP_f(t, P_f(t)) = 0 \right).
     \end{equation}
     Then there is a~$(t, P^\R_f(t))$\nobreakdash-conformal (resp.~$(t, P_f(t))$\nobreakdash-conformal) probability measure~$\tmu$ for~$F_f$, and there is a probability measure~$\trho$ that is invariant by~$F_f$, absolutely continuous with respect to~$\tmu$, and whose density satisfies
     \begin{equation}
       \label{eq:4}
       \Delta_3^{-t} \le \frac{\dd \trho}{\dd \tmu} \le \Delta_3^t.
     \end{equation}
     If in addition
     \begin{multline}
       \label{e:time integrability}
       \sum_{W \in \fD_f^\R} m_{f}(W) \cdot \exp( - m_{f}(W) P^\R_f(t)) \sup_{w \in W \cap \R} |DF_f(w)|^{-t}
       \\ \left( \text{resp. }
         \sum_{W \in \fD_f} m_{f}(W) \cdot \exp( - m_{f}(W) P_f(t)) \sup_{w \in W} |DF_f(w)|^{-t}
       \right)
     \end{multline}
     is finite, then the measure
     $$ \hrho \= \sum_{W \in \fD_f^{\R}} \sum_{j = 0}^{m_{f}(W) - 1} (f^j)_* \left( \trho|_{W \cap \R} \right)
     \left( \text{resp. }
       \sum_{W \in \fD_f} \sum_{j = 0}^{m_{f}(W) - 1} (f^j)_* \left( \trho_t|_W \right)
     \right) $$
     is finite and the probability measure proportional to~$\hrho$ is the unique equilibrium state of~$f|_{I(f)}$ (resp.~$f|_{J(f)}$) for the potential~$-t \log |Df|$. 
   \end{prop}
   \begin{proof}
     The proof is standard, refer to~\cite[\S4]{PrzRiv11} for precisions.
     The existence of the conformal measure follows from the same arguments given in~\cite[Theorem~A in~\S4 and Proposition~4.3]{PrzRiv11}.
     To construct an absolutely continuous invariant measure, let~$\ell \ge 1$ be an integer, and let~$\uW$ be a word in~$E_{f, \ell}$.
     Then by Lemma~\ref{l:bounded distortion to nice} and~\eqref{eq:3}, for every integer~$\ell' \ge 1$ we have in the complex case
     \begin{equation*}
       \begin{split}
         \tmu (F_f^{-\ell'}(\phi_{\uW}(\pV)))
         & =
         \sum_{\uW' \in E_{f, \ell'}} \tmu \left( \phi_{\uW'} \circ \phi_{\uW} (\pV) \right)
         \\ & \ge
         \tmu(\phi_{\uW}(\pV)) \sum_{\uW' \in E_{f, \ell'}} \exp \left(- m_f(\uW') P_f(t) \right) \inf_{z \in \phi_{\uW}(\pV)} |D\phi_{\uW'}(z)|^t
         \\ & \ge
         \Delta_3^{-t} \tmu(\phi_{\uW}(\pV)) \sum_{\uW' \in E_{f, \ell'}} \tmu(\phi_{\uW'}(V_{f}))
         \\ & =
         \Delta_3^{-t} \tmu(\phi_{\uW}(\pV)).
       \end{split}
     \end{equation*}
     A similar argument shows that~$\tmu (F_f^{-\ell'}(\phi_{\uW}(\pV))) \le \Delta_3^{t} \tmu(\phi_{\uW}(\pV))$.
     Analogous inequalities also hold in the real case.
     Since these inequalities hold for every~$\ell \ge 1$, and every~$\uW$ in~$E_{f, \ell}$, it follows that any weak* accumulation measure of~$\left( \frac{1}{k} \sum_{\ell = 0}^{k - 1} (F_f)_*^\ell \tmu \right)_{k = 1}^{+\infty}$ is an invariant probability measure satisfying the desired properties.

     To prove the last statement, note that by~\eqref{eq:3} and~\eqref{eq:4}, our hypothesis~\eqref{e:time integrability} implies that
     $$ \sum_{W \in \fD_f^\R} m_{f}(W) \trho (W)
     \left( \text{resp. } \sum_{W \in \fD_f} m_{f}(W) \trho_t (W) \right) $$
     is finite, so the measure~$\hrho$ is finite.
     The last statement of the proposition follows as in the proof of~\cite[Proposition~A]{CorRiv13}.
   \end{proof}

   \subsection{Estimating the 2~variables pressure function}
   \label{ss:Bowen type formula}
   The following is our main tool to estimate the 2~variables pressure function, in order to verify the hypotheses of Proposition~\ref{p:equilibria}.

   \begin{proproman}
     \label{p:improved MS criterion}
     Let~$\kappa>0$ and~$\upsilon>0$ be given.
     Then there are~$n_3 \ge 5$ and~$C_6 > 1$ that only depend on~$K$, $R$, $\kappa$, and~$\upsilon$, such that for every integer~$n \ge n_3$, and every~$f$ in~$\cK_n(\sF)$ satisfying the Geometric Peierls Condition with constants~$\kappa$ and~$\upsilon$,
     the following properties hold for each~$t \ge  2\log2/\upsilon$.
     \begin{enumerate}
     \item[1.]
       For~$p$ in~$[- t \pchicrit/2, 0)$ satisfying
       $$ \sum_{k = 0}^{+\infty} \exp(- (n + 3k)p)|Df^{n + 3k} (f(0))|^{-t/2}
       \ge
       C_6^{t}, $$
       we have ${\psP^{\R}(t, p) > 0}$ and ${P_f^{\R}(t) > p}$.
       If in addition the sum above is finite, then~$\psP(t, p)$ is finite.
     \item[2.]
       For~$p \ge - t \pchicrit/2$ satisfying
       $$ \sum_{k = 0}^{+\infty} \exp(- (n + 3k)p)|Df^{n + 3k} (f(0))|^{-t/2}
       \le
       C_6^{-t}, $$
       we have $\psP(t, p) < 0$ and~$P_f(t) \le p$.
     \item[3.]
       For~$p \ge - t \pchicrit/2$ satisfying
       $$ \sum_{k = 0}^{+\infty} k \cdot \exp(- (n + 3k)p)|Df^{n + 3k} (f(0))|^{-t/2}
       <
       +\infty, $$
       we have
       $$ \sum_{W \in \pfD} m_f(W) \cdot \exp(- m_f(W)p) \sup_{z \in W}|D\pF(z)|^{-t}
       <
       +\infty. $$
     \end{enumerate}
   \end{proproman}

   The proof of this proposition is given in~\S\ref{ss:proof of MS criterion}.
   In the rest of this subsection we prove~2 results that are used in the proof of the proposition above.
   The first is a Bowen type formula relating~$P_f^{\R}$ (resp.~$P_f$) to the~2 variables pressure function of~$F_f$ (Proposition~\ref{p:Bowen type formula}).
   The second is a lower bound for the pressure function (Proposition~\ref{p:critical line}).

   \begin{proproman}[Bowen type formula]
     \label{p:Bowen type formula}
     For every $\kappa>0$, every $\upsilon>0$, every 
     integer~$n \ge 5$, and every $f$ in~$\cK_n(\sF)$ satisfying the Geometric Peierls Condition with constants $\kappa$ and $\upsilon$,
     we have for each~$t \ge 2\log2/\upsilon$,
     $$ P_f^{\R}(t)
     =
     \inf \left\{ p \mid \psP^{\R}(t, p) \le 0 \right\}
     \left( \text{resp. } P_f(t)
       =
       \inf \left\{ p \mid \psP(t, p) \le 0 \right\} \right). $$
   \end{proproman}
   The proof of this proposition is at the end of the subsection.
   It uses several results that are also used in the next subsection.
   The proof that the geometric pressure function is smaller or equal than the infimum is simple and depends basically on Lemma~\ref{l:bounded distortion to nice}.
   The other inequality is much more involved.
   It requires the Geometric Peierls condition, and a lower bound on the pressure function that we proceed to state and prove.
   \begin{proproman}[Critical line]
     \label{p:critical line}
     For every integer~$n \ge 5$, and every~$f$ in~$\cK_n(\sF)$, we have
     $$ \chiinfR
     \=
     \inf \left\{ \int \log |Df| \dd \mu \mid \mu \in \sM_f^{\R} \right\}
     \le \pchicrit/2. $$
     In particular, for each~$t > 0$ we have
     $$ P_f(t) \ge P_f^{\R}(t) \ge - t \pchicrit/2. $$
   \end{proproman}

   The proof of this proposition is given after the following lemma.

   \begin{lemm}
     \label{l:critical line}
     There is a constant~$C_7 > 0$ that only depends on~$K$ and~$R$, such that for each integer~$n \ge 5$ and each 
     ~$f$ in~$\cK_n(\sF)$, the following property holds: For every integer~$k \ge 0$ there is a connected component~$\hW$ of~$\pD$ contained in~$P_{f, n + 3k + 2}(0)$, such that $h_f^{-1}(\hW)$ intersects~$\R$, and such that~$m_f(\hW) = n + 3k + 3$ and
     $$ \sup_{z \in \hW} |D\pF(z)|
     \le
     C_7 |Df^{n + 3k} (f(0))|^{1 / 2}. $$
   \end{lemm}

   \begin{proof}
     We follow the proof of~\cite[Lemma~6.3]{CorRiv13}. Let $C_1$ be the constant given by Lemma~\ref{l:uniform family} and let $C_2$ be the constant given by Lemma~\ref{l:Mori}.
     Let~$\Delta_2 > 1$ and~$\Delta_3 > 1$ be the constants given by 
     Lemmas~\ref{l:distortion to central 0} 
     and~\ref{l:bounded distortion to nice}, respectively.
     From the proof of~\cite[Lemma~6.3]{CorRiv13}, we have
     \begin{align*}
       \Xi_0
       & \=
         \sup_{c \in \cP_4(-2)} \diam(P_{c, 1}(0))
         <
         +\infty
         \intertext{ and }
         \Xi_1
       & \=
         \sup_{c \in \cP_4(-2)} \sup_{z \in P_{c, 1}(0)} |Df_c^2 (z)|
         <
         +\infty.
     \end{align*}

     By Lemma~\ref{l:Mori} and inequality~\eqref{e:derivative Mori}, we have
     \begin{align*}
       \hXi_0
     & \=
     \sup_{f \in \cP_4(\sF) } \diam(P_{f, 1}(0)) \le C_2 \Xi_0^{\frac{1}{K}}
     <
       +\infty,
       \intertext{ and } 
       \hXi_1
       & \=
         \sup_{f \in \cP_4(\sF)} \sup_{z \in P_{f,1}(0)} |Df^2 (z)| \le C_1^2 \Xi_1^{\frac{2}{K}}
     <
     +\infty.
     \end{align*}

     Fix an integer~$n \ge 5$, $f$ in~$\cK_n(\sF)$, and an integer~$k \ge 0$.
     By the proof of~\cite[Lemma~6.3]{CorRiv13}  there is a connected component~$W$ of~$D_{c(f)}$ contained in~$P_{c(f),n + 3k + 2}(0)$, that intersects~$\R$, and such that~$m_{c(f)}(W) = n + 3k + 3$. 
     The set  $\hW=h_f(W)$ verifies the desired properties of the lemma. To finish the proof it remains to prove the inequality.

     Let~$z_{\hW}$ be the unique point in~$\hW$ such that~$f^{n + 3k + 3}(z_{\hW}) = 0$.
     Then~$f^{n + 3k + 1}(z_{\hW})$ belongs to~$P_{f,1}(0)$, so by definition
     of~$\hXi_0$ we have
     \begin{equation}
       \label{e:image distance estimate}
       |f^{n + 3k + 1}(z_{\hW}) - f^{n + 3k}(f(0))|
       \le
       \diam(P_{f,1}(0))
       \le
       \hXi_0.
     \end{equation}
     Since~$f_c^n$ maps~$P_{c, n + 1}(c)$ biholomorphically to~$P_{c, 1}(0)$ and~$f_c^n(c) \in \Lambda_c$, it follows that~$f_c^{n + 3k}$ maps~$P_{c, n + 3k+ 1}(c)$ biholomorphically to~$P_{c, 1}(0)$, and the same holds for $f^{n+3k}$ and the sets~$P_{f, n + 3k+ 1}(f(0))$ and~$P_{f,1}(0)$; so the distortion of~$f^{n + 3k}$
     on~$P_{f, n + 3k + 1}(f(0))$ is bounded by~$\Delta_2$ (Lemma~\ref{l:distortion to central 0}) and for each point~$y$ in~$P_{f, n + 3k + 1}(f(0))$ we have
     \begin{equation}
       \label{e:binding estimate}
       \Delta_2^{-1} |Df^{n + 3k} (f(0))|
       \le
       |Df^{n + 3k} (y)|
       \le
       \Delta_2 |Df^{n + 3k} (f(0))|.
     \end{equation}
     Together with~\eqref{e:image distance estimate} this implies that,
     $$ |f(z_{\hW}) - f(0)|
     \le
     \Delta_2 \hXi_0 |Df^{n + 3k} (f(0))|^{-1} $$
     and by Lemma~\ref{l:uniform family},
     $$ |Df(z_{\hW})|
     \le
     C_1^2 \Delta_2^{1/2} \hXi_0^{1/2} |Df^{n + 3k} (f(0))|^{-1/2}. $$
     Combined with~\eqref{e:binding estimate} with~$y = f(z_{\hW})$, this implies
     \begin{equation*}
       \label{e:Julia's estimate}
       |Df^{n + 3k + 1} (z_{\hW})|
       \le
       C_1^2 \Delta_2^{3/2} \hXi_0^{1/2} |Df^{n + 3k} (f(0))|^{1/2}.
     \end{equation*}
     Putting~$C_7 \= C_1^2\Delta_3 \hXi_1 \Delta_2^{3/2} \hXi_0^{1/2}$, we get by Lemma~\ref{l:bounded distortion to nice}
     \begin{align*}
       \sup_{z \in \hW} |D\pF(z)|
       & \le
         \Delta_3  |Df^{n + 3k + 3} (z_{\hW})|
       \\ & \le
            \Delta_3 \hXi_1 |Df^{n + 3k + 1} (z_{\hW})|
       \\ & \le
            C_7 |Df^{n + 3k} (f(0))|^{1/2}.
     \end{align*}
   \end{proof}

   \begin{proof}[Proof of Proposition~\ref{p:critical line}]
     The proof follows from Lemma~\ref{l:critical line} by constructing a sequence of  measures supported in periodic points whose Lyapunov exponents converge to~$\pchicrit$.
     For details see the proof of~\cite[Proposition~6.2]{CorRiv13}.
   \end{proof}

   \begin{lemm}
     \label{l:landing contribution}
     Let~$\kappa>0$ and~$\upsilon>0$ be given.
     Then there is~$C_8 > 1$ that only depends on~$K$, $R$, $\kappa$, and~$\upsilon$, such that  for every integer~$n \ge 5$, and every~$f$ in~$\cK_n(\sF)$ satisfying the Geometric Peierls Condition with constants $\kappa$ and $\upsilon$,
     the following property holds: For every
     $$ t \ge  2\log2/\upsilon,
     p \ge - t (\pchicrit + \tfrac{2\upsilon}{3})/2, $$
     and~$y$ in~$\pV$ we have
     $$ \tL_{t, p}(y)
     \=
     1 + \sum_{z \in \pL^{-1}(y)} (m_f(z) + 1) \exp(- m_f(z) p) |D\pL (z)|^{-t}
     \le
     C_8^t.$$
   \end{lemm}
   \begin{proof}
     Fix $\kappa>0$ and $\upsilon>0$, and put $t_0 \= 2\log2/\upsilon$.
     We prove the lemma with
     $$ C_8
     \=
     \max \{1, \kappa^{-1} \} \left( 1 - \exp(\log 2-(2/3)\upsilon t_0) \right)^{- 2/t_0}.$$ 
     Fix $n\ge 5$ and $f$ in~$\cK_n(\sF)$ satisfying the Geometric Peierls Condition with constants $\kappa$ and $\upsilon$.

     Fix~$t \ge t_0$, $p \ge - t(\pchicrit/2 +    \upsilon/2)$, and~$y$ in~$\pV$.
     Since $f$ satisfies the Geometric Peierls Condition, for every  $z \in \pL^{-1}(y)$ we have
     $$ \exp(- m_f(z) p) |D\pL (z)|^{-t}
     \le \kappa^{-t} \exp(-(2/3)\upsilon\cdot  m_f(z)t). $$
     On the other hand, for each integer~$m \ge 1$ the set~$\{ z \in \pL^{-1}(y) \mid
     m_f(z) = m \}$ is contained in~$f^{-m}(y)$ and therefore it contains at most~$2^m$ points.
     So, we have
     \begin{equation*}
       \tL_{t, p}(y)
       \le
       \max \{1, \kappa^{-1} \}^t \sum_{m = 0}^{+\infty} (m + 1) \exp(m(\log2- (2/3)\upsilon t)) \le C_8^t.
       \qedhere
     \end{equation*}
   \end{proof}

   \begin{lemm}
     \label{l:Poincare series}
     Given an integer~$n \ge 5$ and~$f$ in~$\cK_n(\sF)$, the following property
     holds for every~$t$ in~$(0, +\infty)$ and every real number~$p$: If~$\psP^{\R}(t, p) > 0$ (resp. $\psP(t, p) > 0$), then the series
     \begin{multline}
       \label{e:Poincare series}
       \sum_{j = 1}^{+\infty} \exp(- j p) \sum_{y \in f|_{I(f)}^{-j}(0)} |Df^j (y)|^{-t}  
       \\ 
       \left( \text{resp. $\sum_{j = 1}^{+\infty} \exp(- j p) \sum_{y \in f^{-j}(0)} |Df^j (y)|^{-t}$} \right)  
     \end{multline}
     diverges.
     On the other hand,  if for some~$\kappa>0$ and~$\upsilon>0$ the map~$f$ satisfies the Geometric Peierls Condition with constants $\kappa$ and $\upsilon$,  then for  every
     $$ t \ge  2\log 2/\upsilon
     \text{ and }
     p \ge P_f^{\R}(t) - t \tfrac{\upsilon}{3} 
     \left( \text{resp. $p \ge P_f(t) - t \tfrac{\upsilon}{3}$ } \right)$$
     satisfying $\psP^{\R}(t, p) < 0$ (resp. $\psP(t, p) < 0$), the series above converges.
   \end{lemm}

   \begin{proof} 
     The proof of the first part of the lemma depends on Lemma~\ref{l:bounded distortion to nice} and it follows the same lines that the first part of \cite[Lemma~6.5]{CorRiv13}.

     We prove the last assertion concerning~$f|_{J_f}$; the arguments apply without change to~$f|_{I(f)}$. 
     Fix  some positive constants $\kappa$ and $\upsilon$ 
     and  let~$C_8 > 1$ be given by Lemma~\ref{l:landing contribution} for the constants $\kappa$ and $\upsilon$.
     Let~$f$ be a map in~$\sF$ satisfying the Geometric Peierls Condition with constants $\kappa$ and $\upsilon$, and let
     $$ t \ge 2\log 2 / \upsilon
     \text{ and }
     p \ge P_f(t) - t \tfrac{\upsilon}{3}  $$
     be such that~$\sP_f(t, p) < 0$.
     By Proposition~\ref{p:critical line} we have 
     $$p \ge - t(\chicritf + \tfrac{2}{3} \upsilon)/2,$$ so~$t$ and~$p$ satisfy the hypotheses of Lemma~\ref{l:landing contribution}.
     Given an integer~$m \ge 1$ and a point~$z$ in~$f^{-m}(0)$ denote by~$\ell(z)$ the number of those~$j$ in~$\{0, \ldots, m - 1 \}$ such that~$f^j(z)$ is in~$\pV$.
     In the case where~$z$ is not in~$\pV$, this point is in the domain of~$L_f$ and we have~$\ell(z) = 0$ if and only~$L_f(z) = 0$.
     Moreover, if~$z$ is not in~$\pV$ and~$\ell(z) \ge 1$, then~$L_f(z)$ is in the domain of~$F_f^{\ell(z)}$ and~$F_f^{\ell(z)}(L_f(z)) = 0$.
     So, if~$z$ is not in~$\pV$ we have in all the cases,
     $$ |Df_f^m(z)| = |DF_f^{\ell(z)} (L_f(z))| \cdot |DL_f (z)|. $$
     Then Lemma~\ref{l:landing contribution} and our hypothesis $\psP(t, p) < 0$ imply that the series~\eqref{e:Poincare series} is bounded from above by
     \begin{multline*}
       \tL_{t, p}(0) + \sum_{\ell = 1}^{+\infty} \sum_{y \in F_f^{- \ell}(0)} \tL_{t, p}(y) \exp(- (m_f(F_f^{\ell - 1}(y)) + \cdots + m_f(y))p) |DF_f^{\ell} (y)|^{-t}
       \\ \le
       C_8^t \left( 1 + \sum_{\ell = 0}^{+\infty} Z_{f, \ell}(t, p) \right)
       < +\infty.
     \end{multline*}
   \end{proof}
   \begin{proof}[Proof of Proposition~\ref{p:Bowen type formula}]
     We follow the proof of \cite[Proposition~C]{CorRiv13}.

     We prove the assertion for~$f|_{J(f)}$; the arguments apply without change to~$f|_{I(f)}$.
     Let~$\Delta_3 > 1$ be given by Lemma~\ref{l:bounded distortion to nice}.
     Let~$n \ge 5$ be an integer and let~$f$ be  in~$\cK_n(\sF)$.
     We use that fact that for each~$t > 0$ we have
     \begin{equation}
       \label{e:pressure of original}
       P_f(t)
       =
       \limsup_{m \to +\infty} \frac{1}{m} \log \sum_{y \in f^{-m}(0)}
       |Df^m(y)|^{-t},
     \end{equation}
     see for example~\cite{PrzRivSmi04}.

     Fix~$t \ge 2\log 2 /\upsilon$.
     We use the fact that the function~$p \mapsto \psP(t, p)$ is strictly decreasing where it is finite, see~\S\ref{ss:Inducing scheme}.
     In particular, for each~$p$ satisfying
     \begin{equation*}
       p
       <
       p_0
       \=
       \inf \{ p \mid \psP(t, p) \le 0 \},
     \end{equation*}
     we have~$\psP(t, p) > 0$.
     Lemma~\ref{l:Poincare series} implies that for such~$p$ the series~\eqref{e:Poincare series} diverges and by~\eqref{e:pressure of original} we have~$P_f(t) \ge p$. It follows that,~$P_f(t) \ge p_0$.
     To prove the reverse inequality, suppose by contradiction~$p_0 <  P_f(t)$ and let~$p$ be in the interval~$(p_0, P_f(t))$ satisfying~$p \ge P_f(t) - t \tfrac{\upsilon}{3} $.
     Then~$\psP(t, p) < 0$ and by Lemma~\ref{l:Poincare series} the series~\eqref{e:Poincare series} converges.
     Then~\eqref{e:pressure of original} implies~$P_f(t) \le p$ and we obtain a contradiction that completes the proof of the proposition.
   \end{proof}

   \subsection{Proof of Proposition~\ref{p:improved MS criterion}}
   \label{ss:proof of MS criterion}
   The final step in the proof of Proposition~\ref{p:improved MS criterion} is given after the following proposition, which estimates the partition function of the induced map in terms of the derivative of the iterates of the map at its critical value.

   Let~$n \ge 4$ be an integer and~$f$  in~$\cK_n(\sF)$.
   Since the critical point~$z = 0$ does not belong to~$\pD$ (\emph{cf}., \cite[Lemma~4.2]{CorRiv13}), for each integer~$\ell \ge 1$, each connected component of~$\pD$ intersecting~$P_{f, \ell}(0)$ is contained
   in~$P_{f, \ell}(0)$.
   Define the \emph{level} of a connected component~$W$ of~$\pD$ as the largest integer~$k \ge 0$ such that~$W$ is contained in~$P_{f, n + 3k + 2}(0)$.
   Given an integer~$k \ge 0$ denote by~$\fD_{f,k}$ the collection of all connected components of~$\pD$ of level~$k$; we have~$\pfD = \bigcup_{k = 0}^{+\infty} \fD_{f,k}$, and for every~$W$ in~$\fD_{f, k}$ we have~$m_f(W) \ge n + 3k + 1$.

   \begin{prop}
     \label{p:estimating Z_1 by the postcritical series}
     Let~$\kappa > 0$ and~$\upsilon > 0$ be given.
     Then there are~$n_4 \ge 5$ and~$C_9 > 1$ that only depend on~$K$, $R$, $\kappa$, and~$\upsilon$, such that for every integer~$n \ge n_4$, and every~$f$ in~$\cK_n(\sF)$ satisfying the Geometric Peierls Condition with constants~$\kappa$ and~$\upsilon$, the following properties hold for each~$t \ge 2\log 2 / \upsilon$
     and each integer $k \ge 0$:
     \begin{enumerate}
     \item[1.]
       For each~$p$ in~${(-\infty, 0)}$, we have
       \begin{multline*} 
         \sum_{W \in \fD_{f,k}\cap \pfD^\R} \exp(-m_f(W)p) \inf_{z\in W} |D\pF(z)|^{-t}
         \\ >
         C_9^{-t} \exp(- (n + 3k)p)|Df^{n + 3k} (f(0))|^{-t/2}.
       \end{multline*}
     \item[2.]
       For each~$p \ge - t \pchicrit/2 - t \upsilon/3$, we have
       \begin{multline*} 
         \sum_{W \in \fD_{f,k}} \exp(-m_f(W)p) \sup_{z\in W} |D\pF(z)|^{-t}
         \\
         \begin{aligned}
           & \le
           \sum_{W \in \fD_{f,k}} (m_f(W)-(n+3k)) \exp(-m_f(W)p) \sup_{z\in W} |D\pF(z)|^{-t}
           \\ & <
           C_9^{t} \exp(- (n + 3k)p)|Df^{n + 3k} (f(0))|^{-t/2}.
         \end{aligned}
       \end{multline*}
     \end{enumerate}
   \end{prop}

   The proof of this proposition is given after the following lemma.

   \begin{lemm}
     \label{l:level k contribution}
     There is~$C_{10} > 1$ that only depends on~$K$ and~$R$, such that for each integer~$n \ge 5$, each~$f$ in~$\cK_n(\sF)$, each integer~$k \ge 0$, and each pair of real numbers~$t > 0$ and~$p$, we have
     \begin{multline*}
       \sum_{W \in \fD_{f,k}} \exp(- m_f(W) p) \sup_{z \in W}|D\pF(z)|^{-t}
       \\
       \begin{aligned}
         & \le
         2 C_{10}^t  \exp(- (n + 3k + 1) p)|Df^{n + 3k} (f(0))|^{-t/2}
         \\ & \quad \cdot
         \left( 1 + \sum_{w \in \pL^{-1}(0) \text{ in } P_{f,1}(0)} \exp(- m_f(w) p) |D\pL(w)|^{-t} \right).  
       \end{aligned}
     \end{multline*}
     Moreover,
     \begin{multline*}
       \sum_{W \in \fD_{f,k}} (m_f(W) -(n+3k)) \exp(- m_f(W) p) \sup_{z \in W}|D\pF(z)|^{-t}
       \\
       \begin{aligned}
         & \le
         2 C_{10}^t  \exp(- (n + 3k + 1) p)|Df^{n + 3k} (f(0))|^{-t/2}
         \\ & \quad \cdot
         \left( 1 + \sum_{w \in \pL^{-1}(0) \text{ in } P_{f,1}(0)} ( m_f(w) + 1) \exp(- m_f(w) p) |D\pL(w)|^{-t} \right).
       \end{aligned}
     \end{multline*}
   \end{lemm}

   \proof
   The proof follows the same lines that the proof of~\cite[Lemma~7.1]{CorRiv13}, and it depends on Lemmas~\ref{l:distortion to central 0}, \ref{l:first return to central derivative}, and~\ref{l:bounded distortion to nice}, as well as on~\cite[Lemma~5.1]{CorRiv13}. 
   The second inequality does not appear in \cite[Lemma~7.1]{CorRiv13}.
   It follows from the first displayed equation in the proof of~\cite[Lemma~7.1]{CorRiv13}, and from~\cite[(7.1)]{CorRiv13}.
   See the proof of~\cite[Lemma~7.1]{CorRiv13} for further details.
   \endproof

   \begin{proof}[Proof of Proposition~\ref{p:estimating Z_1 by the postcritical series}]
     The proof depends on Lemmas~\ref{l:critical line}, \ref{l:landing contribution}, and~\ref{l:level k contribution}, and on Proposition~\ref{p:Bowen type formula}, and it follows the same lines that the proof of  \cite[Lemma~7.2]{CorRiv13}.
     There are some differences in item~2 since the condition on $p$ is slightly different and we add a new inequality.
     We include the proof of item~2 here.

     Let~$C_8 > 0$ and~$C_{10}>0$  be given by Lemmas~\ref{l:landing contribution} and~\ref{l:level k contribution}, respectively.
     Let $n_2$ be the integer given by Lemma~\ref{l:bound chicritf} with $\varepsilon=\tfrac{1}{10}$.  Put $t_0 \= 2\log 2 / \upsilon$. We prove the lemma for $n_4=n_2$.
     Fix an integer~$n \ge n_2$, a map~$f$ in~$\cK_n(\sF)$, $t \ge t_0$, and an integer $k \ge 0$.

     To prove item~2, note that the first inequality follows from the fact that for every~$W$ in~$\fD_{f, k}$ we have~$m_f(W) \ge n + 3k + 1$.
     To prove the second inequality, let~$p \ge - t\pchicrit/2 - t \upsilon/3 $ be given.
     By Lemma~\ref{l:bound chicritf}, we have
     \begin{equation*}
       \pchicrit \le 1.1 K \log 2.
     \end{equation*}
     Thus~$-p \le t  (0.55 K + \tfrac{\upsilon}{3\log2}) \log 2 < t  (0.55 K + 1/t_0) \log 2$ and therefore
     \begin{equation*}
       2 \exp(- p)
       <
       2^{t(0.55 K + 2/t_0)}.
     \end{equation*}
     Combined with Lemmas~\ref{l:landing contribution} and~\ref{l:level k contribution}, we obtain  item~2 of the proposition with ${C_9 = 2^{0.55 K + 2/t_0} C_{10} C_8}$.
   \end{proof}

   \begin{proof}[Proof of Proposition~\ref{p:improved MS criterion}] We follow the proof of  \cite[Proposition~D]{CorRiv13}. 
     Let~$n_4$ and~$C_9$ be given by Proposition~\ref{p:estimating Z_1 by the postcritical series}, and put $t_0 \= 2\log 2 /\upsilon$.
     To prove the proposition, fix an integer~$n$ satisfying ${n \ge n_4}$, a map~$f$ in~$\cK_n(\sF)$, and a real number~$t$ satisfying ${t \ge t_0} $.

     To prove item~1, let~$p$ in~$[- t \pchicrit / 2, 0)$ be such that the sum
     \begin{equation}
       \label{e:postcritical series}
       \sum_{k = 0}^{+\infty} \exp( - (n + 3k) p) |Df^{n + 3k}(f(0))|^{-t/2}
     \end{equation}
     is greater than or equal to~$(2C_9)^t$.
     Then, we can choose~$p'$ in~${(p, 0)}$ such that the sum above with~$p$ replaced by~$p'$ is strictly larger than~$C_9^t$.
     By item~1 of Proposition~\ref{p:estimating Z_1 by the postcritical series}, we have ${\psP^{\R}(t, p) \ge \psP^{\R}(t, p') > 0}$ and by Proposition~\ref{p:Bowen type formula} we have ${P_f^{\R}(t) \ge p' > p}$.
     This proves the first assertion of item~1 with~$C_6 = 2C_9$.
     To prove the second assertion of item~1, suppose that is finite.
     Then, by item~2 of Proposition~\ref{p:estimating Z_1 by the postcritical series} the sum
     $$ \sum_{W \in \pfD} \exp(- m_f (W)p) \sup_{z \in W}|D\pF(z)|^{-t} $$
     is finite, so~$\psP^{\R}(t, p)$ is also finite.
     This completes the proof of item~1.

     To prove item~2, let~$p \ge - t\pchicrit/2$ be given.
     By item~2 of Proposition~\ref{p:estimating Z_1 by the postcritical series}, if~\eqref{e:postcritical series} is less than or equal to~$C_9^{-t}$, then~$\sP_f(t, p) < 0$ and by Proposition~\ref{p:Bowen type formula} we have~$P_f(t) \le p$.
     This proves item~2 of the proposition with~$C_6 = C_9$.

     To prove item~3, let~$p \ge - t \pchicrit/2$ be given and put~$p' \= p - t \tfrac{\upsilon}{3}$.
     By item~2 of Proposition~\ref{p:estimating Z_1 by the postcritical series} with~$k = 0$, the sum
     $$ \sum_{W \in \fD_{f, 0}} \exp(- m_f(W) p) \sup_{z \in W}|D\pF(z)|^{-t} $$
     is finite.
     Let~$A > 0$ be a constant such that for every pair of integers~$k \ge 1$ and~$m \ge 3k + 1$, we have
     $$ m
     \le
     A k \exp(t_0 \upsilon (m - 3k)/3). $$
     Applying item~2 of Proposition~\ref{p:estimating Z_1 by the postcritical series} with~$p$ replaced by~$p'$, we obtain that for each integer~$k \ge 1$ we have
     \begin{multline*}
       \sum_{W \in \fD_{f,k}} m_f(W) \cdot \exp(- m_f(W) p) \sup_{z \in W}|D\pF(z)|^{-t}
       \\
       \begin{aligned}
         & \le
         \sum_{W \in \fD_{f,k}} A k \exp(t \upsilon (m_f(W) - 3k)/3) \exp(- m_f(W) p) \sup_{z \in W}|D\pF(z)|^{-t}
         \\ & =
         A k \exp(t \upsilon k)  \sum_{W \in \fD_{f,k}} \exp \left( - m_f(W) p' \right) \sup_{z \in W}|D\pF(z)|^{-t}
         \\ & \le
         \left(A C_9^t \exp(-t \upsilon n/3)  \right) k \cdot \exp(- (n + 3k)p) |Df^{n + 3k}(f(0))|^{-t/2}.
       \end{aligned}
     \end{multline*}
     Summing over~$k \ge 0$ we obtain the desired assertion.
   \end{proof}

   \section{Sensitive dependence of geometric Gibbs states}
   \label{s:chaotic dependence at zero-temperature}
   In this section we prove the Main Theorem.
   In~\S\ref{ss:itineraries} we define the family of itineraries of the maps used in the Main Theorem, as well as other combinatorial objects used in the proof.
   In \S\ref{ss:estimating postcritical series} we estimate the postcritical series in terms of a certain 2~variables series that only depends on the combinatorics of the postcritical orbit (Lemma~\ref{l:2 variables functions}).
   The main estimates needed in the proof of the Main Theorem can be stated only in terms of these 2~variables series, and are relegated to Appendix~\ref{s:abstract estimates}.
   The proof of the Main Theorem is given in \S\ref{ss:proof of Main Theorem}, and it is divided in 3~parts.
   In the first part (\S\ref{sss:parameter-selection}), we introduce the family of maps~$(\fs)_{\uvarsigma \in \signs}$, which is mostly defined through the combinatorics of the postcritical orbit.
   In the second part (\S\ref{sss:pressure estimates}), we estimate the geometric pressure, and prove the existence and uniqueness of geometric Gibbs states, as well as the existence of conformal measures.
   The third, and most difficult part of the proof is given in \S\ref{sss:temperature dependence}, where we show that on certain intervals of inverse temperatures the geometric Gibbs states are concentrated near the orbit of~$p^+$, or the orbit of~$p^-$. 

   \subsection{The family of itineraries}
   \label{ss:itineraries}
   In this section we define several combinatorial objects used in the proof of the Main Theorem in~\S\S\ref{ss:estimating postcritical series}, \ref{ss:proof of Main Theorem}.
   One of the most important ones, is a family of itineraries~$(\iota(\uvarsigma))_{\uvarsigma \in \signs}$ in~$\{0, 1\}^{\N_0}$.
   Its definition depends on a choice of nonnegative integers~$q$ and~$\Xi$, satisfying
   \begin{equation}
     \label{eq:5}
     q
     \ge
     50(\Xi + 1)
     \text{ and }
     q + \Xi
     \equiv
     0 \mod 2.
   \end{equation}
   These integers are chosen in~\S\ref{sss:parameter-selection}.
   They depend on the uniform family~$\sF$ and of a choice of a sufficiently large integer~$n$, as in the Main Theorem.

   Endow each of the sets~$\{+, -\}$, $\{0, 1\}$, and~$\{0, 1^+, 1^- \}$ with the discrete topology, and each of the sets~$\signs$, ${\{0, 1 \}^{\N_0}}$, and ${\{0, 1^+, 1^- \}^{\N_0}}$ with the corresponding product topology.
   Roughly speaking, for a map~$f$ in~$\cK_n(\sF)$ a large block formed by repeated~$0$'s (resp.~$1$'s, $10$'s) in the itinerary~$\iota(f)$ of~$f$ indicates that a large portion of the critical orbit is close to the orbit of~$p(f)$ (resp.~$p^+(f)$, $p^-(f)$).
   To make this encoding more transparent, we use the auxiliary symbolic space ${\{0, 1^+, 1^- \}^{\N_0}}$ in which the symbol~$0$ (resp.~$1^+$, $1^-$) represents the orbit of~$p(f)$ (resp.~$p^+(f)$, $p^-(f)$).
   Thus, to define the family of itineraries~$(\iota(\uvarsigma))_{\uvarsigma \in \signs}$ in~$\{0, 1\}^{\N_0}$, we first define a family of sequences~$(\whx(\uvarsigma))_{\uvarsigma \in \signs}$ in~$\{0, 1^+, 1^-\}^{\N_0}$.
   Denote by~$\hSigma$ be the subspace of~${\{0, 1^+, 1^- \}^{\N_0}}$, given by
   \begin{multline}
     \label{eq:7}
     \hSigma
     \=
     \left\{ (\whx_j)_{j \in \N_0} \in \{0, 1^+, 1^-\}^{\N_0} \mid
     \right. \\ \left.
       \whx_j = 1^+ \Rightarrow \whx_{j + 1} \neq 1^-, \whx_j = 1^- \Rightarrow \whx_{j + 1} \neq 1^+ \right\}.
   \end{multline}

   For each~$s$ in~$\N_0$, define the integers
   \begin{equation*}
     a_s \= 2^{qs^3}
     \text{ and }
     b_s\=2^{qs^3} + q(2s + 1) + \Xi.
   \end{equation*}
   Note that ${b_s < a_{s+1}}$ and that in the case where ${s \ge 1}$ both~$a_s$ and~$b_s$ are even.
   Define the intervals~$I_s$ and~$J_s$ of~$\R$, by
   \[
     I_s
     \=
     \left[ a_s, b_s \right)
     \text{ and }
     J_s
     \=
     \left[ b_s, a_{s+1} \right).
   \]
   As~$s$ runs through~$\N_0$, the intervals~$I_s$ and~$J_s$ form a partition of~$[1, +\infty)$.
   Define functions
   \begin{equation*}
     N \colon \N_0 \to \N_0
     \text{ and }
     B \colon \N_0 \to \N_0
   \end{equation*}
   by $N(0) \= 0$, $B(0) \= 0$, for~$k$ in~$\N$ by
   \[
     N(k)
     \=
     \sharp \left\{ j \in \{0, \ldots, k-1 \} \mid j + 1 \in \bigcup_{s \in \N_0} I_s \right\},
   \]
   and for~$s$ in~$\N_0$ by
   \begin{equation}
     \label{e:B}
     B^{-1}(2s + 1) = I_s
     \quad \text{and} \quad
     B^{-1}(2(s + 1)) = J_s.
   \end{equation}
   Note that for every~$k$ in~$\N$ we have ${B(k) \le 2 N(k) + 1}$, and that
   \begin{equation}
     \label{eq:6}
     \lim_{k \to +\infty} \frac{N(k)}{k} = 0.
   \end{equation}

   For each~$\uvarsigma$ in~$\signs$, let~$\whx(\uvarsigma)$ be the sequence in ${\{0, 1^+, 1^- \}^{\N_0}}$ defined for~$j$ in~$\N_0$ by
   \begin{equation}
     \label{eq:8}
     \whx(\uvarsigma)_j
     \=
     \begin{cases}
       0
       &
       \text{if for some~$s$ in~$\N_0$ we have $j+1 \in I_s$};
       \\
       1^+
       &
       \text{if $j+1 \in J_0$};
       \\
       1^{\varsigma(m)}
       &
       \text{for } j +1 \in J_{4m - 3} \cup J_{4m - 2} \cup J_{4m - 1} \cup J_{4m}.
     \end{cases}
   \end{equation}
   Note that the first~$q$ entries of this sequence are equal to~$0$, that~$\whx(\uvarsigma)$ is in~$\hSigma$, and that the map ${\signs \to \hSigma}$ given by ${\uvarsigma \mapsto \whx(\uvarsigma)}$ is continuous.
   Moreover, the length of each maximal block of~$1^-$'s in~$\whx(\uvarsigma)$ is even, and for every~$k$ in~$\N$ we have
   \begin{equation}
     \label{eq:16}
     N(k)
   =
   \sharp \{ j \in \{0, \ldots, k-1 \} \mid \whx_j = 0 \},
   \end{equation}
   and, if in addition ${k \ge 2}$, then we have
   \begin{equation}
     \label{eq:17}
     B(k)
   =
   1 + \sharp \{ j \in \{0, \ldots, k - 2 \} \mid \whx_j \neq \whx_{j + 1} \}.
   \end{equation}
   Thus, for every~$k$ in~$\N$ the number~$B(k)$ is equal to the number of maximal blocks of~$0$'s, $1^+$'s, and~$1^-$'s in the sequence~$(\whx_j)_{j = 0}^{k-1}$.

   For~$\uvarsigma$ in~$\signs$, define the itinerary~$\iota(\uvarsigma)$ in~${\{0, 1\}}^{\N_0}$ for each~$j$ in~$\N_0$ by
   \begin{equation}
     \label{eq:9}
     \iota(\uvarsigma)_j
     =
     \begin{cases}
       0
       &
       \text{if $\whx(\uvarsigma)_j = 0$};
       \\
       1
       &
       \text{if~$\whx(\uvarsigma)_j = 1^+$};
       \\
       0
       &
       \text{if~$\whx(\uvarsigma)_j = 1^-$ and~$j$ is even};
       \\
       1
       &
       \text{if~$\whx(\uvarsigma)_j = 1^-$ and~$j$ is odd}.
     \end{cases}
   \end{equation}
   Note that the first~$q$ entries of~$\iota(\uvarsigma)$ are equal to~$0$, and that the map ${\signs \to \{0, 1\}^{\N_0}}$ given by ${\uvarsigma \mapsto \iota(\uvarsigma)}$ is continuous.

   \subsection{The 2~variables series}
   \label{ss:estimating postcritical series}
   Given a real number~$\xi$, define for each~$k$ in~$\N_0$ and each~$(\tau, \lambda)$ in~$[0, +\infty) \times [0, +\infty)$ the number
   \[
     \pi_k^{\pm}(\tau,\lambda)
     \=
     2^{- \lambda k - \tau N(k) \pm \xi \tau B(k)}.
   \]

   In this subsection we fix a uniform family of quadratic-like maps~$\sF$, let~$\Delta_2$ be the constant given by Lemma~\ref{l:distortion to central 0}, and recall that ${\Delta_2 > 1}$.
   For each integer~$n$ satisfying ${n \ge 5}$, and each~$f$ in~$\cK_n(\sF)$, put
   \begin{equation}
     \label{d:theta n}
     \theta(f)
     \=
     \left| \frac{Dg_f(p(f))}{Dg_f(\wtp(f))} \right|^{1/2}.
   \end{equation}
   Note that the condition~$\theta(f) > 1$ is equivalent to~$\chi_f(p(f)) > \chi_f(\wtp(f))$.
   When this holds, define
   $$ \xi(f)
   \=
   \frac{\log \Delta_2}{2 \log \theta(f)}. $$

   The purpose of this subsection is to prove the following lemma.

   \begin{lemm}
     \label{l:2 variables functions}
     Let~$\Delta_1$ be the constant given by Lemma~\ref{l:landing to central derivative}, let~$q$ and~$\Xi$ be nonnegative integers satisfying~\eqref{eq:5}, and let~$(\iota(\uvarsigma))_{\uvarsigma \in \signs}$ be the family of itineraries in~$\{0, 1\}^{\N_0}$ defined in~\S\ref{ss:itineraries} for these choices of~$q$ and~$\Xi$.
     Then, for every~$\uvarsigma$ in~$\signs$, every integer~$n$ satisfying ${n \ge 6}$, and every~$f$ in~$\cK_n(\sF)$ such that
     \begin{equation*}
       \iota(f) = \iota(\uvarsigma)
       \text{ and }
       \chi(\whp(f)) = \chi(\wtp(f)),
     \end{equation*}
     we have
     \begin{equation}\label{e:chicrit}
       \chicritf
       =
       \frac{1}{3} \log |Dg_f(\wtp(f))|.
     \end{equation}
     If in addition 
     $$ \chi_f(p(f)) > \chi_f(\wtp(f)), $$
     then the following property holds for every number~$\xi$ satisfying ${\xi \ge \xi(f)}$.
     For every~$k$ in~$\N_0$, and all real numbers~$t$ and~$\delta$ satisyfing ${t > 0}$ and ${\delta \ge 0}$, we have
     \begin{multline}
       \label{eq:2 variables series}
       \Delta_1^{- \frac{t}{2}} \exp(-n\delta)\left( \frac{\exp(\chicritf)}{|Df(\beta(f))|}
       \right)^{\frac{t}{2}n }
       \pi_k^- \left( \frac{\log \theta(f)}{\log 2} t, \frac{3 \delta}{\log 2}
       \right)
       \\ 
       \begin{aligned}
         & \le
         \exp \left(-(n+3k) \left(- t \frac{\chicritf}{2} + \delta \right) \right)
         |Df^{n+3k}(f(0))|^{-\frac{t}{2}}
         \\ & \le
         \Delta_1^{\frac{t}{2}} \exp(-n\delta) \left( \frac{\exp (\chicritf)}{|Df(\beta(f))|}
         \right)^{\frac{t}{2} n }
         \pi_k^+ \left( \frac{\log \theta(f)}{\log 2} t, \frac{3 \delta}{\log 2}
         \right).
       \end{aligned}
     \end{multline}
   \end{lemm}

   \proof
   Put~$\whc \= f^{n + 1}(0)$.
   For every~$k$ in~$\N$ and every~$j$ in~$\{0,1,2\}$, we have by the chain rule
   \begin{align*}
     Df^{3k + j} (f(0))
     & =
       Df^j ((f^{3k})(\whc)) \cdot Df^{3k} (\whc) \cdot Df^n(f(0))
     \\ & =
          Df^j (g_f^{k}(\whc)) \cdot Dg_f^k (\whc) \cdot Df^n(f(0)).
   \end{align*}
   Since $|Df^j ((g_f^{k})(\whc))|$ is bounded independently of~$k$ and~$j$, we have
   \begin{equation}
     \label{eq:intermediate chicrit}
     \chicritf
     = 
     \liminf_{m\rightarrow +\infty} \frac{1}{m} \log |Df^m(f(0))|
     =
     \frac{1}{3} \liminf_{k\rightarrow +\infty}\frac{1}{k} \log |Dg_f^k (\whc)|.
   \end{equation}
   On the other hand, by Lemma~\ref{l:distortion to central 0}, the hypothesis~$\chi(\whp(f)) = \chi(\wtp(f))$, \eqref{eq:16}, \eqref{eq:17}, and the fact that the maximal blocks of $1^-$'s in~$\whx(\uvarsigma)$ have even length, we have that 
   for each integer~$k$ in $\N$,
   \begin{equation*}
     \Delta_2^{-B(k)}
     \le
     \frac{|Dg_f^k(\whc)|} {|Dg_f(\wtp (f))|^{k-N(k)} |Dg_f( p(f))|^{N(k)}}
     \le
     \Delta_2^{B(k)}.
   \end{equation*}
   Taking logarithms yields
   \begin{align*}
     - B(k)\log \Delta_2 + N(k) \log \frac{|Dg_f(p(f))|}{|Dg_f(\wtp(f))|}
     & \le 
       \log |Dg_f^k(\whc)| - k \log |Dg_f(\wtp(f))|
     \\ & \le
          B(k) \log \Delta_2 + N(k) \log \frac{|Dg_f(p(f))|}{|Dg_f(\wtp(f))|}.
   \end{align*}
   Since for each~$k$ in~$\N$ we have $B(k) \le 2 N(k) + 1$, by~\eqref{eq:6} we conclude that
   $$ \lim_{k\rightarrow +\infty}\frac{1}{k} \log |Dg_f^k (\whc)|
   =
   \log |Dg_f(\wtp(f))|. $$
   Combined with~\eqref{eq:intermediate chicrit}, this completes the proof of~\eqref{e:chicrit}.

   In the case where~$k = 0$, the chain of inequalities~\eqref{eq:2 variables series} is given by Lemma~\ref{l:landing to central derivative}.
   Fix~$k$ in~$\N$ and a number~$t$ satisfying ${t > 0}$.
   Using
   \[
     Df^{n+3k}(f(0)) = Dg_f^k (f^{n + 1}(0)) \cdot Df^n(f(0))
   \]
   and Lemmas~\ref{l:landing to central derivative} and~\ref{l:distortion to central 0}, the hypothesis~$\chi(\whp(f)) = \chi(\wtp(f))$, \eqref{eq:16}, \eqref{eq:17}, and the fact that the maximal blocks of~$1^-$'s in~$\whx(\uvarsigma)$ have even length, we have
   \begin{equation}
     \label{e:postcritical derivative}
     \begin{split}
       \Delta_1^{-t} \theta(f)^{- 2t N(k)}  \Delta_2^{- tB(k)}
       & \le
       \frac{|Df^{n+3k}(f(0))|^{-t}}{|Dg_f(\wtp(f))|^{-tk} |Df(\beta(f))|^{-tn}}
       \\ & \le
       \Delta_1^{t} \theta(f)^{- 2t N(k)} \Delta_2^{tB(k)}.
     \end{split}
   \end{equation}
   Since by~\eqref{e:chicrit} we have
   \begin{equation*}
     \exp((n+3k)t\chicritf)
     =
     \exp(nt\chicritf) |Dg_f(\wtp(f))|^{tk},
   \end{equation*}
   if we multiply each term in the chain of inequalities~\eqref{e:postcritical derivative} by
   $$ \left( \frac{\exp(\chicritf)}{|Df(\beta(f))|} \right)^{ t n}, $$
   then we get
   \begin{multline*}
     \Delta_1^{-t} \left( \frac{\exp(\chicritf)}{|Df(\beta(f))|} \right)^{ t n} \theta(f)^{-2t N(k)} \Delta_2^{- tB(k)}
     \\
     \begin{aligned}
       &\le 
       \exp((n+3k)t\chicritf) |Df^{n+3k}(f(0))|^{-t}
       \\ &\le 
       \Delta_1^{t} \left( \frac{\exp(\chicritf)}{|Df(\beta(f))|} \right)^{ t n} \theta(f)^{-2t N(k)} \Delta_2^{tB(k)}.
     \end{aligned}
   \end{multline*}
   Taking square roots, and then by multiplying by~$\exp(-(n + 3k) \delta)$ in each of the terms of the chain of inequalities above, we obtain
   \begin{multline*}
     \Delta_1^{-t/2} \exp(- n \delta) \left( \frac{\exp(\chicritf)}{|Df(\beta(f))|} \right)^{\frac{t}{2} n}
     \exp(- 3k \delta) \theta(f)^{-t N(k)} \Delta_2^{- tB(k)/2}
     \\
     \begin{aligned}
       & \le
       \exp \left(- (n+3k)\left(- t \frac{\chicritf}{2} + \delta\right)\right) |Df^{n+3k}(f(0))|^{-\frac{t}{2}}
       \\ & \le
       \Delta_1^{t/2}\exp(- n \delta) \left( \frac{\exp(\chicritf)}{|Df(\beta(f))|} \right)^{\frac{t}{2} n} \exp(- 3k \delta)  \theta(f)^{-t N(k)} \Delta_2^{tB(k)/2}.
     \end{aligned}
   \end{multline*}
   Together with our hypothesis~$\xi \ge \xi(f)$, and our definition of~$\pi_k^{\pm}$, this implies the desired chain of inequalities.
   \endproof

   \subsection{Proof of the Main Theorem}
   \label{ss:proof of Main Theorem}
   Put~$\upsilon \= \frac{1}{4} \log 2$, let~$R > 0$ be given, and let~$K_1$, $\kappa_1$, and~$n_1$ be given by Proposition~\ref{p:transporting Peierls}.
   We prove that the Main Theorem holds with~$K_0 = K_1$.
   Let~$\sF$ be a uniform family of quadratic-like maps with constants~$K_1$ and~$R$ that is admissible.
   By Proposition~\ref{p:transporting Peierls}, for every~$n \ge n_1$, every~$f$ in~$\cK_n(\sF)$ satisfies the Geometric Peierls Condition with constants~$\kappa_1$ and~$\upsilon$, and we have
   \begin{equation}
     \label{eq:10}
     \chi_f(\beta(f)) > \chicritf + 2 \upsilon.
   \end{equation}
   Taking~$n_1$ larger if necessary, assume that for every~$n \ge n_1$ there is a continuous function~$s_n \colon \cK_n \to \cK_n(\sF)$ such that~$c \circ s_n$ is the identity, and that~\eqref{eq:1} holds for every~$f$ in~$s_n(\cK_n)$.

   Let~$\Delta_1$, $\Delta_2$, $C_5$, $\upsilon_1$ and~$\Delta_3$ be the constants given by Lemmas~\ref{l:landing to central derivative}, \ref{l:distortion to central 0}, \ref{l:contractions}, and~\ref{l:bounded distortion to nice}, respectively.
   Moreover, let~$n_3$ and~$C_6$ the constants given by Proposition~\ref{p:improved MS criterion}, and $\kappa = \kappa_1$, let~$n_4$ and~$C_9$ be given by Proposition~\ref{p:estimating Z_1 by the postcritical series}, and let~$n_{\&} \ge \max \{6,n_3, n_4\}$ be sufficiently large so that
   \begin{equation}
     \label{e:choice of n I}
     \exp(n_{\&} \upsilon) \ge \Delta_1^{\frac{1}{2}} C_6 \left(2 + \Delta_2 \right).
   \end{equation}

   \subsubsection{The subfamily}
   \label{sss:parameter-selection}
   In this subsection we define the family $(f_{\uvarsigma})_{\uvarsigma\in \{+,-\}^{\N}}$, as in the statement of the Main Theorem.

   Fix an integer~$n$ satisfying ${n \ge n_{\&}}$, let~$c_{\&}$ in~$\cK_n$ be such that the itinerary~$\iota(c_{\&})$ is the constant sequence equal to~$0$ (Proposition~\ref{p:ps}), and put~$f_{\&} \= s_n(c_{\&})$.
   By~\eqref{eq:1} we have~$\theta(f_{\&}) > 1$, so there is~$r_{\&} > 0$ such that for every~$c$ in~$B(c_{\&}, r_{\&}) \cap \cK_n$ the number~$\theta(s_n(c))$ is defined, and depends continuously on~$c$.
   Reducing~$r_{\&}$ if necessary, assume that for all~$c$ and~$c'$ in~$B(c_{\&}, r_{\&}) \cap \cK_n$ we have~$\theta(s_n(c)) \le \theta(s_n(c'))^2$.
   By Proposition~\ref{p:ps} it follows that there is an integer~$q_{\&} \ge 0$ such that the set
   $$ \{ c \in \cK_n \mid \text{ for every~$j$ in~$\{0, \ldots, q_{\&} \}$, $\iota(c)_j = 0$} \} $$
   is a compact set contained in~$B(c_{\&}, r_{\&})$.
   On the other hand, by~\eqref{eq:1} for each~$c$ in~$\cK_n$ we can define the number~$\xi(s_n(c))$ as in~\S\ref{ss:estimating postcritical series}.
   Since~$s_n$ is continuous, it follows that this number depends continuously on~$c$ in~$\cK_n$, so
   $$ \sup_{c \in \cK_n \cap B(c_{\&}, r_{\& })} \xi(s_n(c)) < +\infty. $$
   Denote the supremum on the left side by~$\xi$.

   Put ${\Xi \= \lceil 2\xi \rceil + 1}$ as in~\S\ref{ss:estimating zero-temperature 2 variables series}, and let~$q$ be an integer satisfying ${q \ge q_{\&}}$ and \eqref{eq:5}.
   Moreover, let
   \begin{equation*}
     (\whx(\uvarsigma))_{\uvarsigma \in \signs}
     \text{ and }
     (\iota(\uvarsigma))_{\uvarsigma \in \signs}
   \end{equation*}
   be given by~\eqref{eq:8} and~\eqref{eq:9}, respectively, in~\S\ref{ss:itineraries}.
   Given~$\uvarsigma$ in~$\signs$, let~$c(\uvarsigma)$ in~$\cK_n$ be the unique parameter such that ${\iota(f_{c(\uvarsigma)}) = \iota(\uvarsigma)}$ (Proposition~\ref{p:ps}), and put ${\fs \= s_n(c(\uvarsigma))}$.
   Note that the function~$\uvarsigma \mapsto \fs$ so defined is continuous.
   On the other hand, since for each~$\uvarsigma$ in~$\signs$ the first~$q$ entries of~$\iota(\uvarsigma)$ are equal to~$0$ and we have ${q \ge q_{\&}}$ and ${\iota(\fs) = \iota(\uvarsigma)}$, we conclude that the parameter~$c(\uvarsigma)$ is in~$B(\lambda_{\&}, r_{\&})$.
   So, for all~$\uvarsigma$ and~$\uvarsigma'$ in~$\signs$ we have
   \begin{equation}
     \label{e:theta distortion}
     \theta(\fs) \le \theta(f_{\uvarsigma'})^2.
   \end{equation}

   \subsubsection{Pressure estimates, and the existence of equilibria}
   \label{sss:pressure estimates}
   The purpose of this subsection is to prove item~1 of the Main Theorem, and at the same time to estimate for each~$\uvarsigma$ in~$\signs$ the pressure functions of~$\fs$ at large values of~$t$.
   That for each~$\uvarsigma$ in~$\signs$ the interval map~$\fs|_{I(\fs)}$ is topologically exact follows from the fact that this map is not renormalizable, see~\cite[\S3]{CorRiv13} for details.
   Thus, to prove item~1 of the Main Theorem we only need to prove the assertions about equilibrium states.

   Let~$N \colon \N_0 \to \N_0$ and~$B \colon \N_0 \to \N_0$ be the functions defined in~\S\ref{ss:itineraries} for our choices of~$\Xi$ and~$q$ in~\S\ref{sss:parameter-selection}.
   Let~$\Pi^{\pm}$ be the~2 variables series, and for each~$s$ in~$\N_0$, let~$I_s^{\pm}$ and~$J_s^{\pm}$ be the series defined in~\S\ref{ss:estimating zero-temperature 2 variables series}.
   They satisfy
   $$ \Pi^{\pm} = \sum_{k = 0}^{+\infty} \pi_k^{\pm},
   I_s^{\pm} = \sum_{k \in I_s} \pi_k^{\pm},
   \text{ and }
   J_s^{\pm} = \sum_{k \in J_s} \pi_k^{\pm}. $$
   Furthermore, for each nonnegative real number~$s$ put~$\lambda(s) \= |J_s|^{-1}$ as in~\S\ref{ss:estimating zero-temperature 2 variables series}.

   Let~$A \colon \signs \to (0, +\infty)$ be the continuous function defined by
   $$ A(\uvarsigma) \= \frac{4 \log 2}{\log \theta(\fs)}, $$
   and define
   \begin{equation*}
     A_{\sup}
     \=
     \sup_{\uvarsigma \in \signs} A(\uvarsigma),     
   \end{equation*}
   and
   \begin{equation*}
   \eta_0
   \=
     \sup \left\{ \exp(\chi_{\fs}(\beta(\fs)) - \chi_\text{crit}(\fs)) \mid \uvarsigma \in \signs \right\}.
   \end{equation*}
   Moreover, let~$t_{\&}$ be a sufficiently large number so that
   \begin{multline}
     \label{e:choice of t I}
     t_{\&}
     \ge
     \frac{2 \log 2}{\upsilon},
     t_{\&}
     \ge
     \left( \inf_{\uvarsigma \in \signs} \chi_{\fs}(\wtp(\fs)) \right)^{-1},
     t_{\&}
     \ge
     \frac{25}{2} A_{\sup},
     \\ 2^{\left( \frac{4}{A_{\sup}} \right)^2 t_{\&}} 
     \ge
     2^n \Delta_1^{\frac{1}{2}} C_6 \eta_0^{\frac{n}{2}},
     \text{and }
     \frac{\log C_5}{t_{\&}^2}
     \le
     \upsilon_1 \frac{8}{A_{\sup}^2}.
   \end{multline}

   For the rest of this subsection we fix~$\uvarsigma$ in~$\signs$, and put
   $$ f \= \fs,
   \wtp \= \wtp(\fs),
   \whp \= \whp(\fs),
   $$
   $$ P^{\R} \= P_{f_{\uvarsigma}}^{\R},
   \sP^{\R} \= \sP_{f_{\uvarsigma}}^{\R},
   P\= P_{f_{\uvarsigma}},
   \text{ and }
   \sP \= \sP_{f_{\uvarsigma}}. $$
   Moreover, fix a real number~$t$ satisfying ${t \ge t_{\&}}$, and put
   $$ \tau
   \=
   \frac{4 }{A(\uvarsigma)} t, $$
   $$ P^+
   \=
   - t \frac{\pchicrit}{2} + \frac{\log 2}{3} \lambda(\tau - 1),
   \quad \text{and} \quad
   P^-
   \=
   - t \frac{\pchicrit}{2} + \frac{\log 2}{3} \lambda(\tau). $$
   Note that
   \begin{equation*}
     \tau
     =
   \frac{\log \theta(f)}{\log 2} t,
   \end{equation*}
that by~\eqref{e:chicrit} we have~$\pchicrit = \chi_{f}(\wtp)$, and that by~\eqref{e:choice of t I} we have
\begin{equation*}
  \tau \ge 50
  \text{ and }
  P^- < P^+ < 0.
\end{equation*}
   Moreover, by~\eqref{e:theta distortion} we have
   $$
   2^{\tau \xi} = \theta(f)^{t\xi}\le \theta(f)^{2t\xi(f)} = \Delta_2^t. 
   $$
   Combined with~\eqref{eq:1}, Lemma~\ref{l:2 variables functions} with ${\delta = \frac{\log 2}{3} \lambda(\tau - 1)}$, \eqref{eq:10}, \eqref{e:choice of n I}, and item~1 of Lemma~\ref{l:zero-temperature first floor}, this implies
   \begin{multline}
     \label{e:postcritical series upper bound}
     \sum_{k = 0}^{+\infty} \exp \left(-(n+3k) P^+ \right) |Df^{n+3k}(f(0))|^{-\frac{t}{2}}
     \\
     \begin{aligned}
       & \le \Delta_1^{\frac{t}{2}} \left( \frac{\exp (\pchicrit)}{|Df(\beta(f))|} \right)^{\frac{t}{2} n }
       \Pi^+ \left( \tau, \lambda(\tau - 1) \right).
       \\ & \le
       \left( \Delta_1^{\frac{1}{2}} \exp(- n \upsilon) \right)^t (2 + 2^{\tau \xi})
       \\ & \le
       \left( \Delta_1^{\frac{1}{2}} \exp(- n \upsilon)  \left( 2 + \Delta_2 \right) \right)^t
       \\ & \le
       C_6^{-t}.
     \end{aligned}
   \end{multline}
   Together with item~2 of Proposition~\ref{p:improved MS criterion}, this implies
   \begin{equation}
     \label{e:pressure upper bound}
     P^{\R}(t) \le P(t) \le P^+
     \quad \text{and} \quad
     \sP^{\R}(t, P^+) \le \sP(t, P^+) < 0.  
   \end{equation}

   On the other hand, by~\eqref{eq:1}, Lemma~\ref{l:2 variables functions} with ${\delta = \frac{\log 2}{3} \lambda(\tau) \le \log 2}$, the definition of $\eta_0$, \eqref{e:choice of t I}, and item~2 of Lemma~\ref{l:zero-temperature first floor}, we have
   \begin{multline}
     \label{e:postcritical series lower bound}
     \sum_{k = 0}^{+\infty} \exp \left(-(n+3k) P^- \right) |Df^{n+3k}(f(0))|^{-\frac{t}{2}}
     \\
     \begin{aligned}
       & \ge \Delta_1^{- \frac{t}{2}} \exp \left( - n \frac{\log 2}{3} \lambda(\tau) \right) \left( \frac{\exp (\pchicrit)}{|Df(\beta(f))|}
       \right)^{\frac{t}{2} n }
       \Pi^- \left( \tau, \lambda(\tau) \right).
       \\ & \ge
       2^{-n} \left( \Delta_1^{\frac{1}{2}} \eta_0^{\frac{n}{2}} \right)^{-t} 2^{\tau^2}
       \\ & 
       \ge \left( 2^n\Delta_1^{\frac{1}{2}}  \eta_0^{\frac{n}{2}} 2^{- \left(\frac{4}{A(\uvarsigma)}\right)^2 t_{\&}}
       \right)^{-t}
       \\ &
       \ge
       C_6^t.
     \end{aligned}
   \end{multline}
   Then item~1 of Proposition~\ref{p:improved MS criterion} implies
   \begin{equation}
     \label{e:pressure lower bound}
     P(t) \ge P^{\R}(t) > P^-
     \quad \text{and} \quad
     \sP(t, P^-) \ge \sP^{\R}(t, P^-) > 0.  
   \end{equation}

   We proceed to prove the existence and uniqueness of equilibrium states.
   Combining~\eqref{e:pressure upper bound} and~\eqref{e:pressure lower bound}, we have that the number~$\chiinfR(f)$ defined in the statement of Proposition~\ref{p:critical line} satisfies
   $$ \chiinfR(f)
   =
   - \lim_{t\to +\infty} \frac{P^{\R}(t)}{t}
   =
   \frac{\chicritf}{2}. $$
   Similarly,
   $$ \chiinfC(f)
   \=
   \inf \left\{ \int \log |Df| \dd \mu \mid \mu \in \sM_f \right\}
   =
   \frac{\chicritf}{2}. $$
   Using~\eqref{e:pressure lower bound} again, we conclude that for every~$t > 0$ we have
   \begin{equation*}
     P^{\R}(t) > - t \chiinfR (f)
     \quad \text{and} \quad 
     P(t) > - t \chiinfC (f).
   \end{equation*}
   The existence and uniqueness of equilibrium states follows from~\cite[Theorem~A]{PrzRiv19} in the real case.
   In the complex case it is proved in~\cite[Main Theorem]{PrzRiv11} for rational maps, and the proof applies without changes to quadratic-like maps.
   This completes the proof of item~1 of the Main Theorem.

   \subsubsection{Temperature dependence}
   \label{sss:temperature dependence}
   In this subsection we complete the proof of the Main Theorem by showing item~2.
   We give the proof in the complex setting; except for the obvious notational changes, it applies to the real case without modifications.
   We adopt the notation introduced in the previous subsections.

   Fix a number~$t$ satisfying ${t \ge t_{\&}}$, and let~$m_0$ be the integer in~$\N$ such that~$t$ is in~$(A(\uvarsigma) (m_0 - 1), A(\uvarsigma) m_0]$.
   Note that~$\tau \ge 50$, and that the integer~$\tau_0 \= \lceil \tau \rceil$ 
   satisfies $4m_0 - 3 \le  \tau_0 \le 4 m_0$.
   On the other hand, by~\eqref{e:pressure upper bound} and~\eqref{e:pressure lower bound} there is~$s^{\C}$ in~$[\tau - 1, \tau]$ such that
   \begin{equation}
     \label{e:pressure control}
     P(t) = - t \frac{\chicritf}{2} + \frac{\log 2}{3} \lambda(s^{\C}).
   \end{equation}
   Put~$s_0 \= \lceil s^{\C} \rceil$ and note that~$s_0$ is either equal to~$\tau_0 - 1$ or~$\tau_0$.

   We first prove that the hypotheses of Proposition~\ref{p:equilibria} are satisfied for this value of~$t$.
   By~\eqref{eq:1}, Lemma~\ref{l:2 variables functions} with ${\delta = \frac{\log 2}{3} \lambda(\tau)}$, and item~1 of Lemma~\ref{l:zero-temperature tower}, we have
   \begin{equation}
     \label{e:time postcritical series}
     \sum_{k = 0}^{+\infty} k \cdot \exp \left(-(n+3k)P^- \right) |Df^{n+3k}(f(0))|^{-\frac{t}{2}}
     <
     +\infty.  
   \end{equation}
   In particular, this implies that the sum in~\eqref{e:postcritical series lower bound} is finite, so by item~1 of Proposition~\ref{p:improved MS criterion} we have~$\sP(t, P^-) < +\infty$.
   This implies that~$\sP(t, \cdot)$ is continuous and strictly decreasing on~$[P^-, +\infty)$, so by~\eqref{e:pressure upper bound}, \eqref{e:pressure lower bound}, and Proposition~\ref{p:Bowen type formula} we have the second equality in~\eqref{e:Bowen formula}.
   Finally, combining~\eqref{e:time postcritical series}, and item~3 of Proposition~\ref{p:improved MS criterion}, we obtain that the second sum in~\eqref{e:time integrability}, with~$P_f(t)$ replaced by~$P^-$, is finite.
   In view of~\eqref{e:pressure lower bound}, this implies that the second sum in~\eqref{e:time integrability} is finite.
   This completes the proof that the hypotheses of Proposition~\ref{p:equilibria} are satisfied.

   Let~$\trho$, and~$\hrho$ be the measures given by Proposition~\ref{p:equilibria}.
   Moreover, put ${\fD \= \fD_f}$, ${F \= F_f}$, and for every~$k$ in~$\N_0$ put ${\fD_k \= \fD_{f, k}}$.
   Recall from~\S\ref{ss:itineraries}, that~$a_0 = 1$ and that for every~$s$ in~$\N_0$ we have
   $$ I_s = [a_s, b_s)
   \quad \text{and} \quad
   J_s = [b_s, a_{s + 1}). $$
   Note that by item~(b) of Lemma~\ref{l:the itinerary}, and the hypothesis~$q \ge 50 (\Xi + 1)$, we have~$a_{s + 1} - b_s = |J_s| \ge (s + 1)^2$.
   For each integer~$\varsigma$ in~$[\tau_0 - 3, s_0]$ put
   $$ \hrho_{\varsigma}'
   \=
   \sum_{k = b_{\varsigma} + {\varsigma}^2}^{a_{{\varsigma} + 1} - 1} \sum_{j = n + 3b_{\varsigma}-2}^{n + 3(k + 1 - \varsigma^2)} \sum_{W \in \fD_{k}} (f^j)_* \left(\trho|_W \right), $$
   and put
   $$ \hrho''
   \=
   \sum_{k \in J_{s_0}} \sum_{j = n + 3b_{s_0 - 1}-2}^{n + 3(a_{s_0} + 1 - s_0^2)} \sum_{W \in \fD_{ k}} (f^j)_* \left(\trho|_W \right). $$

   In part~1 we estimate the total mass of the measure
   $$ \hrho' \= \left( \sum_{\varsigma = \tau_0 - 3}^{s_0} \hrho_{\varsigma}' \right) +\hrho'' $$
   from below, and in part~2 we show that the total mass of~$\hrho - \hrho'$ is small in comparison to that of~$\hrho'$. 
   In part~3 we complete the proof of item~2 of the Main Theorem by showing that~$\hrho'$ is supported on a small neighborhood of the orbit of~$\wtp$ or~$\whp$.

   The following series defined in~\S\ref{sec:estim-weight-2}, are used in parts~1 and~2 below: $\tPi^{\pm}$, and for each~$s$ in~$\N_0$, the series~$\tPi^{\pm}$, $\tI_s^+$, $\tJ_s^+$, and~$\hJ_s^{\pm}$.
   They satisfy 
   $$ \tPi^+ = \sum_{k = 0}^{+\infty} k \cdot \pi_k^+,
   \tI_s^+ = \sum_{k \in I_s} k \cdot \pi_k^+,
   \tJ_s^+ = \sum_{k \in J_s} k \cdot \pi_k^+,
   \text{ and }
   \hJ_s^{\pm} = \sum_{k = b_s + s^2}^{a_{s + 1} - 1} (k + 1 - b_s - s^2) \pi_k^{\pm}. $$

   \partn{1}
   To estimate the total mass of~$\hrho'$ from below, put~$\Upsilon_1 \= \Delta_3 C_9 \Delta_1^{\frac{1}{2}} \eta_0^{\frac{n}{2}} 2^n$, and for each~$\varsigma$ in~$[\tau_0 - 3, s_0]$, put
   $$ H_{\varsigma}
   \=
   \{ k \in \N_0 \mid b_\varsigma + \varsigma^2 \le k \le a_{\varsigma + 1} - 1 \}. $$
   By item~1 of Proposition~\ref{p:estimating Z_1 by the postcritical series}, \eqref{eq:1}, Lemma~\ref{l:2 variables functions} with ${\delta = \frac{\log 2}{3}\lambda(s^{\C}) \le \log 2}$, the definition of~$\eta_0$, \eqref{eq:3}, \eqref{eq:4}, and~\eqref{e:pressure control}, we have
   \begin{multline}
     \label{e:central tower section I}
     |\hrho'|
     \\
     \begin{aligned}
       & \ge
       \sum_{\varsigma = \tau_0 - 3}^{s_0} |\hrho_{\varsigma}'|
       \\ & =
       \sum_{\varsigma = \tau_0 - 3}^{s_0} \sum_{k \in H_{\varsigma}} 3(k + 1 - b_{\varsigma} - {\varsigma}^2) \sum_{W \in \fD_{ k}} \trho (W)
       \\ & \ge
       \left( \Delta_3 C_9 \right)^{-t} \sum_{\varsigma = \tau_0 - 3}^{s_0} \sum_{k \in H_{\varsigma}} 3(k + 1 - b_{\varsigma} - {\varsigma}^2) \exp(-(n + 3k)P(t)) |Df^{n + 3k}(f(0))|^{- t/2}
       \\ & \ge
       \left( \Delta_3 C_9 \Delta_1^{\frac{1}{2}} \left( \frac{|Df(\beta(f))|}{\exp(\chicritf)} \right)^{\frac{n}{2}}\right)^{-t} 2^{-n} \sum_{\varsigma = \tau_0 - 3}^{s_0} \sum_{k \in H_{\varsigma}} 3(k + 1 - b_{\varsigma} - {\varsigma}^2) \pi_k^-(\tau, \lambda(s^{\C}))
       \\ & \ge
       3 \Upsilon_1^{-t} \sum_{\varsigma = \tau_0 - 3}^{s_0} \hJ_{\varsigma}^-(\tau, \lambda(s^{\C})).
     \end{aligned}
   \end{multline}

   \partn{2}
   By~\eqref{eq:3}, \eqref{eq:4}, item~2 of Proposition~\ref{p:estimating Z_1 by the postcritical series}, and~\eqref{e:pressure control}, we have
   \begin{equation*}
     \begin{split}
       \left| \hrho - \hrho' \right|
       & =
       \sum_{\substack{k \in \N_0 \\ k \not \in \bigcup_{\varsigma = \tau_0 - 3}^{s_0} H_{\varsigma}}} \sum_{W \in \fD_{ k}} m_{f}(W) \trho(W)
       \\ & \quad +
       \sum_{\varsigma = \tau_0 - 3}^{s_0 - 1} \sum_{k \in H_{\varsigma}} \sum_{W \in \fD_{ k}} \left(m_f(W) - 3 (k + 2 - b_{\varsigma} - \varsigma^2) \right) \trho(W)
       \\ & \quad +
       \sum_{k \in H_{s_0}} \sum_{W \in \fD_{ k}} \left(m_f(W) - 3 (k + 4 - b_{s_0} - 2s_0^2 + |J_{s_0 - 1}|))\right) \trho(W)
       \\ & \le
       (\Delta_3 C_9)^t \left[ \sum_{\substack{k \in \N_0 \\ k \not \in \bigcup_{\varsigma = \tau_0 - 3}^{s_0} H_{\varsigma}}} (n + 3k + 1) \exp(-(n + 3k) P(t)) |Df^{n + 3k}(f(0))|^{- \frac{t}{2}}
       \right. \\ & \quad +
       \sum_{\varsigma = \tau_0 - 3}^{s_0 - 1} \sum_{k \in H_{\varsigma}} (n + 3(b_{\varsigma} + \varsigma^2)) \exp(-(n + 3k) P(t)) |Df^{n + 3k}(f(0))|^{- \frac{t}{2}}
       \\ & \quad + \left.
         \sum_{k \in H_{s_0}} (n + 3(b_{s_0} + 2s_0^2 - |J_{s_0 - 1}|)) \exp(-(n + 3k) P(t)) |Df^{n + 3k}(f(0))|^{- \frac{t}{2}}
       \right].
     \end{split}
   \end{equation*}
   Thus, if we put~$\Upsilon_2 \= \Delta_3 C_9 \Delta_1^{\frac{1}{2}} \exp(-n \upsilon)$, then by~\eqref{eq:1}, Lemma~\ref{l:2 variables functions} with ${\delta = \frac{\log 2}{3} \lambda(s^{\C})}$, and item~2 of Lemma~\ref{l:zero-temperature tower},
   \begin{equation*}
     \begin{split}
       \left| \hrho - \hrho' \right|
       & \le
       (n + 4) \Upsilon_2^t \left[ \tPi^+ (\tau, \lambda(s^{\C}))
         - \sum_{\varsigma = \tau_0 - 3}^{s_0} \hJ_{\varsigma}^+(\tau, \lambda(s^{\C}))
         - (|J_{s_0 - 1}| - s_0^2) \hJ_{s_0}^+(\tau, \lambda(s^{\C})) \right]
       \\ & \le
       (n + 4) \Upsilon_2^t 2^{-q \tau^2} \sum_{\varsigma = \tau_0 - 3}^{s_0} \hJ_{\varsigma}^-(\tau, \lambda(s^{\C})).    
     \end{split}
   \end{equation*}
   Together with~\eqref{e:central tower section I} and the definitions of~$\tau$ and~$A_{\sup}$, the previous chain of inequalities implies
   \begin{displaymath}
     \left| \hrho - \hrho' \right|
     \le
     3(n + 4) \left(\Upsilon_1 \Upsilon_2 \right)^t 2^{- q \tau^2} |\hrho'|
     \le
     3(n + 4) \left(\Upsilon_1 \Upsilon_2 \right)^t 2^{- q\left(\frac{4}{A_{\sup}}\right)^2 t^2} |\hrho'|.  
   \end{displaymath}
   Thus, if we put
   \begin{displaymath}
     \upsilon_0'
     \=
     \frac{1}{2} q \left( \frac{4}{A_{\sup}} \right)^2 \log 2,
     \text{ and }
     C_0'
     \=
     3(n + 4) \exp \left( \frac{(\log (\Upsilon_1 \Upsilon_2))^2}{4 \upsilon_0'} \right),
   \end{displaymath}
   then
   \begin{equation}
     \label{eq:11}
     \frac{|\hrho - \hrho'|}{|\hrho|}
     \le
     \frac{|\hrho - \hrho'|}{|\hrho'|}
     \le
     C_0' \exp(- \upsilon_0' t^2).  
   \end{equation}

   \partn{3}
   Using the inequality~$\tau \ge 50$, and the definitions of~$\tau$ and~$A_{\sup}$ we have for every~$\varsigma$ in~$[\tau_0-3,s_0]$,
   $$ \varsigma^2
   \ge
   \frac{\tau^2}{2}
   \ge
   \frac{8}{A_{\sup}^2} t^2. $$
   So, if we put~$\upsilon_0''\= \upsilon_1 \frac{8}{A_{\sup}^2} - \frac{\log C_5}{t_{\&}^2} > 0$, then
   \begin{equation}
     \label{e:radius}
     C_5 \exp( - \upsilon_1 \varsigma^2))
     \le
     \exp( - \upsilon_0'' t^2)).
   \end{equation}

   For $\varsigma \in \{+,-\}$, denote by~$\cO^\varsigma$ the forward orbit of~$p^\varsigma$ under~$f$.
   Let~$\varsigma$ in~$[\tau_0 - 3, s_0]$ be given, and put~$m(\varsigma) \= \lceil \varsigma / 4 \rceil$, so that~$4m(\varsigma) - 3 \le \varsigma \le 4m(\varsigma)$.
   For every integer~$j$ such that $j+1$ is in~$J_{\varsigma}$ we have~$\whx(\uvarsigma)_j = 1^{\varsigma(m(\varsigma))}$, so
   $$ \iota(\uvarsigma)_j
   =
   \begin{cases}
     1 & \text{if $\varsigma(m(\varsigma)) = +$};
     \\
     0 & \text{if $\varsigma(m(\varsigma)) = -$ and~$j$ is even};
     \\
     1 & \text{if $\varsigma(m(\varsigma)) = -$ and~$j$ is odd}.
   \end{cases} $$
   Since $b_\varsigma$ is even, for every~$\ell$ in~$[0,a_{\varsigma + 1} - 1-b_\varsigma]$ the points
   $ f^{n + 1 + 3(b_{\varsigma}+\ell-1)}(0)$ and $f^{3\ell}(p^{\varsigma(m(\varsigma))})$ are both in~$Y_f$ or both in~$\tY_f$.
   It follows that 
   $$ P_{f, 3(a_{\varsigma + 1} - 1-b_\varsigma)+ 4}(f^{n + 1 + 3(b_\varsigma-1)}(0))
   =
   P_{f, 3(a_{\varsigma + 1} - 1-b_\varsigma)+ 4}(p^{\varsigma(m(\varsigma))}). $$
   Then, for each integer~$j$ in~$[b_{\varsigma}-1, a_{\varsigma+1} -2]$ we have
   \begin{equation}\label{e:equal puzzle pieces}
     P_{f, 3(a_{\varsigma + 1}-2 -j)+4}(f^{n + 1 + 3j}(0))
     =
     P_{f, 3(a_{\varsigma + 1}-2 -j)+4}(f^{3j}(p^{\varsigma(m(\varsigma))})),
   \end{equation}
   which implies that  for each integer~$k$  in~$J_\varsigma$ and for each integer~$j$ in~$[b_{\varsigma}-1, k -1]$,
   $$ P_{f, 3(k - j) +1}(f^{n + 1 + 3j}(0))
   =
   P_{f, 3(k - j) +1}(f^{3j}(p^{\varsigma(m(\varsigma))})).$$
   Note that by definition of~$\fD_{ k}$, every element~$W$ of~$\fD_{ k}$ is contained in~$P_{f, n + 3k + 2}(0)$, so, if in addition we have $k\ge b_\varsigma+\varsigma^2$ and~$j \le  k  - \varsigma^2$, then by~\eqref{e:radius} and Lemma~\ref{l:contractions} we obtain
   $$ f^{n + 1 + 3j}(W) \cup  f^{(n + 1 + 3j) + 1}(W) \cup 
   f^{(n + 1 + 3j) + 2}(W)
   \subset
   B(\cO^{\varsigma(m(\varsigma))}, \exp( - \upsilon_0'' t^2)). $$
   This proves that~$\hrho'_\varsigma$ is supported on~$B(\cO^{\varsigma(m(\varsigma))}, \exp( - \upsilon_0'' t^2))$.

   On the other hand, for each integer~$k$ in~$J_{s_0}$,  every element $W$ of $\fD_{k}$ is contained in $P_{f,n+3k+2}(0)$, and hence in $P_{f,n+3(a_{s_0}-1)+2}(0)$.
   Thus, by~\eqref{e:equal puzzle pieces} with $\varsigma=s_0-1$, \eqref{e:radius}, and Lemma~\ref{l:contractions}, we have that for every integer $j$ in $[b_{s_0-1}-1, a_{s_0}-s_0^2]$, 
   $$f^{n + 1 + 3j}(W) \cup  f^{(n + 1 + 3j) + 1}(W) \cup f^{(n + 1 + 3j) + 2}(W)
   \subset
   B(\cO^{\varsigma(m(s_0 - 1))} , \exp( - \upsilon_0'' t^2)), $$
   which proves that~$\hrho''$ is supported on~$B(\cO^{\varsigma(m(s_0 - 1))}, \exp( - \upsilon_0'' t^2))$.

   Assume that there are integers~$m$ and~$\whm$ as in the statement of the Main Theorem, so that
   \begin{equation}
     \label{eq:12}
     \whm \ge m \ge 1,
     \uvarsigma(m) = \cdots = \varsigma(\whm),
     \text{ and }
     t \in [A(\uvarsigma) m, A(\uvarsigma) \whm].
   \end{equation}
   Then~$4m \le \tau_0 \le 4 \whm$, so for every~$\varsigma$ in~$[\tau_0 - 3, s_0]$ we have~$\varsigma(m(\varsigma)) = \varsigma(m)$.
   It follows from the considerations above that the measure~$\hrho'$ is supported on~$B(\cO^{\varsigma(m_0)}, \exp( - \upsilon_0'' t^2))$.
   Since the equilibrium state~$\rho_t$ of~$f|_{J(f)}$ for the potential $- t \log |D f|$ is the probability measure proportional to~$\hrho$, by~\eqref{eq:11} we have
   $$ \rho_t(\C \setminus B(\cO^{\varsigma(m_0)}, \exp(- \upsilon_0'' t^2))
   \le
   \frac{|\hrho - \hrho'|}{|\hrho|}
   \le
   C_0' \exp( - \upsilon_0' t^2). $$
   Under our assumption~$t \ge t_{\&}$, this proves item~2 of the Main Theorem with~$\upsilon_0 = \min \{ \upsilon_0', \upsilon_0'' \}$ and~$C_0 = C_0'$.
   In the case where~$t$ is in~$(0, t_{\&})$, it suffices to take the same value of~$\upsilon_0$ and replace~$C_0$ by a constant bounded from below by~$\exp \left( \upsilon_0 t_{\&}^2 \right)$, if necessary.
   The proof of the Main Theorem is thus complete.

   \begin{rema}
     \label{r:periodic concentration}
     Without assuming the existence of~$m$ and~$\whm$ satisfying~\eqref{eq:12}, the measure~$\hrho'$ is supported on~$B(\cO^+ \cup \cO^-, \exp(- \upsilon_0'' t^2))$, and the estimate above gives that for every~$t > 0$ we have
     $$ \rho_t(\C \setminus B(\cO^+ \cup \cO^-, \exp(- \upsilon_0 t^2)))
     \le
     C_0 \exp( - \upsilon_0 t^2). $$
   \end{rema}

   \appendix
   \section{Estimating the~2 variables series}
   \label{s:abstract estimates}
   In this appendix we make some of the main estimates in the proof of the Main Theorem, in an abstract setting that is independent of the rest of the paper.

   After some preliminary estimates in~\S\ref{ss:zero-temperature partition}, the main estimates are given in~\S\ref{ss:estimating zero-temperature 2 variables series}, and~\S\ref{sec:estim-weight-2}.

   \subsection{Preliminary estimates}
   \label{ss:zero-temperature partition}
   Fix a nonnegative integer~$\Xi$, and an integer~$q$ satisfying~$q \ge 50 (\Xi + 1)$.
   For each~$s$ in~$[0,+\infty)$, define
   \begin{equation*}
     a_s
     \=
     2^{qs^3}
     \text{ and }
     b_s
     \=
     2^{qs^3} + q(2s + 1) + \Xi,
   \end{equation*}
   \[
     I_s
     \=
     \left[ a_s, b_s \right)
     \text{ and }
     J_s
     \=
     \left[ b_s, a_{s+1} \right).
   \]
   In the case where~$s$ is an integer, these definitions coincide with those in~\S\ref{ss:itineraries}.
   For~$s$ in~$[0,+\infty)$ that is not necessarily an integer, the interval~$J_s$ is used in the proof of Lemmas~\ref{l:zero-temperature first floor} and~\ref{l:zero-temperature tower} in~\S\ref{ss:estimating zero-temperature 2 variables series}.
   Note that~$|I_0| = q + \Xi$, and that for integer values of~$s$ the intervals~$I_s$ and~$J_s$ form a partition of $[1,+\infty)$.

   Let ${N \colon \N_0 \to \N_0}$ and ${B \colon \N_0 \to \N_0}$ be as in~\S\ref{ss:itineraries}.
   Observe that for every~$s$ in~$\N_0$, we have for every~$k$ in~$J_s$
   \begin{equation}
     \label{eq:N J}
     N(k)
     =
     \sum_{j = 0}^s |I_j|
     =
     \sum_{j = 0}^s (q(2j + 1) + \Xi)
     =
     q(s + 1)^2 + \Xi \cdot (s+1)
   \end{equation}
   and for every~$k$ in~$I_s$
   \begin{equation}
     \label{eq:N I}
     N(k) = k - (2^{qs^3} - 1) + qs^2 + \Xi s.
   \end{equation}

   \begin{lemm}
     \label{l:the itinerary}
     The following properties hold.
     \begin{enumerate}
     \item[(a)]
       For each real number~$s \ge 0$, we have $b_s\le a_{s+1}/2$.
     \item[(b)]
       For each real number~$s \ge 0$, we have $a_{s+1}/2 \le |J_s|$.
     \item[(c)]
       For each real number~$s \ge 1$, we have $b_{s}/a_s \le 5/4$.
     \end{enumerate}
   \end{lemm}
   \begin{proof}
     Item~$(a)$ with~$s = 0$ follows from our hypothesis~$q \ge 50 (\Xi + 1)$.
     For~$s > 0$, it follows from this and from the fact that the derivative of the function
     $$ s \mapsto 2^{q(s + 1)^3  - 1} - (2^{qs^3} + q(2s + 1) + \Xi) $$
     is strictly positive on~$[0, +\infty)$.
     Item~$(b)$ follows easily from item~$(a)$.
     For item~$(c)$ notice that by our hypothesis~$q\ge 50(\Xi+1)$ it is enough to prove that for every $s\ge 1$ we have $2q(s+1)\le (1/4)\cdot 2^{qs^3}$.
     The case~$s = 1$ is given by our hypothesis~$q\ge 50(\Xi+1)$.
     For~$s > 1$, it follows from this and from the fact that the derivative of the function
     $$ s \mapsto 2^{qs^3} - 8q(s+1) $$
     is strictly positive on~$[1, +\infty)$.
   \end{proof}

   \subsection{Estimating the~2 variables series}
   \label{ss:estimating zero-temperature 2 variables series}
   Let~$\xi > 0$ be given, put~$\Xi \= \lceil 2 \xi \rceil + 1$, and let~$q$, $N$, and~$B$ be as in the previous subsection.
   For~$s$ in~$\N_0$ define the following~2 variables series on~$[0,+\infty) \times [0,+\infty),$
   \begin{equation*}
     I_s^{\pm}(\tau, \lambda)
     \=
     \sum_{k\in  I_s} 2^{- \lambda k -\tau N(k) \pm \tau \xi B(k)}
     \quad \text{and} \quad
     J_s^{\pm}(\tau, \lambda)
     \=
     \sum_{k\in J_s} 2^{- \lambda k -\tau N(k) \pm \tau \xi B(k)},
   \end{equation*}
   and put
   $$ \Pi^{\pm}(\tau,\lambda)
   \=
   1 + \sum_{s=0}^{+\infty} I_s^{\pm}(\tau, \lambda) + \sum_{s=0}^{+\infty} J_s^{\pm}(\tau, \lambda). $$
   Note that by~\eqref{e:B} and~\eqref{eq:N J}, for every~$j$ in~$\N_0$ and every $\tau > 0$ we have
   \begin{equation}
     \label{e:J formula}
     J_j^{\pm}(\tau,\lambda)
     =
     2^{- q \tau \cdot (j+1)^2 - (\Xi \mp 2 \xi) \tau \cdot (j+1)} \sum_{k \in J_j} 2^{- \lambda k}.
   \end{equation}
   Moreover, for every real number~$s$ in~$[0,+\infty)$ define
   \[
     \lambda(s) \= \frac{1}{|J_s|}.
   \]
   \begin{lemm}
     \label{l:zero-temperature first floor}
     For every $\tau\ge 2$ the following inequalities hold:
     \begin{enumerate}
     \item[1.]
       $ \Pi^+(\tau,\lambda(\tau - 1))
       \le
       2 + 2^{\tau\xi}. $
     \item[2.]
       $ 2^{\tau^2}
       \le
       \Pi^-(\tau,\lambda(\tau))$.
     \end{enumerate}
   \end{lemm}
   \proof
   \

   \partn{1}
   By~\eqref{e:B}, \eqref{eq:N I}, our hypothesis $\tau\ge 2$, and the inequality~$\Xi - 2 \xi \ge 1$, for every~$\lambda \ge 0$ we have
   \begin{equation}
     \label{eq:Is+1}
     \begin{split}
       \sum_{s=0}^{+\infty} I_s^+(\tau, \lambda)
       & \le
       \sum_{s=0}^{+\infty}  \sum_{m = 1}^{|I_s|} 2^{-\tau(qs^2 + \Xi s + m) +
         \tau \xi \cdot (2s + 1)}
       \\ & \le
       2^{\tau\xi} \sum_{s=0}^{+\infty} 2^{-(\Xi - 2 \xi) \tau s }
       \sum_{m = 1}^{|I_{s}|} 2^{-\tau m}
       \\ & \le
       2^{\tau\xi} \frac{2^{-\tau}}{1-2^{-\tau}} \sum_{s=0}^{+\infty} 2^{- (\Xi - 2\xi) \tau s}
       \\ & \le
       2^{\tau\xi} \frac{ 2^{-\tau}}{\left(1-2^{-\tau}\right)^2}.
       \\ & \le
       2^{\tau\xi}.
     \end{split}
   \end{equation}

   To complete the proof of item~1, note that
   \begin{equation}
     \label{eq:geometrica2}
     \sum_{m = 1}^{+\infty} 2^{- \lambda(\tau - 1) m}
     =
     \frac{1}{2^{\lambda(\tau - 1)} - 1}
     \le
     \frac{1}{\lambda(\tau - 1) \log 2}
     \le
     2 |J_{\tau - 1}|
     \le
     2 \cdot a_{\tau}.
   \end{equation}
   Combined with~\eqref{e:J formula} and the inequality~$\Xi - 2\xi \ge 1$, the previous chain of inequalities implies that for every~$j$ in~$\N_0$
   we have
   \begin{equation*}
     \begin{split}
       J_j^+(\tau,\lambda(\tau - 1))
       & \le 
       2 \cdot 2^{q\tau^3 - q \tau \cdot (j+1)^2 - (\Xi - 2 \xi) \tau \cdot (j + 1)} 
       \\ & \le
       2 \cdot 2^{q\tau^3 - q \tau \cdot (j+1)^2 - \tau \cdot (j + 1)} .
     \end{split}
   \end{equation*}
   We obtain for every integer~$j \ge \lfloor \tau \rfloor \ge \tau - 1$
   $$ J_j^+(\tau, \lambda(\tau - 1))
   \le
   2 \cdot 2^{q\tau^3 - q \tau \cdot (j + 1)^2 - \tau \cdot (j + 1)}
   \le
   2 \cdot 2^{-\tau \cdot (j + 1)}. $$
   To estimate~$J_j^+(\tau, \lambda(\tau - 1))$ for~$j$ in~$\{0, \ldots, \lfloor \tau \rfloor - 1 \}$, note that
   \begin{equation*}
     \sum_{m = 1}^{|J_j|} 2^{- \lambda(\tau - 1) m}
     \le
     |J_j|
     \le
     a_{j+1}.
   \end{equation*}
   Combined with~\eqref{e:J formula} and the inequality~$\Xi - 2\xi \ge 1$, this implies that for every integer~$j$ in~$\{0, \ldots, \lfloor \tau \rfloor - 1\}$ we have
   \begin{equation}
     \label{e:first J's}
     J_j^+(\tau, \lambda(\tau - 1))
     \le
     2^{q(j + 1)^3 - \tau q (j + 1)^2 - (\Xi - 2\xi) \tau \cdot (j + 1)}
     \le
     2^{-\tau \cdot (j + 1)}.
   \end{equation}
   Thus,
   \begin{equation*}
     \sum_{j=0}^{+\infty}J_j^+(\tau, \lambda(\tau - 1))
     \le
     2 \sum_{j=0}^{+\infty} 2^{-\tau \cdot (j + 1)}
     =
     2 \frac{2^{-\tau}}{1 - 2^{- \tau}}
     \le
     2.
   \end{equation*}
   Together with~\eqref{eq:Is+1} this implies the desired inequality.

   \partn{2}
   Put~$\tau_0 \= \lceil \tau \rceil$.
   By item~(b) of Lemma~\ref{l:the itinerary} and the definition of~$\lambda(\tau)$, we have
   \[
     \lambda(\tau)
     =
     |J_{\tau}|^{-1}
     \le
     \frac{2}{a_{\tau+1}}
     \le
     \frac{2}{a_{\tau_0}}.
   \]
   From this inequality, item~(c) of Lemma~\ref{l:the itinerary}  and our hypothesis~$\tau \ge 2$, we obtain
   \begin{equation}
     \label{e:first J term}
     \lambda(\tau) (b_{\tau_0}-1)
     \le
     2\frac{b_{\tau_0}}{a_{\tau_0}}
     \le
     3.
   \end{equation}
   On the other hand, note that for every~$m$ in~$\{1, \ldots, \lfloor |J_\tau| \rfloor \}$ we have~$\lambda(\tau)m \le 1$, so by item~(b) of Lemma~\ref{l:the itinerary} we have
   $$ \sum_{m = 1}^{\lfloor |J_\tau| \rfloor} 2^{- \lambda(\tau) m}
   \ge
   \frac{1}{2} (|J_{\tau}| - 1)
   \ge
   \frac{1}{2^2} |J_{\tau}|
   \ge
   \frac{1}{2^3} 2^{q(\tau + 1)^3}. $$

   Suppose~$\tau \ge \tau_0 - 1/3$.
   In view of~\eqref{e:J formula} and~\eqref{e:first J term}, the previous chain of inequalities implies
   \begin{multline}
     \label{e:main J up}
     \frac{1}{2^6} 2^{q(\tau + 1)^3 - q \tau \cdot (\tau_0 + 1)^2 - (\Xi + 2 \xi) \tau \cdot (\tau_0 + 1)}
     \\
     \begin{aligned}
       & \le
       \frac{1}{2^3} \left( \sum_{m = 1}^{|J_{\tau_0}|} 2^{- \lambda(\tau) m} \right) 2^{- q \tau \cdot (\tau_0 + 1)^2 - (\Xi + 2 \xi) \tau \cdot (\tau_0 + 1)}
       \\ & \le
       \left( \sum_{m = 1}^{|J_{\tau_0}|} 2^{- \lambda(\tau) m} \right) 2^{- \lambda(\tau) (b_{\tau_0} - 1) - q \tau \cdot (\tau_0 + 1)^2 - (\Xi + 2 \xi) \tau \cdot (\tau_0 + 1)}
       \\ & =
       J_{\tau_0}^-(\tau, \lambda(\tau)).
     \end{aligned}
   \end{multline}
   On the other hand, by our assumption~$\tau \ge \tau_0 - 1/3$ we have
   $$ (\tau + 1)^3 - \tau (\tau_0 + 1)^2
   \ge
   (\tau + 1)^3 - \tau \left( \tau + \frac{4}{3} \right)^2
   =
   \frac{\tau^2}{3} + \frac{11\tau}{9} + 1
   \ge
   \frac{\tau^2}{4} + \frac{\tau(\tau_0 + 1)}{12}. $$
   Combined with our hypotheses~$q \ge 50 (\Xi + 1) \ge 25 (\Xi + 2 \xi + 3)$ and~$\tau \ge 2$, and with~\eqref{e:main J up}, this implies item~2 of the lemma when~$\tau \ge \tau_0 - 1/3$.

   To complete the proof, suppose~$\tau \le \tau_0 - 1/3$.
   Similarly as above we have
   \begin{multline}
     \label{e:main J down}
     \frac{1}{2^6} 2^{q\tau_0^3 - q \tau \tau_0^2 - (\Xi + 2 \xi) \tau \tau_0}
     \\
     \begin{aligned}
       & \le
       \frac{1}{2^3} \left(\sum_{m = 1}^{|J_{\tau_0 - 1}|} 2^{- \lambda(\tau_0 - 1) m} \right) 2^{- q \tau \tau_0^2 - (\Xi + 2 \xi) \tau \tau_0}
       \\ & \le
       \frac{1}{2^3} \left(\sum_{m = 1}^{|J_{\tau_0 - 1}|} 2^{- \lambda(\tau) m} \right) 2^{- q \tau \tau_0^2 - (\Xi + 2 \xi) \tau \tau_0}
       \\ & \le
       \left(\sum_{m = 1}^{|J_{\tau_0 - 1}|} 2^{- \lambda(\tau) m} \right) 2^{- \lambda(\tau) (b_{\tau_0-1} - 1) - q \tau \tau_0^2 - (\Xi + 2 \xi) \tau \tau_0}
       \\ & =
       J_{\tau_0 - 1}^-(\tau, \lambda(\tau)).
     \end{aligned}
   \end{multline}
   On the other hand, our assumption~$\tau \le \tau_0 - 1/3$ implies
   $$ \tau_0^3 - \tau\tau_0^2
   \ge
   \frac{\tau_0^2}{3}
   \ge
   \frac{\tau^2}{4} + \frac{\tau\tau_0}{12}. $$
   Combined with our hypotheses~$q \ge 50 (\Xi + 1) \ge 25 (\Xi + 2 \xi + 3)$ and~$\tau \ge 2$, and with~\eqref{e:main J down}, we obtain item~2 of the lemma when~$\tau \le \tau_0 - 1/3$.
   The proof of the lemma is thus complete.
   \endproof

   \subsection{Estimating the weighted~2 variables series}
   \label{sec:estim-weight-2}
   For each~$s$ in~$\N_0$, $\tau > 0$, and~$\lambda \ge 0$ put
   \begin{align*}
     \tI_s^+(\tau, \lambda)
     & \=
       \sum_{k\in  I_s} k \cdot 2^{- \lambda k -\tau N(k) + \tau \xi B(k)},
     \\
     \tJ_s^+(\tau, \lambda)
     & \=
       \sum_{k\in J_s} k \cdot 2^{- \lambda k -\tau N(k) + \tau \xi B(k)},
     \intertext{and}
     \tPi^+(\tau,\lambda)
     & \=
       1 + \sum_{s=0}^{+\infty} \tI_s^+ (\tau, \lambda)
       + \sum_{s=0}^{+\infty} \tJ_s^+ (\tau, \lambda).
   \end{align*}
   Noting that by item~(b) of Lemma~\ref{l:the itinerary} we have~$a_{s + 1} - b_s = |J_s| \ge s^2 + 1$, define for each~$\tau > 0$ and~$\lambda\ge 0$,
   $$ \hJ_s^{\pm}(\tau, \lambda)
   \=
   \sum_{k = b_s + s^2}^{a_{s + 1} - 1} (k + 1 - b_s - s^2) \cdot 2^{- \lambda k -\tau N(k) \pm \tau \xi B(k)}. $$

   \begin{lemm}
     \label{l:zero-temperature tower}
     For each~$\tau \ge 50$, the following properties hold:
     \begin{enumerate}
     \item[1.]
       $ \tPi^+(\tau, \lambda(\tau)) < +\infty $.
     \item[2.]
       Let~$s$ in~$[\tau - 1, \tau]$ be given, put~$s_0 \= \lceil s \rceil$ and~$\tau_0 \= \lceil \tau \rceil$, and note that~$s_0$ is equal to either~$\tau_0 - 1$ or~$\tau_0$.
       Then
       \begin{multline*}
         \Pi^+ (\tau, \lambda(s))
         \le
         \tPi^+ (\tau, \lambda(s)) - \sum_{\varsigma = \tau_0 - 3}^{s_0} \hJ_{\varsigma}^+(\tau, \lambda(s)) - (|J_{s_0 - 1}| - s_0^2) J_{s_0}^+(\tau, \lambda(s))
         \\ \le
         2^{- q\tau^2} \sum_{\varsigma = \tau_0 - 3}^{s_0} \hJ_{\varsigma}^-(\tau, \lambda(s)).
       \end{multline*}
     \end{enumerate}
   \end{lemm}

   The proof of this lemma is given after the following one.
   \begin{sublemma}
     \label{l:zero-temperature core}
     Given~$\tau \ge 50$ and~$s$ in~$[\tau - 1, \tau]$, put~$\tau_0 \= \lceil \tau \rceil$ and~$s_0 \= \lceil s \rceil$.
     Then the following properties hold.
     \begin{enumerate}
     \item[1.]
       $ \hJ_{s_0}^-(\tau, \lambda(s))
       \ge
       2^{2q(s + 1)^3 - q \tau (s_0 + 1)(s_0 + 2)}. $
     \item[2.]
       $ \hJ_{s_0}^-(\tau, \lambda(s))
       \ge
       2^{q\tau^2 (\tau - 4)}. $
     \item[3.]
       For every integer~$\varsigma$ in~$[\tau_0 - 3, s_0 - 1]$ we have
       $$ (b_\varsigma + \varsigma^2) J_{\varsigma}^+(\tau, \lambda(s))
       \le
       \frac{1}{20} 2^{-q\tau^2} \cdot \hJ_{\varsigma}^-(\tau, \lambda(s)). $$
     \item[4.]
       $ (b_{s_0} - |J_{s_0 - 1}| + 2s_0^2) J_{s_0}^+(\tau, \lambda(s))
       \le
       \frac{1}{4} 2^{-q\tau^2} \cdot \hJ_{s_0}^-(\tau, \lambda(s)). $
     \end{enumerate}
   \end{sublemma}
   \begin{proof}
     \

     \partn{1}
     By item~(b) of Lemma~\ref{l:the itinerary} and the definition of~$\lambda(s)$, we have
     \[
       \lambda(s)
       =
       |J_s|^{-1}
       \le
       \frac{2}{a_{s+1}}
       \le
       \frac{2}{a_{s_0}}.
     \]
     On the other hand,
     \begin{equation}
       \label{e:core residue}
       \lambda(s_0) s_0^2
       =
       \frac{s_0^2}{|J_{s_0}|}
       \le
       \frac{s_0^2}{2^{qs_0^3 - 1}}
       \le
       \frac{1}{q s_0}
       \le
       \frac{1}{100}.  
     \end{equation}
     From these~2 inequalities  and item~(c) of Lemma~\ref{l:the itinerary}, we obtain
     \begin{equation}
       \label{e:first block exponent}
       \lambda(s) (b_{s_0} + s_0^2)
       \le
       \frac{2b_{s_0} }{a_{s_0}} + \frac{1}{100}
       \le
       3.
     \end{equation}

     By~\eqref{e:B}, \eqref{eq:N J}, and~\eqref{e:first block exponent}, we have
     \begin{multline}
       \label{eq:tgeometrica2}
       \hJ_{s_0}^-(\tau,\lambda(s))
       \\
       \begin{aligned}
         & =
         2^{-\tau(q(s_0+1)^2 + \Xi \cdot (s_0+1)) - 2 \tau \xi \cdot (s_0+1)}
         \sum_{k = b_{s_0} + s_0^2}^{a_{s_0 + 1} - 1} (k + 1 - b_{s_0} - s_0^2) \cdot 2^{- \lambda(s)k}
         \\ & \ge
         \frac{1}{2^3} 2^{- q \tau \cdot (s_0+1)^2 - (\Xi + 2 \xi) \tau \cdot (s_0+1)}
         \sum_{m = 1}^{|J_{s_0}| - s_0^2} m \cdot 2^{- \lambda(s) m}.
       \end{aligned}
     \end{multline}
     Noticing that for every integer~$N \ge 1$ we have
     $$ \sum_{m = 1}^{N} m \cdot 2^{- \lambda(s) m}
     =
     \frac{2^{\lambda(s)}}{(2^{\lambda(s)} - 1)^2}
     \left( 1 - (N + 1) 2^{- \lambda(s) N} + N 2^{- \lambda(s) (N + 1)} \right), $$
     and that the function
     $$ \eta \mapsto 1 - (N + 1) \eta^N + N \eta^{N + 1} $$
     is decreasing on~$[0, 1]$, we have by~\eqref{e:core residue} and the inequality~$1 - 2^{- \lambda(s_0)} \le \lambda(s_0) \log 2$
     \begin{multline}
       \label{e:core main estimate}
       \sum_{m = 1}^{|J_{s_0}| - s_0^2} m \cdot 2^{- \lambda(s) m}
       \\
       \begin{aligned}
         & \ge
         \frac{2^{\lambda(s)}}{(2^{\lambda(s)} - 1)^2}
         \cdot \left( 1 - (|J_{s_0}| - s_0^2 + 1) 2^{- \lambda(s_0) (|J_{s_0}| - s_0^2)} + (|J_{s_0}| - s_0^2) 2^{- \lambda(s_0) (|J_{s_0}| - s_0^2 + 1)} \right)
         \\ & =
         \frac{2^{\lambda(s)}}{(2^{\lambda(s)} - 1)^2}
         \cdot \left( 1 - 2^{\lambda(s_0) s_0^2 - 1} - 2^{\lambda(s_0) s_0^2 - 1} (|J_{s_0}| - s_0^2) \left( 1 - 2^{- \lambda(s_0)} \right) \right)
         \\ & \ge
         \frac{2^{\lambda(s)}}{(2^{\lambda(s)} - 1)^2} \left( 1 - 2^{\lambda(s_0) s_0^2 - 1}(1 + \log 2) \right) 
         \\ & \ge
         \frac{1}{2^4} \frac{1}{(2^{\lambda(s)} - 1)^2}.
       \end{aligned}
     \end{multline}
     Note that by~$\lambda(s) \le 1$, we have~$2^{\lambda(s)} - 1 \le  \lambda(s)$.
     Thus, together with item~(b) of Lemma~\ref{l:the itinerary}, the previous chain of inequalities implies
     $$ \sum_{m = 1}^{|J_{s_0}| - s_0^2} m \cdot 2^{- \lambda(s) m}
     \ge
     \frac{1}{2^4} \cdot |J_s|^2
     \ge
     \frac{1}{2^6} \cdot 2^{2q (s + 1)^3}. $$
     Together with~\eqref{eq:tgeometrica2}, the inequality $\Xi \ge 2 \xi$, and our hypotheses~$q \ge 50 (\Xi + 1)$ and~$\tau \ge 50$, this implies
     \begin{equation*}
       \begin{split}
         \hJ_{s_0}^-(\tau,\lambda(s))
         & \ge
         \frac{1}{2^{9}} 2^{2q(s + 1)^3 - q \tau \cdot (s_0+1)^2 - (\Xi + 2 \xi) \tau \cdot (s_0+1)}
         \\ & \ge
         2^{2q(s + 1)^3 - q \tau \cdot (s_0+1)(s_0 + 2)}.
       \end{split}
     \end{equation*}
     This proves item~1.

     \partn{2}
     When~$s_0 \le \tau$ we have by our hypothesis~$\tau \ge 50$,
     $$ 2(s + 1)^3 - \tau \cdot (s_0 + 1)(s_0 + 2)
     \ge
     2\tau^3 - \tau(\tau + 1)(\tau + 2)
     \ge
     \tau^2 (\tau - 4). $$
     On the other hand, in the case where~$s_0 \ge \tau$ we have by our hypothesis~$\tau \ge 50$,
     \begin{multline*}
       2(s + 1)^3 - \tau \cdot (s_0 + 1)(s_0 + 2)
       \ge
       2\tau s_0^2 - \tau \cdot (s_0 + 1)(s_0 + 2)
       \\ =
       \tau s_0 (s_0 - 3) - 2\tau
       \ge
       \tau^2(\tau - 3) - 2 \tau
       \ge
       \tau^2 (\tau - 4).
     \end{multline*}
     In all the cases, item~2 follows from item~1.

     \partn{3}
     Let~$\varsigma$ be an integer in~$[\tau_0 - 3, s_0]$ and note that by~\eqref{e:B}, \eqref{e:J formula}, and the definition of~$\hJ_{\varsigma}^-$, we have
     \begin{equation}
       \label{e:core quotient formula}
       \frac{J_{\varsigma}^+(\tau, \lambda(s))}{\hJ_{\varsigma}^-(\tau, \lambda(s))}
       =
       2^{4 \tau \xi \cdot (\varsigma + 1) + \lambda(s) \varsigma^2} \frac{\sum_{m = 1}^{|J_{\varsigma}|} 2^{-\lambda(s) m}}{\sum_{m = 1}^{|J_{\varsigma}| - \varsigma^2} m 2^{-\lambda(s) m}}.
     \end{equation}
     Suppose~$\varsigma$ is in~$[\tau_0 - 3, s_0 - 1]$.
     Then~$\lambda(s) |J_{\varsigma}| \le 1$, so
     \begin{equation*}
       \begin{split}
         \frac{J_{\varsigma}^+(\tau, \lambda(s))}{\hJ_{\varsigma}^-(\tau, \lambda(s))}
         & \le
         2 \cdot 2^{4 \tau \xi \cdot (\varsigma + 1) + \lambda(s) \varsigma^2} \frac{\sum_{m = 1}^{|J_{\varsigma}|} 2^{- \lambda(s) m} }{\sum_{m = 1}^{|J_{\varsigma}| - \varsigma^2} m}
         \\ & \le
         2^2 \cdot 2^{4 \tau \xi \cdot (\varsigma + 1) + \lambda(s) \varsigma^2} \frac{|J_{\varsigma}|}{(|J_{\varsigma}| - \varsigma^2)^2}.
       \end{split}
     \end{equation*}
     Noting that by items~(a) and~(b) of Lemma~\ref{l:the itinerary} we have
     $$ \lambda(s) \varsigma^2 \le \varsigma^2 / |J_\varsigma| \le 1
     \quad \text{and} \quad
     |J_{\varsigma}| \le 2(|J_{\varsigma}| - \varsigma^2), $$
     by item (c) of Lemma~\ref{l:the itinerary}, the inequality~$\Xi \ge 2 \xi$, and our hypotheses~$q \ge 50 (\Xi + 1)$ and~$\tau \ge 50$
     we obtain
     \begin{equation*}
       \begin{split}
         (b_\varsigma + \varsigma^2) \frac{J_{\varsigma}^+(\tau, \lambda(s))}{\hJ_{\varsigma}^-(\tau, \lambda(s))}
         & \le
         2^6 a_s 2^{ 4 \tau \xi \cdot (\varsigma + 1)} |J_{\varsigma}|^{-1}
         \\ & \le
         2^7 \cdot 2^{q \varsigma^3 + 4 \tau \xi \cdot (\varsigma + 1) - q (\varsigma + 1)^3}
         \\ & \le
         \frac{1}{20} 2^{- q \tau^2}.
       \end{split}
     \end{equation*}
     This proves item~3.

     \partn{4}
     By our hypotheses~$q \ge 50(\Xi + 1)$ and $\tau \ge 50$ we have
     \begin{multline*}
       b_{s_0} - |J_{s_0 - 1}| + 2 s_0^2
       =
       2^{q(s_0 - 1)^3} + 2 s_0^2 + 4 q s_0 + 2 \Xi
       \\ \le
       2^{q(s_0 - 1)^3} + qs_0^2
       \le
       2 \cdot 2^{q(s_0 - 1)^3}.
     \end{multline*}
     Thus, by item~(b) of Lemma~\ref{l:the itinerary}, \eqref{e:core main estimate}, \eqref{e:core quotient formula}, and the inequality~$\lambda(s) \le 1$, we have
     \begin{equation*}
       \begin{split}
         (b_{s_0} - |J_{s_0 - 1}| + 2 s_0^2) \frac{J_{s_0}^+(\tau, \lambda(s))}{\hJ_{s_0}^-(\tau, \lambda(s))}
         & \le
         2^5 \cdot 2^{q(s_0 - 1)^3 + 4 \tau \xi \cdot (s_0 + 1) + \lambda(s) s_0^2} (2^{\lambda(s)} - 1)
         \\ & \le
         2^5 \lambda(s) \cdot 2^{q(s_0 - 1)^3 + 4 \tau \xi \cdot (s_0 + 1) + \lambda(s) s_0^2}
         \\ & \le
         2^6 \cdot 2^{- q (s + 1)^3 + q(s_0 - 1)^3 + 4 \tau \xi \cdot (s_0 + 1) + \lambda(s) s_0^2}
       \end{split}
     \end{equation*}
     Using~$\lambda(s) s_0^2 \le s_0^2 |J_s|^{-1} \le 1$, the inequality~$\Xi \ge 2\xi$, and our hypotheses~$q \ge 50 (\Xi + 1)$ and~$\tau \ge 50$, we have
     \begin{equation*}
       \begin{split}
         (b_{s_0} - |J_{s_0 - 1}| + 2 s_0^2) \frac{J_{s_0}^+(\tau, \lambda(s))}{\hJ_{s_0}^-(\tau, \lambda(s))}
         & \le
         2^{- q s_0^3 + q(s_0 - 1)^3 + q \tau \cdot (s_0 + 1)}
         \\ & \le
         2^{- 3 q s_0 (s_0 - 1) + q \tau \cdot (s_0 + 1)}
         \\ & \le
         \frac{1}{4} 2^{- q \tau^2}.
       \end{split}
     \end{equation*}
     This completes the proof of item~4 and of the lemma.
   \end{proof}

   \begin{proof}[Proof of Lemma~\ref{l:zero-temperature tower}]
     \partn{1}
     Note that for every~$s \ge 0$, we have~$\lambda(s) \le 1$ and
     \begin{multline}
       \label{e:linear geometric}
       \sum_{m = 1}^{+\infty} m \cdot  2^{- \lambda(s) m}
       =
       \frac{2^{\lambda(s)}}{\left(2^{\lambda(s)} - 1 \right)^2}
       \le
       \frac{2^{\lambda(s)}}{(\lambda(s) \log 2)^2}
       \\ \le
       2^3 |J_s|^2
       \le
       \left( 2^3 \right) 2^{2q(s+1)^3}.
     \end{multline}
     Together with~\eqref{e:B}, \eqref{eq:N J}, \eqref{eq:N I}, and the inequality~$\Xi - 2\xi \ge 1$, for every~$j$ in~$\N_0$ we have
     \begin{multline}
       \label{eq:t10}
       \tJ_j^+(\tau,\lambda(s))
       + \tI_{j + 1}^+(\tau,\lambda(s))
       \\
       \begin{aligned}
         & \le
         2^{-\tau(q(j+1)^2 + \Xi \cdot (j+1)) + \tau \xi \cdot (2j+3)}
         \sum_{k \in J_j \cup I_{j + 1}} k \cdot 2^{- \lambda(s)k}
         \\ & \le 
         (2^{\tau \xi + 3}) 2^{2q(s + 1)^3 - q \tau \cdot (j+1)^2 - \tau \cdot (j + 1)}.
       \end{aligned}
     \end{multline}
     Taking~$s = \tau$, for every~$j \ge 2 \tau + 1$ we have
     $$ \tJ_j^+(\tau,\lambda(\tau))
     + \tI_{j + 1}^+(\tau,\lambda(\tau))
     \le
     \left( 2^{\tau \xi + 3} \right) 2^{- \tau \cdot (j + 1)}. $$
     This implies that~$\tPi^+(\tau, \lambda(\tau))$ is finite, as wanted.

     \partn{2}
     The first inequality follows directly from the definitions.
     To prove the second inequality, note that by~\eqref{e:B} and~\eqref{eq:N I}, and our hypotheses~$\tau \ge 50$ and~$\xi > 0$, we have
     \begin{equation}
       \label{eq:first t block}
       1 + \tI_0^+(\tau, \lambda(s))
       \le
       1 + \sum_{k = 1}^{+\infty} k \cdot 2^{- \tau k + \tau \xi}
       =
       1 + 2^{\tau \xi} \frac{2^{- \tau}}{(1 - 2^{- \tau})^2}
       \le
       2^{\tau \xi}.
     \end{equation}
     On the other hand, by item (c) of Lemma~\ref{l:the itinerary}, \eqref{e:B}, \eqref{eq:N I}, and the inequality~$\Xi - 2\xi \ge 1$, for every integer $j\ge 1$ we have
     \begin{equation*}
       \begin{split}
         \tI_{j}^+(\tau, \lambda(s))
         & \le
         \left( 2^{\tau \xi} \right) 2^{- \tau (q j^2 + (\Xi - 2 \xi)j)}
         \sum_{m = 1}^{|I_{j}|} \left( 2^{qj^3} + m \right) 2^{- \tau m}
         \\ & \le
         \left( 2^{\tau \xi + 1} \right)
         2^{q j^3 - q \tau j^2 - (\Xi - 2 \xi) \tau j} \frac{1}{1 - 2^{-\tau}}
         \\ & \le
         \left( 2^{\tau \xi + 2} \right)
         2^{q (j - \tau) j^2 - \tau j}.
       \end{split}
     \end{equation*}
     Combined with~\eqref{eq:first t block}, the inequality~$\Xi \ge 2 \xi$, and our hypotheses~$q \ge 50 (\Xi + 1)$ and~$\tau \ge 50$, this implies
     \begin{equation*}
       \label{eq:total tI}
       1 + \sum_{j = 0}^{\tau_0 + 1}  \tI_{j}^+(\tau, \lambda(s))
       \le
       \left( 2^{\tau \xi + 2 }\right) \frac{2^{2q (\tau + 2)^2}}{1 - 2^{-\tau}}
       \le
       2^{2q \tau (\tau + 6)}.
     \end{equation*}
     Together with item~2 of Sublemma~\ref{l:zero-temperature core} and our hypothesis~$\tau \ge 50$, this chain of inequalities implies
     \begin{equation}
       \label{eq:first t I}
       1 + \sum_{j = 0}^{\tau_0 + 1}  \tI_{j}^+(\tau, \lambda(s))
       \le
       \frac{1}{2^3} 2^{- q \tau^2} \cdot \hJ_{s_0}^-(\tau, \lambda(s)).
     \end{equation}

     On the other hand, by~\eqref{e:B}, \eqref{eq:N J}, and our hypothesis~$\tau \ge 50$, for every~$j$ in~$\{0, \ldots, \tau_0 - 4 \}$ we have
     \begin{equation*}
       \begin{split}
         \tJ_j^+(\tau,\lambda(s))
         & =
         2^{-\tau(q(j+1)^2 + \Xi \cdot (j+1)) + 2 \tau \xi \cdot (j+1)}
         \sum_{k \in J_j} k \cdot 2^{- \lambda(s)k}
         \\ & \le 
         |J_j|2^{q(j + 1)^3 - q \tau \cdot (j+1)^2 - (\Xi - 2 \xi) \tau \cdot (j + 1)}
         \\ & \le 
         2^{2q(j + 1)^3 - q \tau \cdot (j+1)^2 - (\Xi - 2 \xi) \tau \cdot (j + 1)}
         \\ & \le
         2^{q(j + 1)^2 (2j + 2 - \tau)}
         \\ & \le
         2^{q(\tau - 2)^2 (\tau - 4)}
         \\ & \le
         2^{q\tau^2(\tau - 7)}.
       \end{split}
     \end{equation*}
     Together with item~2 of Sublemma~\ref{l:zero-temperature core} and with our hypothesis~$\tau \ge 50$, this implies
     \begin{equation}
       \label{eq:first t Js}
       \frac{\sum_{j = 0}^{\tau_0 - 4} \tJ_j^+(\tau, \lambda(s))}{\hJ_{s_0}^-(\tau, \lambda(s))}
       \le
       \tau 2^{- 3 q \tau^2}
       \le
       \frac{1}{2^2} 2^{- q \tau^2}.
     \end{equation}

     On the other hand, by~\eqref{eq:t10}, item~1 of Sublemma~\ref{l:zero-temperature core}, the inequality $\Xi\ge 2\xi$ and our hypothesis~$q \ge 50 (\Xi + 1)$, for every integer~$j \ge s_0 + 1$ we have
     \begin{equation*}
       \begin{split}
         \frac{\tJ_j^+(\tau, \lambda(s)) + \tI_{j + 1}^+(\tau, \lambda(s))}{\hJ_{s_0}^-(\tau, \lambda(s))}
         & \le
         \left( 2^{\tau \xi + 3} \right) 2^{-q \tau \cdot ((j + 1)^2 - (s_0 + 1)(s_0 + 2))}
         \\ & \le
         \left( 2^{\tau \xi + 3} \right) 2^{-q \tau \cdot (s_0 + 2)(j - s_0)}
         \\ & \le
         \left( 2^{\tau \xi + 3} \right) 2^{-q \tau^2(j - s_0) - q \tau}
         \\ & \le
         \frac{1}{2^3} 2^{-q \tau^2(j - s_0)}.
       \end{split}
     \end{equation*}
     Summing over~$j \ge s_0 + 1$ and using our hypotheses~$q \ge 50 (\Xi + 1)$ and~$\tau \ge 50$, we obtain
     \begin{equation*}
       \frac{\sum_{j = s_0 + 1}^{+\infty} \left( \tJ_j^+(\tau, \lambda(s)) + \tI_{j + 1}^+(\tau, \lambda(s)) \right)}{\hJ_{s_0}^-(\tau, \lambda(s))}
       \le
       \frac{1}{2^3} \frac{2^{-q \tau^2}}{1 - 2^{-q \tau^2}}
       \le
       \frac{1}{2^2} 2^{- q \tau^2}.
     \end{equation*}
     Combined with~\eqref{eq:first t I} and~\eqref{eq:first t Js}, this implies
     \begin{equation}
       \label{eq:13}
       \tPi^+(\tau, \lambda(s)) - \sum_{\varsigma = \tau_0 - 3}^{s_0} \tJ_{\varsigma}^+(\tau, \lambda(s))
       \le
       \frac{1}{2} 2^{- q\tau^2} \cdot \hJ_{s_0}^-(\tau, \lambda(s)).
     \end{equation}

     For each integer~$\varsigma$ in~$[\tau_0 - 3, s_0]$, we have
     \begin{multline}
       \label{eq:14}
       \tJ_{\varsigma}^+(\tau, \lambda(s)) - \hJ_{\varsigma}^+(\tau, \lambda(s))
       \\
       \begin{aligned}
         & \le
         \sum_{k = b_\varsigma}^{b_\varsigma + \varsigma^2 - 1} k \cdot 2^{- \lambda k -\tau N(k) + \tau \xi B(k)}
         +
         \sum_{k = b_\varsigma + \varsigma^2}^{a_{\varsigma + 1} - 1} (b_\varsigma + \varsigma^2) 2^{- \lambda k -\tau N(k) + \tau \xi B(k)}
         \\ & \le
         (b_\varsigma + \varsigma^2) J_{\varsigma}^+(\tau, \lambda(s)).
       \end{aligned}
     \end{multline}
     Together with item~3 of Sublemma~\ref{l:zero-temperature core}, this implies that for~$\varsigma$ in~$[\tau_0 - 3, s_0 - 1]$ we have
     \begin{equation}
       \label{eq:15}
       \tJ_{\varsigma}^+(\tau, \lambda(s)) - \hJ_{\varsigma}^+(\tau, \lambda(s))
       \le
       \frac{1}{20} 2^{- q \tau^2} \cdot \hJ_{\varsigma}^-(\tau, \lambda(s)).
     \end{equation}
     On the other hand, \eqref{eq:14} with~$\varsigma = s_0$ and item~4 of Sublemma~\ref{l:zero-temperature core} imply
     $$ \tJ_{\varsigma}^+(\tau, \lambda(s)) - \hJ_{\varsigma}^+(\tau, \lambda(s)) - (|J_{s_0 - 1}| - s_0^2) J_{s_0}^+(\tau, \lambda(s))
     \le
     \frac{1}{4} 2^{- q \tau^2} \cdot \hJ_{s_0}^-(\tau, \lambda(s)). $$
     Together with~\eqref{eq:13} and~\eqref{eq:15}, this implies the desired inequality and completes the proof of the lemma.
   \end{proof}

   \bibliographystyle{alpha}

\end{document}